\documentclass[12pt,a4paper]{amsart}
\usepackage{graphicx}
\usepackage{amsmath}
\usepackage{color}
\usepackage{pgf,tikz,pgfplots}
\usetikzlibrary{arrows}
\usepackage{mathrsfs}
\usetikzlibrary[patterns]

\theoremstyle{plain}
\newtheorem{thm}{Theorem}[section]

\newtheorem{lemma}[thm]{Lemma}
\newtheorem{ques}[thm]{Question}
\newtheorem{corollary}[thm]{Corollary}
\newtheorem{prop}[thm]{Proposition}
\newtoks\prt
\newtheorem{proclaim}[thm]{\the\prt}
\theoremstyle{definition}

\newtheorem{definition}[thm]{Definition}

\def\eqn#1$$#2$${\begin{equation}\label#1#2\end{equation}}

\numberwithin{equation}{section}

\def\C{\mathcal{C}}

\def\T{\mathcal T}

\def\diam{\operatorname{diam}}
\def\dist{\operatorname{dist}}

\def\epsilon{\varepsilon}
\def\en{\mathbb N}
\def\er{\mathbb R}

\def\G{\mathcal{G}}

\def\id{\operatorname{id}}

\def\K{\mathcal{K}}
\def\loc{\operatorname{loc}}

\def\mE{\mathcal{E}}
\def\mira{\mathcal{L}^4}
\def\mir4{\mathcal{L}^4}

\def\Q{\tilde{Q}}
\def\R{\mathcal{R}}
\def\r2{\er^2}
\def\sgn{\operatorname{sgn}}

\def\rn{\mathbb R^n}

\def\sgn{\operatorname{sgn}}
\def\u{\tilde{u}}
\def\w{\boldsymbol w}

\def\z{\tilde{z}}

\makeatletter
\newcommand{\labeltext}[2]{%
	\@bsphack
	\def\@currentlabel{#1}{\label{#2}}%
	\@esphack
}
\makeatother

\def\step#1#2#3{\par \noindent{{\vskip 5pt \bf Step~\labeltext{#1}{#3}#1. }{\bf #2. }}}

\def\ve{\boldsymbol v}
\def\u{\boldsymbol u}
\def\V{\mathbb V}

\newtoks\by
\newtoks\paper
\newtoks\book
\newtoks\jour
\newtoks\yr
\newtoks\pages
\newtoks\vol
\newtoks\publ
\def\ota{{\hbox\vol{???}}}
\def\cLear{\by=\ota\paper=\ota\book=\ota\jour=\ota\yr=\ota
\pages=\ota\vol=\ota\publ=\ota}
\def\endpaper{\the\by, {\the\paper},
\textit{\the\jour} \textbf{\the\vol} (\the\yr), \the\pages.\cLear}
\def\endbook{\the\by, \textit{\the\book}, \the\publ.\cLear}
\def\endprep{\the\by, \textit{\the\paper}, \the\jour.\cLear}
\def\endyearprep{\the\by, \textit{\the\paper}, \the\jour, (\the\yr).\cLear}
\def\name#1#2{#2 #1}

\definecolor{qqzzqq}{rgb}{0.,0.6,0.}
\definecolor{qqwuqq}{rgb}{0.,0.39215686274509803,0.}
\definecolor{qqffqq}{rgb}{0.,1.,0.}
\definecolor{ffqqqq}{rgb}{1.,0.,0.}
\definecolor{qqqqff}{rgb}{0.,0.,1.}
\definecolor{ubqqys}{rgb}{0.29411764705882354,0.,0.5098039215686274}
\definecolor{ffffff}{rgb}{1.,1.,1.}

\title[A sense preserving Sobolev homeomorphism with negative Jacobian]
{A sense preserving Sobolev homeomorphism with negative Jacobian almost everywhere}
\author{Daniel Campbell, Luigi D'Onofrio and Stanislav Hencl }
\thanks{The first author was supported by the grant GA\v{C}R 20-19018Y. Part of the research was conducted at Universit\`a Parthenope under the Young Investigator Training Program 2018 supported by ACRI (Associazione di Fondazioni e di Casse di Risparmio Spa) and Università degli Studi di Urbino Carlo Bo. The third author was supported by the grant GA\v{C}R P201/18-07996S}
\address{Department of Mathematics, University of Hradec Kr\'alov\'e, Rokitansk\'eho 62, 500 03 Hradec Kr\'alov\'e, Czech Republic}
\email{campbell@uhk.cz}
\address{Dipartimento di Scienze e Tecnologie \\ Universit\`a \lq\lq Parthenope \rq\rq
Centro Direzionale Isola C4, 80100 Napoli, Italy}
\email{donofrio@uniparthenope.it}
\address{Department of Mathematical Analysis, Charles University, So\-ko\-lovsk\'a 83, 186~00 Prague 8, Czech Republic}
\email{\tt hencl@karlin.mff.cuni.cz}
\begin{document}

\begin{abstract}
For every $1\leq p<\frac{3}{2}$ we construct a Sobolev homeomorphism $f\in W^{1,p}([-1,1]^4,[-1,1]^4)$ such that $f(x)=x$ for every $x\in \partial[-1,1]^4$ but $J_f<0$ a.e.
\end{abstract}

\maketitle

\section{Introduction}


In 2001 Haj{\l}asz posed a series of questions about the Jacobians of homeomorphisms which have some kind of derivative (weak or approximative).  These questions appeared in several lecture notes and were recently reprinted in \cite{GH2}. The essence of the questions can be summarized as follows:
\begin{ques}\label{quesie}
	Let $Q\subset \rn$ be the open unit cube.
	\begin{enumerate}
		\item[a)] Does there exist an $f\in W^{1,p}(Q,\rn)$ with $J_f = \det Df$ positive on a set of positive measure in $Q$ and negative on a set of positive measure in $Q$?
		\item[b)] Does there exist an $f\in W^{1,p}(Q,\rn)$ with $J_f = \det Df$ negative almost everywhere on $Q$ but $f =\id$ on $\partial Q$?
	\end{enumerate}
\end{ques}
In fact this formulation is somewhat stronger than the original questions posed by Haj{\l}asz since every $f\in W^{1,p}$ has an approximatively differentiable representative but the opposite is far from true. Despite the fact that the questions were well publicized and important to a range of areas, they remained widely open for a long time. Given a homeomorphism $f$ (strongly) differentiable at a point $x$ with $J_f(x)\neq 0$, the sign of the Jacobian at $x$ determines the degree of $f$. We call a homeomorphism with topological degree 1 sense-preserving and a homeomorphism with degree -1 sense-reversing. A naive intuition would suggest that the same might hold when we replace with the weak or approximative derivative, if not at one point then at least almost everywhere. Surprisingly the answer turns out to be more complex than one might expect and the notion that Sobolev homeomorphisms behave like diffeomorphisms turns out to be false. Before we expound the results let us first explore a little more the relevance of the question.

The study of the Jacobian of a homeomorphism is a natural question in the context of the study of Nonlinear Elasticity, Quasiconformal mappings, Mappings of finite distortion. In models of Nonlinear Elasticity (see for instance the pioneering work by Ball \cite{B} or the monograph of Ciarlet \cite{Ci}), one is led to study existence and properties of minimizers of energy functionals of the form
\begin{equation}\label{energyI}
I(f)=\int_{\Omega} W(Df)\,dx\,,
\end{equation}
where $f:\rn\supseteq\Omega\to\Delta\subseteq\rn$ models the deformation of a homogeneous elastic material with respect to a reference configuration $\Omega$ and prescribed boundary values, while $W:\er^{n\times n}\to\er$ is the stored-energy functional. In order for the model to be physically relevant, as pointed out by Ball in \cite{B}, \cite{B2}, one has to require that $u$ is a homeomorphism or at least one-to-one a.e.- this corresponds to the non-impenetrability of the material - and that
\begin{align}\label{E_WA}
W(A)\to+\infty\quad\text{as $\det A\to 0$}\,, && W(A)=+\infty\quad\text{if $\det A\leq0$}\,.
\end{align}
The first condition in \eqref{E_WA} prevents too high compressions of the elastic body, while the latter guarantees that the orientation is preserved (at least in the analytical sense).
In particular, if $u$ is an admissible deformation with finite energy, then one has that
\begin{equation}\label{E_det>0}
J_f:=\det Df>0\quad\text{ a.e. in $\Omega$}
\end{equation}
and hence we restrict our attention to mappings which do not change orientation.

A key question of interest is to prove the regularity of the solutions of this problem. A regular solution - a diffeomorphism - must of necessity have Jacobian that does not change sign. Critically the Jacobian of any topologically sense-preserving diffeomorphism must be positive. On the other hand working only with diffeomorphisms would be too restrictive to the tools of the Calculus of Variations. Thus one is naturally led to Sobolev homeomorphisms and to questions about their Jacobian. Specifically we want to marry the concept of an almost everywhere positive $\det Df$ with the topological concept of sense-preservation (i.e. that matter is not inverted). This is the essence of Question~\ref{quesie}

The first significant answer to the questions of Haj{\l}asz was published in \cite{HM}. Precisely

\prt{Theorem}
	\begin{proclaim}\label{henclmaly}
	Let  $n=2,3$ and $p\geq 1$ or $n\geq 4$ and $p>[n/2]$.
 Let $\Omega\subset \mathbb R^n$ be a domain and $f\in W^{1,p}_{loc}(\Omega, \mathbb R^n)$ be homeomorphism.
 Then either $J_f \geq 0$ a.e. in $\Omega$ or $J_f\leq 0$ a.e. in $\Omega$. 
	\end{proclaim}
Here $[x]$ stands for the greatest integer less or equal to $x$. This result has since been pushed to the limiting case $p=[n/2]$ under additional assumptions (including on the inverse), see  \cite{GH1}. On the other hand the surprising construction of \cite{HV}, showed that if $n\geq 4$ then there exists a homeomorphism in $W^{1,1}$ whose Jacobian is negative on a set of positive measure and positive on a set of positive measure. Later this construction was improved in \cite{CHT} giving  the following result:
\begin{thm}\label{cahenten}
 Let $\Omega\subset \mathbb R^n$, $n\geq 4$ and $1\leq p < [n/2]$, then there is a homeomorphism $f \in W^{1,p}((-1,1)^n, \mathbb R^n)$ such that $J_f>0$ on a set of positive measure and $J_f<0$  on a set of positive measure. Moreover $f$ satisfies the Lusin $(N)$ condition.
\end{thm}
The combination of Theorem~\ref{henclmaly} and Theorem~\ref{cahenten} answers, up to the critical case $p=[n/2]$, Question~\ref{quesie} $a)$. Let us note that constructions of almost everywhere approximately differentiable homeomorphisms with everywhere negative (approximate) Jacobian are to be found in Goldstein and Haj{\l}asz \cite{GH} and \cite{GH2}. These maps lack the Sobolev regularity but have other striking properties e.g. measure preservation or H\"older continuity of the map and its inverse.

The construction in \cite{CHT} opens the question of point $b)$ from Question~\ref{quesie}, can the construction be pushed to give negative Jacobian almost everywhere? It turns out that this question is of even greater relevance to the problem of non-linear elasticity than the previous (point $a)$), especially in connection with the Ball-Evan's approximation problem. This problem can be simply formulated as ``Is it possible to approximate a Sobolev homeomorphism with diffeomorphisms in the Sobolev space?''. The motivation of this question comes from wanting to understand the regularity of the minimizers. We point out that finding diffeomorphisms near a given homeomorphism is not an easy task, as the usual approximation techniques like mollification or Lipschitz extension using the maximal operator, destroy, in general, the injectivity. An overview of the history of diffeomorphic approximation can be found, for example, in \cite{HP}. As yet there are positive planar results (see for example \cite{IKO, HP, Ca, PR}) but in higher dimensions only a negative result derived in \cite{CHT}: Assume for contradiction that $f$ of Theorem \ref{cahenten} can be approximated by diffeomorphisms (or piecewise affine homeomorphisms) $\{f_k\}_{k=1}^{\infty}$ then the pointwise limit of a subsequence (which we denote in the same way) satisfies:
$$
Df_k(x) \rightarrow Df(x)\,\,\,\textrm{and}\,\,\, J_{f_k}(x) \rightarrow J_f(x)\,\,a.e. \, x\in (-1,1)^n
$$
As $f_k$ are locally Lipschitz we know that $J_{f_k}\geq 0$ a.e. in $(-1,1)^n$ or $J_{f_k}\leq 0$ a.e. in $(-1,1)^n$. So the pointwise limit of nonnegative (or nonpositive) $J_{f_k}$ cannot change sign which gives a contradiction. 


On the other hand, as noted by Buttazzo in 2016 in Naples, such a homeomorphism by \eqref{E_WA} would of necessity have infinite energy and so not be of great relevance to the minimization process. Conversely an answer to Question~\ref{quesie}~$b)$, would supply a homeomorphism $f$, $f=\id$ on $\partial Q$ with $J_f<0$ a.e. Therefore $\tilde{f}(x_1, x_2,\dots x_n) = f(-x_1, x_2,\dots,x_n)$ would have $J_f >0$ a.e.. On the other hand $\deg(g,Q) =-1$ and any approximating diffeomorphism would have to have infinite energy. We supply precisely such a mapping for $n=4$

	\prt{Theorem}
	\begin{proclaim}\label{main}
		For every $1\leq p <\frac{3}{2} $ there exists a Sobolev homeomorphism $f\in W^{1,p}([-1,1]^4,[-1,1]^4)$ such that $f(x)=x$ for every $x\in \partial [-1,1]^4$ but $J_f(x)<0$ for a.e. $x\in (-1,1)^4$. Further $f$ satisfies the Lusin $(N)$ and $(N^{-1})$ conditions.
	\end{proclaim}
	(For the definition of the $(N)$ and $(N^{-1})$ condition see Definition~\ref{Losing}). This result yields the following;
	\prt{Corollary}
	\begin{proclaim}
		Set $\tilde{f}(x_1,x_2,x_3,x_4)=f(-x_1,x_2,x_3,x_4)$ where $f$ is from Theorem \ref{main}.
		Then $J_{\tilde{f}}(x)>0$ a.e. but there are no diffeomorphisms (or piecewise affine homeomorphisms) $f_k$ such that
		$f_k\to \tilde{f}$ in $W^{1,p}$ for $1\leq p < \frac 3 2.$
	\end{proclaim}

\subsection{A brief overview of the construction and the main ideas}

 Just for the convenience of the reader we recall some ingredients of the construction of Theorem \ref{cahenten} in \cite{CHT}. Firstly they fix a Cantor type set $\mathcal C_A \subset (-1,1)$ of positive measure and they set
$$
\begin{array}{ll}
  \mathcal K_A =& (\mathcal C_A \times \mathcal C_A \times \mathcal C_A  \times [-1,1])\cup (\mathcal C_A \times \mathcal C_A  \times [-1,1] \times \mathcal C_A ) \\
   & (\mathcal C_A \times [-1,1]  \times \mathcal C_A \times \mathcal C_A )\cup ( [-1,1]\times \mathcal C_A \times \mathcal C_A  \times \C_A)
\end{array}
$$
They also fix a Cantor type set $\mathcal C_B \subset (-1,1)$ of zero measure and consider the corresponding $\mathcal K_B$. The first mapping $S_q: \mathbb R^n \longrightarrow \mathbb R^n$ squeezes $\mathcal K_A$ onto $\mathcal K_B$ homeomorphically in a natural way. The key ingredient is the construction of a bi-Lipschitz sense-preserving homeomorphism $F$ such that
\eqn{uuu}
$$
F(x_1,x_2,x_3,x_4)=(x_1,x_2,x_3,-x_4) \textrm{ for every } x\in \mathcal K_B.
$$
At last they find a mapping $S_t: \mathbb R^n \longrightarrow \mathbb R^n $ which stretches  $\mathcal C_B\times \mathcal C_B\times \mathcal C_B \times \mathcal C_B$ back to  $\mathcal C_A\times \mathcal C_A\times \mathcal C_A \times \mathcal C_A$ such that lines in $\mathcal K_B$ are not prolonged too much. Since we control the behaviour of 
$f=S_t\circ F\circ S_q$ on lines parallel to coordinate axes it is possible to check that 
$f$ satisfies the ACL property and by delicate computations that even $f\in W^{1,p}$.  
By \eqref{uuu} and $S_t=S_q^{-1}$ on $\mathcal C_B\times \mathcal C_B\times \mathcal C_B \times \mathcal C_B$ we obtain that $f$ behaves like $[x_1,x_2,x_3,-x_4]$ on $\mathcal C_A\times \mathcal C_A\times \mathcal C_A \times \mathcal C_A$ and hence $J_f<0$ on this set of positive measure. One of the key properties of $F$ is that line segments parallel to coordinate axes and close to $\mathcal H_B$ but far away from $\mathcal C_B\times \mathcal C_B\times \mathcal C_B \times \mathcal C_B$ are mapped to segments parallel to coordinate axes close to $\mathcal H_B$ which allows the estimates of the derivatives of $f$.  

The first idea to construct our mapping is to iterate the procedure done by \cite{CHT}, i.e. we construct $f_1$ with $J_{f_1} <0$ on a closed set of positive measure $\mE_1$ and $f_1=id$ on $\partial Q$. We cover $Q\setminus \mE_1$ by small cubes such that $f_1$ is close to a linear mapping on each of these cubes. We define 
$f_2=f_1$ on $\mE_1$ and $f_2$ is a composition of $f_1$ and scaled and translated copy of $f_1$ on cubes covering the rest of $Q$. Then $J_{f_2}<0$ on a new set $\mE_2\subset Q\setminus \mE_1$, we cover $Q\setminus(\mE_1\cup \mE_2)$ with small cubes and continue. By induction we construct a sequence of maps $\{f_m\}_{m=2}^{\infty}$ and we want them to converge in $W^{1,p}$ to some $f$ which satisfy our thesis. 
The derivative of $f$ is a derivative of composition but the huge problem with this approach is that the integral of $|Df|^p$ is too big and we cannot control it.  

Hence we need to substantially modify the construction of $f_1$. Our aim is to construct a map $f_1$ such that the integral of $|D_jf_1|^p$, $j=1,2,3$, is much smaller than $|D_4 f_1|^p$ in average (see $v)$ in Theorem \ref{TheBigLebowski} below). 
Then we need a mapping $f_2$ that is a clever rotation of $f_1$ in the well-chosen small cubes covering $Q\setminus \mE_1$. The main aim of this rotation is that the big derivative from $f_1$ (i.e. $D_4 f_1$) doesn't multiply with the big derivative of the next mapping $f_2$ in the matrix multiplication. The cubes are rotated to send $x_1$ in the direction of $D_4 f_1$, and the derivative of $f_2$ in direction $x_1$ is small. Then the integral of the derivative of the composition is small and we can make everything work.

As in \cite{CHT} we construct our Cantor sets. 
We choose a big parameter $K>0$ so that the Cantor set $C_{A,K}\subset[-1,1]$ has measure almost $2$. 
We denote by 
$\mathcal C=\mathcal C_{A,K}\times \mathcal C_{A,K}\times \mathcal C_{A,K} \times \mathcal C_{A,K}$ the corresponding Cantor set in dimension $4$. For $m\in\en$ we refer to the set
$$
\mathcal{C}_{A,K,m}:=\bigcup_{n_1,n_2,n_3=-m}^m (\mathcal C+[2n_1,2n_2,2n_3,0])
$$
as a $4$-dimensional \lq $m$-Cantor plate\rq so that in each of $(2m+1)^3$ cubes $Q$ we have a copy of $\mathcal C$. 
This choice of $m$-Cantor plate allows us to make the integral outside of Cantor set of $|D_j f_1|^p$, $j=1,2,3$, rather small. Since $f_1(x)=x$ on 
$\partial ([-2m-1,2m+1]^3\times [-1,1])$ the derivative close to the boundary is big but the integral of the derivative between the copies of $\mathcal C$ inside (in $x_1,\ x_2,\ x_3$ directions) is quite small as the measure of the complement of $\mathcal C$ is small. 
In our proof we fix $K$ so that the measure of complement of $C_{A,K}$ is small and we choose $m$ so big that in average only the derivative between the copies of $\mathcal C$ is important and thus the integral of $|D_j f_1|^p$ is small. Analogously we define a Cantor $m$-plate of zero measure $\mathcal{C}_{B,m}$. 

In Section 3 we construct a special homeomorphism $\tilde{G}_{K,N,m,\eta}$ that squeezes the Cantor set, i.e. it maps $\mathcal C_{A,K,m}$  onto $\mathcal C_{B,m}$. We cannot use an analogy of mapping $S_q$ from \cite{CHT} because for our iteration procedure we need a mapping which is locally bilipschitz outside of $\mathcal C_{A,K,m}$. 
In Section 3 we also construct a specific homeomorphism $G_{K,N,m,\eta}$ that stretches $\mathcal C_{B,m}$ onto $\mathcal C_{A,K,m}$. We follow the construction in \cite[Lemma 3.2]{CHT} with a difference that our set $\mathcal C_{A,K}$ is bigger and occupies almost all of $[-1,1]$.
The key middle map $F_{\beta}$ that satisfies the crucial property \eqref{uuu} on 
$\mathcal{C}_{A,K,m}$ is constructed in Section 4.


\section{Preliminaries}
A point $x \in \mathbb R^n$ in coordinates is denoted as $(x_1,\ldots,x_n)$. We denote by $|x|:=\sqrt{\sum_{i=1}^{n} x^2_i}$ the Euclidean norm of a point $x\in \mathbb R^n$, and $||x||:=\sup_{i}|x_i|$ denotes the supremum norm of $x$. For $c\in\rn$ and $r>0$ we denote the cube as $Q(c,r)=(c_1-r,c_1+r)\times\hdots\times (c_n-r,c_n+r)$. 

Let $\Omega\subset \mathbb R^n$ be an open set. We say that $f:\Omega\rightarrow \mathbb R^n$ belongs to the Sobolev space $W^{1,p}(\Omega, \mathbb R^n)$, $1\leq p < \infty$, if $f$ is $p$-integrable and if the coordinate functions of $f$ have $p$-integrable distributional derivatives. We say that $f$ belongs to the space $W^{1,p}_{loc}(\Omega, \mathbb R^n)$ if $f\in W^{1,p}(\Omega', \mathbb R^n)$ for every subdomain $\Omega'\subset \Omega$. \par

By $D_if$ we denote the derivative of mapping $f$ with respect to first coordinate $x_i$, i.e. the matrix $Df$ consists of lines $D_1f,\ D_2f,\ D_3 f$ and $D_4 f$.

\subsection{ACL condition}
It is a well-known fact (see e.g. \cite[Section 3.11]{A}) that a mapping
$u\in L^1_{\loc}(\Omega,\er^m)$ is in  $W^{1,1}_{\loc}(\Omega,\er^m)$
if and only if there is a representative which is an absolutely continuous function
on almost all lines parallel to coordinate axes and the derivative on these lines is integrable.
More precisely, let $i\in\{1,2,\hdots,n\}$
and denote by $\pi_i$ the projection on to the hyperplane perpendicular to the $x_i$-axis.
Suppose that $Q(c,r):=(c_1-r,c_1+r)\times\hdots\times(c_n-r,c_n+r)\subset\Omega$ for some
$c\in\rn$, $r>0$ and set $Q^i(c,r)=\pi_i(Q(c,r))$.
Let  $y\in Q^i(c,r)$ and denote
$$u_{i,y}(t)=u(y_1,\hdots,y_{i-1},t,y_{i+1},\hdots,y_n)\quad\text{ for }\quad t\in(c_i-r,c_i+r).$$

\prt{Theorem}
\begin{proclaim}\label{ACL}
Let $\Omega\subset\rn$ be open and let $u\in L^1_{\loc}(\Omega,\er^m)$.
Then $u\in W^{1,1}_{\loc}(\Omega,\er^m)$ if and only if the following happens. For every cube $Q(c,r)\subset\subset \Omega$ and for every $i\in\{1,\hdots,n\}$
there is a representative $\u$ of $u$ such that the function $\u_{i,y}(t)$ is absolutely continuous on $(c_i-r,c_i+r)$ (i.e. each coordinate function is absolutely continuous) for $\mathcal{L}^{n-1}$ almost every $y\in Q^i(c,r)$ and moreover
$$
\int_{Q^i(c,r)}\int_{c_i-r}^{c_i+r}|\nabla\u_{i,y}(t)|\; dt\; dy<\infty .
$$
\end{proclaim}

\subsection{Topological degree}
Given a smooth map $f$ from $\Omega\subset\rn$ into $\rn$ we can define the topological degree as
$$\deg(f,\Omega,y_0)=\sum_{\{x\in\Omega: f(x)=y_0\}} \sgn(J_f(x))$$
if $J_f(x)\neq 0$ for each $x\in f^{-1}(y_0)$.
This definition can be extended to arbitrary continuous mappings and each point, see e.g.\ \cite{FG}.

A continuous mapping
$f:\Omega\to\rn$ is called \emph{sense-preserving} if
$$\deg(f,\Omega',y_0)>0$$ for all domains $\Omega'\subset\subset\Omega$ and
all $y_0\in f(\Omega')\setminus f(\partial\Omega')$. Similarly we call $f$ \emph{sense-reversing} if
$\deg(f,\Omega',y_0)<0$ for all $\Omega'$ and $y_0$.
Let us recall that each homeomorphism on a domain is either sense-preserving or sense-reversing see
\cite[II.2.4., Theorem 3]{RR}.

\subsection{Composition and integration}
	
	\begin{definition}\label{Losing}
		Let $f:\rn\supset G\to\rn$, we say that $f$ satisfies the Lusin $(N)$ condition on $G$ if $\mathcal{L}^n(f(E)) = 0$ for every $E\subset G$ such that $\mathcal{L}^n(E) = 0$. We say that $f$ satisfies the Lusin $(N^{-1})$ condition on $G$ if $\mathcal{L}^n(f^{-1}(E)\cap G) = 0$ for every $E\subset \rn$ such that $\mathcal{L}^n(E) = 0$.
	\end{definition}
	Obviously a mapping locally bi-Lipschitz of $G$ satisfies both of these conditions on $G$.
	
	For the following Theorem see \cite[Theorem 3.16 and Corollary 3.19]{A}:
	\begin{thm}\label{AFPtheo}
		Let $\Omega, \Delta \subset \rn$ be open. Let $u\in W^{1,1}_{\loc}(\Omega,\er^d)$ and suppose that $F:\Delta \to \Omega = F(\Delta)$ is a bi-Lipschitz homeomorphism then $u\circ F\in W^{1,1}_{\loc}(\Delta,\er^d)$ and
			$$
				Du\circ F(x) = Du(F(x))DF(x) \quad \text{ for almost all } x \in \Delta.
			$$
	\end{thm}
	\begin{lemma}\label{stupid}
		Let $p\in [1,\infty)$ and let $a,b \geq 0$ then
		$$
			|a^p - b^p| \leq p (a+b)^{p-1}|a-b|\leq p 2^{p-1}(a^{p-1}+b^{p-1})|a-b|.
		$$
	\end{lemma}
	\begin{proof}
		We have
		$$
			|a^p - b^p| = \int_{\min\{a,b\}}^{\max\{a,b\}}pt^{p-1} \leq p (a+b)^{p-1}|a-b|.
		$$
	\end{proof}
	
Let $f:Q(0,1)\to\er$. 	
For each $c\in \rn$ and $r>0$ we denote the translated and scaled function as
$$
f_{c,r}(x) = r f(\tfrac{x-c}{r})+c\ .
$$	
A simple linear change of variables shows that for $f \in W^{1,p}(Q(0,1))$ we have 
		\begin{equation}\label{2207}
		f_{c,r} \in W^{1,p}(Q(c,r))\text{ and } \int_{Q(c,r)} |Df_ {c,r}|^p \leq C_{\ref{2207}} r^n\int_{Q(0,1)} |Df|^p.
		\end{equation}
	
	The following lemma is the key for the estimate of the derivative of the composition.
	At the point $a=F^{-1}(c)$ of differentiability of $F$ we know that $F$ is really close to $F(a)+DF(a)(x-a)$ and we can estimate the derivative of $f_{c,r}\circ F$ by the linearization (see \eqref{TheSpiritIsMoving}). This form is crucial for us as later we will perform a linear change of variables on the right-hand side of \eqref{TheSpiritIsMoving} and its Jacobian is a constant and can be put out of the integral.
	
	\begin{lemma}\label{Patchwork}
		 Assume that $f \in W^{1,p}(Q(0,1))$. 
		Let $G\subset \er^n$ be an open set and let $F:G\to \rn$ be a mapping locally bi-Lipschitz on $G$. Let $Q(c,r)\subset\subset F(G)$ and denote $a=F^{-1}(c)$.
		Then for every $\rho > 0$ and
  almost every $c \in G$ there exists an $r_c$ such that for all $0<r<r_c$ we have
		\begin{equation}\label{TheSpiritIsMoving}
			\begin{aligned}
				&\int_{F^{-1}(Q(c,r))} |Df_{c,r}(F(x))DF(x)|^p \; dx \\
				&\qquad \leq  \int_{[DF(a)]^{-1}(Q(0,r))} \bigl|Df_{c,r}(c + DF(a)(x-a))DF(a)\bigr|^p\; dx+\rho r^n.
			\end{aligned}
		\end{equation}
		Similarly, for any measurable $X\subset Q(0,1)$ and $X_{c,r} = rX+c$
		\begin{equation}\label{WeAllLive}
			\begin{aligned}
			&\int_{F^{-1}(X_{c,r})} |Df_{c,r}(F(x))DF(x)|^p\; dx  \\
			&\qquad \leq \int_{[DF(a)]^{-1}(rX)} \bigl|Df_{c,r}(c + DF(a)(x-a))DF(a)\bigr|^p\; dx+\rho r^n.
			\end{aligned}
		\end{equation}
	\end{lemma}
	
	\begin{proof}
		We start with rough sketch of the proof and then expound the individual steps in detail. We call
		$$
		A_{c,r} = F^{-1}(Q(c,r))\cap \bigl(a +  [DF(a)]^{-1}(Q(0,r))\bigr).
		$$
		Given that $r$ is very small and $a=F^{-1}(c)$ is a differentiation point for $F$ we can restrict just to $A_{c,r}$ because the complement of $A_{c,r}$ becomes very small and the integral over it disappears by absolute continuity of the integral. We can approximate $f$ by a smooth mapping $f^{\eta}$ as well as we like. For $f^{\eta}_{c,r} \in \mathcal{C}^2$ it becomes obvious that the difference between the linearization and the non-linear integral is very small as soon as $r$ is very small concluding the proof.
		
		From \eqref{2207} we have $Df_{c,r} \in L^p$ and Theorem~\ref{AFPtheo} shows that $f\circ F \in W^{1,p}$. For every $\epsilon>0$ there exists a $\delta >0$ such that whenever $\mathcal{L}^n(E)<\delta$ we have $\int_{E}|Df|^p<\epsilon$. Therefore, it follows from a simple change of variables, that when 
\eqn{star}
$$
\mathcal{L}^n(E)<\delta r^n\text{ we have }\int_{E}|Df_{c,r}|<\epsilon r^n.
$$ 
Almost every point $c$ is such that $a = F^{-1}(c)$ is a point of differentiability of $F$ and a point of approximate continuity of $DF$. Take any such $c$,  fix $r_0$ such that $F$ is $L$-bi-Lipschitz on $F^{-1}(Q(c,r_0))$, fix $0<r_1\leq r_0$ such that when $0<r<r_1$ the set
$$
E_{c,r} = \Bigl\{x\in F^{-1}(Q(c,r)):\ |DF(x) - DF(a)|>\tfrac{\rho}{ C_{\ref{2207}} 10 p(2L)^{p-1}L^n\int_{Q(0,1)}|Df|^p}\Bigr\} 
$$
is so small that by \eqref{star}
\eqn{doublestar}
$$
\int_{F(E_{c,r})} |Df_{c,r}|^p < r^n\tfrac{\rho}{10 p(2L)^{p}L^n}.
$$
Calculate using Lemma~\ref{stupid}, the $L$-Lipschitz quality of $F$, the definition of $E_{c,r}$, $J_F^{-1}<L^n$ in the change of variables formula, \eqref{doublestar} and \eqref{2207}
		\begin{equation}\label{Linda}
			\begin{aligned}
				&\int_{A_{c,r}}\Big||Df_{c,r}(F(x))DF(x)|^p - |Df_{c,r}(F(x))DF(a)|^p\Big|\; dx\\
				&\qquad \leq p \int_{A_{c,r}}\Bigl(|Df_{c,r}(F(x))DF(x)| + |Df_{c,r}(F(x))DF(a)|\Bigr)^{p-1}\\
				&\qquad \qquad \qquad \qquad \bigl|Df_{c,r}(F(x))DF(x) - Df_{c,r}(F(x))DF(a)\bigr|\; dx \\
				&\qquad \leq p(2L)^{p-1}  \int_{A_{c,r}}|Df_{c,r}(F(x))|^{p}|DF(x) -DF(a)| \; dx\\
				& \qquad \leq \frac{\rho}{ C_{\ref{2207}} 10L^n\int_{Q(0,1)}|Df|^p}\int_{A_{c,r}\setminus E_{c,r}}|Df_{c,r}(F(x))|^{p}\; dx
				+ p(2L)^{p-1}\int_{E_{c,r}}|Df_{c,r}(F(x))|^p 2L\; dx\\
				& \qquad \leq \frac{\rho}{ C_{\ref{2207}} 10 L^n\int_{Q(0,1)}|Df|^p}\int_{F^{-1}(Q(c,r))}|Df_{c,r}(F(x))|^p\; dx+p(2L)^p L^n
				\int_{F(E_{c,r})}|Df_{c,r}|^p\\
				& \qquad \leq \frac{\rho}{ C_{\ref{2207}} 10\int_{Q(0,1)}|Df|^p}\int_{Q(c,r)}|Df_{c,r}|^p+\rho\frac{r^n}{10}\leq \frac{\rho}{5} r^n.
			\end{aligned}
		\end{equation}
		
Let us fix $\eta>0$. We can fix $f^{\eta}$ a smooth approximation of $f$ with
$$
\int_{Q(0,1)} |Df - Df^{\eta}|^p < \eta \int_{Q(0,1)} |Df|^p
$$
and calling $f_{c,r}^{\eta}(x) = rf^{\eta}(\tfrac{x-c}{r})+c$ we have that
\eqn{approx}
$$
\int_{Q(c,r)} |Df_{c,r} - Df_{c,r}^{\eta}|^p < \eta \int_{Q(c,r)} |Df_{c,r}|^p.
$$
Clearly
$$
\begin{aligned}
|Df_{c,r}(F)&-  Df_{c,r}(c + DF)|^p \leq 3^p|Df_{c,r}(F)- Df^{\eta}_{c,r}(F)|^p+\\
&+3^p|Df^{\eta}_{c,r}(F)- Df^{\eta}_{c,r}(c + DF)|^p+3^p|Df^{\eta}_{c,r}(c + DF)- Df_{c,r}(c + DF)|^p\\
\end{aligned}
$$
	and therefore we can estimate with the help of Lemma~\ref{stupid}, $|DF(a)|\leq L$ and notation $a=F^{-1}(c)$
		\begin{equation}\label{IntersectionalityIsRubbish}
			\begin{aligned}
				&\int_{A_{c,r}}\Big||Df_{c,r}(F(x))DF(a)|^p - |Df_{c,r}(c + DF(a)(x-a))DF(a)|^p\Big|\\
				&\quad \leq 3^p pL^p\int_{A_{c,r}}2^{p-1}\bigl(|Df_{c,r}(F(x))|^{p-1}+ |Df^{\eta}_{c,r}(F(x))|^{p-1}\bigr)\\
				&\qquad  \qquad   \qquad \qquad \qquad \bigl|Df_{c,r}(F(x)) - Df_{c,r}^{\eta}(F(x))\bigr| \\
&\qquad  + 3^p pL^p\int_{A_{c,r}}2^{p-1}\bigl(|Df_{c,r}^{\eta}(F(x))|^{p-1}+ |Df^{\eta}_{c,r}(c + DF(a)(x-a)\bigr)|^{p-1})\\
				&\qquad\qquad\qquad\qquad\qquad
				\bigl|Df_{c,r}^{\eta}(c + DF(a)(x-a)) - Df_{c,r}^{\eta}(F(x))\bigr| \\
				&\qquad  + 3^p pL^p\int_{A_{c,r}}2^{p-1}\bigl(|Df^{\eta}_{c,r}(c + DF(a)(x-a))|^{p-1}
				+ |Df_{c,r}(c + DF(a)(x-a))|^{p-1}\bigr)\\
				&\qquad\qquad\qquad\qquad\qquad
				 \bigl|Df^{\eta}_{c,r}(c + DF(a)(x-a)) - Df_{c,r}(c + DF(a)(x-a))\bigr|.
			\end{aligned}
		\end{equation}				
		Expanding the parentheses we get 6 terms. We estimate the first term using the H\"{o}lder inequality and \eqref{approx}
		\eqn{square}
		$$
			\begin{aligned}
				& \int_{A_{c,r}}|Df_{c,r}(F(x))|^{p-1}|Df_{c,r}(F(x)) - Df_{c,r}^{\eta}(F(x))| \\
				&\quad \leq\Big(\int_{A_{c,r}}|Df_{c,r}(F(x))|^{p}\Big)^{\tfrac{p-1}{p}} \Big(\int_{A_{c,r}}|Df_{c,r}(F(x)) - Df_{c,r}^{\eta}(F(x))|^p\Big)^{\tfrac{1}{p}}\\
				&\quad \leq L^n\Big(\int_{Q_{c,r}}|Df_{c,r}(y)|^{p}\; dy\Big)^{\tfrac{p-1}{p}}\Big(\int_{Q_{c,r}}|Df_{c,r}(y) - Df_{c,r}^{\eta}(y)|^p\; dy\Big)^{\tfrac{1}{p}}\\
				& \quad \leq L^n\eta^{\tfrac{1}{p}} r^n\int_{Q(0,1)}|Df|^p.
			\end{aligned}
		$$
		Since $\|Df_{c,r}^{\eta}\|_p\leq \|Df_{c,r}\|_p$ we have the same estimate for the second term. The fifth and sixth terms are almost identical and yield the same estimate (up to slightly changing the multiplicative constant). From \eqref{square} we see that by choosing $\eta$ sufficiently small we have
		\begin{equation}\label{ImWorthIt}
			\begin{aligned}
				3^ppL^p2^{p-1}\int_{A_{c,r}}|Df_{c,r}(F(x))|^{p-1}|Df_{c,r}(F(x)) - Df_{c,r}^{\eta}(F(x))|
				& \leq \frac{\rho}{10}r^n
			\end{aligned}
		\end{equation}
		and the same estimate holds for the second, fifth and sixth terms.
		
		It remains to estimate the third and fourth terms on the righthand side of \eqref{IntersectionalityIsRubbish}. Let us call $M_{\eta}$ the Lipschitz constant of $|Df_{c,r}^{\eta}|$. Our $\eta$ is fixed so $M_{\eta}$ is an absolute constant.
		Since $a = F^{-1}(c)$ is a point of differentiability of $F$ we choose $0<r_2<r_1$ so small that for all $0<r<r_2$ we have
		\eqn{est}
		$$
		\|F(x) - c - DF(a)(x-a)\|_{L^{\infty}(Q(c,r))} < \frac{\rho}{C_{\ref{est}}(p) M_{\eta}\int_{Q(0,1)}|Df|^{p-1}}
		$$
		for a fixed constant $C_{\ref{est}}(p)$. 
	Now by \eqref{est} we have for a well-chosen value of $C_{\ref{est}}(p)$ 
		\begin{equation}\label{ThirdFourth}
			\begin{aligned}
				&  \int_{A_{c,r}}\bigl(|Df_{c,r}^{\eta}(F(x))|^{p-1}+ |Df^{\eta}_{c,r}(c + DF(a)(x-a))|^{p-1}\bigr)\\
				& \qquad \qquad\qquad\qquad\qquad \bigl|Df_{c,r}^{\eta}(F(x)) - Df_{c,r}^{\eta}(a + DF(a)(x-a))\bigr|\\
				& \qquad \leq \int_{A_{c,r}}\bigl(|Df_{c,r}^{\eta}(F(x))|^{p-1}+ |Df^{\eta}_{c,r}(c + DF(c)(x-c))|^{p-1}\bigr)\\
				& \qquad \qquad\qquad\qquad\qquad M_{\eta}\bigl\|F(x) - c - DF(a)(x-a)\bigr\|_{L^{\infty}(Q(c,r))} \\
				& \qquad \leq \frac{\rho}{C_{\ref{est}}(p) \int_{Q(0,1)}|Df|^{p-1}}  \int_{A_{c,r}}|Df_{c,r}|^{p-1}\\
				& \qquad \leq \frac{\rho r^n}{10\ 3^ppL^p2^{p-1}}  .
			\end{aligned}
		\end{equation}
		
		Now we apply the estimate \eqref{Linda} and in \eqref{IntersectionalityIsRubbish} we estimate the first, second, fifth and sixth term by \eqref{ImWorthIt}, the third and fourth terms by \eqref{ThirdFourth} to get
		\begin{equation}\label{BeatTheIntersectionality}
				\int_{A_{c,r}}\Big||Df_{c,r}(F(x))DF(x)|^p -| Df_{c,r}(c + DF(a)(x-a))DF(a)|^p\Big|
				 \leq \frac{4}{5} \rho r^n.\\
		\end{equation}
		We consider the remaining part. 
		We choose $\delta$ so that 
		\eqn{acint}
		$$
		\int_{E} |Df|^p < \tfrac{\rho}{10 L^{n}}\text{ as soon as }\mathcal{L}^n(E)\leq \delta.
		$$
		Since $a=F^{-1}(c)$ is a point of differentiability of $F$ we can estimate that
		$$
			\mathcal{L}^n\big([F^{-1}(Q(c,r))\cup (a +  [DF(a)]^{-1}(Q(0,r)))] \setminus A_{c,r}\big)\leq \delta r^nL^{-n}
		$$
		as soon as $r\leq r_3$ (without loss of generality assume that $r_3\leq r_2$) and hence by \eqref{acint}
		$$
		\int_{E} |Df_{c,r}|^p < \tfrac{\rho r^n}{10 L^{n}}\text{ as soon as }\mathcal{L}^n(E)\leq \delta r^n.
		$$
		Therefore, calling
		$$
		S_{c,r} = F^{-1}(Q(c,r)) \setminus A_{c,r}\text{ and }Z_{c,r} = \big(a +  [DF(a)]^{-1}(Q(0,r))\big) \setminus A_{c,r},
		$$
		we have $\mathcal{L}^n(S_{c,r})\leq \delta r^n L^{-n}$, $\mathcal{L}^n(Z_{c,r}) \leq \delta r^n L^{-n}$ and so
		$$
			\int_{S_{c,r}}|Df_{c,r}(F(x))|^p \leq \frac{\rho r^n}{10 L^{n}}\; dx \text{ and } \int_{Z_{c,r}}|Df_{c,r}(F(x))|^p\; dx \leq \frac{\rho r^n}{10 L^{n}}.
		$$
		Adding this to \eqref{BeatTheIntersectionality} we get
		\begin{equation}\label{TheFrench}
		\begin{aligned}
				&\Big|\int_{F^{-1}(Q(c,r))} |Df_{c,r}(F(x))DF(x)|^p\; dx\\
	 &\qquad - \int_{[DF(a)]^{-1}(Q(0,r))} |Df_{c,r}(c + DF(a)(x-a))DF(a)|^p\Big|\; dx
 \leq \rho r^n.
		\end{aligned}
		\end{equation}
		which proves the claim of \eqref{TheSpiritIsMoving}. The equation~\eqref{WeAllLive} is proved by applying the above estimates to the set $F^{-1}(X_{c,r})$ (respectively $a + [DF(a)]^{-1}(rX))$.
	\end{proof}

\subsection{Covering lemmata}

In our main proof we construct mapping and some closed set where the Jacobian is negative. We do not alter our mapping on this closed set and we cover the remaining open set $\Omega$ by (small enough) disjoint cubes where we compose our mapping with a translated and rotated copy of the previous construction.
	
	\begin{lemma}\label{StickThatUpYourPipeAndSmokeItTillTheCowsComeHome}
		Let $\Omega \subset \er^n$ be an open set and let $\mathcal {N}\subset \Omega$ satisfy $\mathcal{L}^n(\mathcal {N})=0$.
Assume that for every $c\in\Omega\setminus \mathcal {N}$ we have a number $r_c >0$. Further let $S:\Omega\setminus \mathcal {N} \to \er^4$ be a mapping such that $|S(c)| = 1$. By $O_c$ denote a sense-preserving unitary map such that $O_cS(c)= e_1$. By $\mathcal{Q}_{c,r}$ denote the set $c+ O_c^{-1}(Q(0,r))$. Then, we can find a countable system of rotated cubes $\mathcal{Q}_{c_i, r_i}\subset\Omega$ with pairwise disjoint interiors such that $c_i\in\Omega\setminus \mathcal {N}$, $r_i<r_{c_i}$ for every $i\in \en$ and $\mathcal{L}^n(\Omega \setminus \bigcup_i \mathcal{Q}_{c_i, r_i} ) = 0$.
	\end{lemma}
	\begin{proof}
		To prove our claim it suffices to prove that there exists an $\alpha >0$ such that for any open $\Omega$ we can find a finite number of $\mathcal{Q}_{c_i,r_i} \subset \Omega$, $i=1,2,\dots, I_{\Omega}$ with disjoint interiors so that
		$$
			\mathcal{L}^n\Big(\bigcup_{i=1}^{I_{\Omega}}\mathcal{Q}_{c_i,r_i}\Big) \geq \alpha \mathcal{L}^n(\Omega).
		$$
		By iterating the above process (applied to the new open set $\Omega\setminus \bigcup_{i=1}^{I_{\Omega}}\overline{\mathcal{Q}_{c_i,r_i}}$) we get precisely the claim. Therefore we prove the above claim.
		
		We segregate $\Omega$ by a Whitney decomposition and select a finite number of cubes $Q_1, Q_2, \dots, Q_m$ of that covering so that $\mathcal{L}^n(\bigcup_{i=1}^mQ_i)\geq \tfrac{1}{2}\mathcal{L}^n(\Omega)$. From here on we work inside a single cube $Q_1$ and therefore, without loss of generality, we may assume that $Q_1 = Q(0,1)$.
		
		We can find an $r_0$ small enough such that
		\begin{equation}\label{ElbowRoom}
			\mathcal{L}^n\bigl(\{c\in Q(0,1): r_c>r_0\}\bigr) > \mathcal{L}^n(Q(0,1)) -\tfrac{1}{2}.
		\end{equation}
Now we choose $k\in \en$ so that $\frac{1}{k}<r_0$. We separate $Q(0,1)$ into $k^n$ identical cubes of type $Q(x,\tfrac{1}{k})$. Clearly for any unitary map $O : \rn \to \rn$ we have that $O\big(Q(0,\tfrac{r}{\sqrt{n}})\big)\subset B(0,r) \subset Q(0,r)$. Therefore given a cube of type $Q(x,\tfrac{1}{k})$ and any $y\in Q(0,\tfrac{1}{2k})$ we have that
		\begin{equation}\label{BreathingSpace}
			Q(x,\tfrac{1}{k}) \supset Q(x+y, \tfrac{1}{2k}) \supset B(x+y, \tfrac{1}{2k}) \supset x+y+O(Q(0,\tfrac{1}{\sqrt{n}2k})).
		\end{equation}
		The center of each rotated cube in \eqref{BreathingSpace} is of the form $x+y$ where $x$ is the center of the large cube of radius $\tfrac{1}{k}$ and $y \in  Q(0,\tfrac{1}{2k})$. The number of such cubes in $Q(0,1)$ is $k^n$. Therefore the measure of the set of possible centers of the rotated cubes is $k^n (\tfrac{1}{2k})^n 2^n = 1$. Therefore by \eqref{ElbowRoom} there must be at least $\tfrac{k^n}{2}$ cubes of type-$Q(x,\tfrac{1}{2k})$ which intersect the set
		$\{c\in Q(0,1): r_c<r_0\}$ from \eqref{ElbowRoom}. That is to say we can find at least $k^n/2$ points $z=x+y$, each in a separate cube of type-$Q(x,\tfrac{1}{2k})$ such that $z+ O_z(Q(0,\tfrac{1}{\sqrt{n}2k}))$ are all pairwise disjoint by \eqref{BreathingSpace} where $O_z$ is a sense-preserving unitary map with $O_z(S(z)) = e_1$. The measure of the set covered by this collection of rotated cubes is at least
		$$
			\frac{k^n}{2}\frac{2^n}{n^{\frac{n}{2}}2^nk^n} = \frac{1}{2n^{n/2}} = 2\alpha.
		$$
		We repeat this in all of the chosen cubes $Q_1, \dots Q_m$ which together make at least half of the measure of $\Omega$ and so our rotated cubes (of which we have a finite number) cover a set of measure at least $\alpha \mathcal{L}^n(\Omega)$. Iterating this technique we arrive at a covering of all of $\Omega$ up to a closed null set by rotated cubes with pairwise disjoint interiors.
	\end{proof}

	\begin{corollary}\label{NoCakeForLosers}
	Let $\Omega \subset \er^n$ be an open set and let $\mathcal {N}\subset \Omega$ with $\mathcal{L}^n(\mathcal {N})=0$.
Assume that for every $c\in\Omega\setminus \mathcal {N}$ we have $r_c>0$. Then, we can find a countable system of cubes $Q(c_i, r_i)\subset\Omega$ with pairwise disjoint interiors such that $c_i\in\Omega\setminus N$, $r_i<r_{c_i}$ for every $i\in \en$ and $\mathcal{L}^n(\Omega \setminus \bigcup_i Q(c_i, r_i) ) = 0$.
	\end{corollary}
	\begin{proof}
		It suffices to apply Lemma~\ref{StickThatUpYourPipeAndSmokeItTillTheCowsComeHome} with $S(c) =e_1$ for all $c \in \Omega\setminus \mathcal {N}$ and $O_c = \id$.
	\end{proof}

\section{Construction of Cantor sets and mappings between them}\label{Cantor}

\subsection{A map that has negative Jacobian on a set of positive measure and equals to identity on the boundary}
	Sections~\ref{Cantor}, \ref{CarCrash}, \ref{Germany}, \ref{PardubiceJeJenVesniceUHradce} and \ref{TheBigDeal} are dedicated to proving Theorem~\ref{TheBigLebowski}. Section~\ref{Cantor} introduces the concept of Cantor sets and defines some mappings between them. Section~\ref{CarCrash} ensures identity on the boundary of our map. Sections~\ref{Germany} and~\ref{PardubiceJeJenVesniceUHradce} list the necessary estimates of the derivatives needed to prove Theorem~\ref{TheBigLebowski} and we combine all the results in the proof in Section~\ref{TheBigDeal}.
	
	\begin{thm}\label{TheBigLebowski}Let $n=4$ and $1\leq p< \frac{3}{2}$.
	For every $\epsilon>0$ there exists a closed set $E \subset Q(0,1)$ and a map $f_1 \in W^{1,p}(Q(0,1), \er^4)$ such that,
	\begin{enumerate}
		\item[$i)$] $f_1(x) = x$ for $x\in \partial Q(0,1)$,
		\item[$ii)$] $f_1$ is locally bi-Lipschitz on $Q(0,1) \setminus E$,
		\item[$iii)$] $J_{f_1}<0$ on $E$,
		\item[$iv)$] $\mathcal{L}^4(Q(0,1)\setminus E) < \epsilon$,
		\item[$v)$] for $j=1,2,3$ we have
		$$
			\int_{Q(0,1)\setminus E} |D_j f_1|^p \int_{Q(0,1)\setminus E} |D f_1|^p \leq \frac{1}{12}.
		$$
	\end{enumerate}
	\end{thm}

\subsection{Notation}\label{Notation}
The construction depends on a large parameter $m\in\en$ whose value is chosen later.
	We record some notation here, which we use throughout Sections~\ref{Cantor}, \ref{CarCrash}, \ref{Germany}, \ref{PardubiceJeJenVesniceUHradce} and \ref{TheBigDeal}. We define
	\begin{equation}\label{defm}
	R_{m,t} = [-2m-5,2m+5]^3\times[-1-t,1+t]
\end{equation}
Later we work in detail with $R_{m,2}$, $R_{m,13}$ and especially $R_{m,\eta}$ for $\eta >0 $ small. Eventually we choose $\eta  = K^{-\alpha}$, where $K$ and $\alpha$ are parameters that define our Cantor type set of positive measure (see \eqref{defK} below). 

	We define
	$$
	Z_m = [-2m-1,2m+1]^3\times [-1, 1].
	$$

\subsection{Construction of Cantor type sets}\label{first}
	Let $n=3$ or $n=4$. Given a sequence of numbers $\{s_k\}_{k=0}^\infty$, such that $2^ks_k$ is decreasing with $s_0=1$, we define the Cantor-type set in $\rn$ corresponding to the sequence $\{s_k\}$ as is described below. Call $\V_n$ the set of $2^n$ vertices of the cube $[-1,1]^n$.  We set $z_0=\z_0=0$ and $[-1,1]^{n}=Q(z_0,\tilde{s}_0)$ and further we proceed by induction.
	For $\ve=[v_1,\hdots, v_k]\in \V^{k}_n$ we denote $\w=[v_1,\hdots,v_{k-1}]$ and we define
	\begin{equation}\label{GiveZ}
		z_{\ve}=z_{\w}+\frac{1}{2} {s}_{k-1} v_k=z_0+\frac{1}{2}\sum_{j=1}^k {s}_{j-1}v_j.
	\end{equation}
	Then we define
	\begin{equation}\label{SueMe}
		{Q}'_{\ve}=Q(z_{\ve},\tfrac{{{s}}_{k-1}}{2})\ \text{ and }\ {Q}_{\ve}=Q(z_{\ve},{s}_k).
	\end{equation}
	
\vskip -10pt	
\unitlength=0.7mm
\begin{figure}[h!]
\begin{picture}(110,60)(0,-5)
\put(10,10){\framebox(40,40){}}
\put(30,10){\line(0,1){40}}
\put(10,30){\line(1,0){40}}
\put(12,12){\framebox(16,16){}}

\put(12,32){\framebox(16,16){}}
\put(32,12){\framebox(16,16){}}
\put(32,32){\framebox(16,16){}}
\put(60,10){\framebox(40,40){}}
\put(80,10){\line(0,1){40}}
\put(60,30){\line(1,0){40}}
\put(62,12){\framebox(16,16){}}
\put(70,12){\line(0,1){16}}
\put(62,20){\line(1,0){16}}
\put(63,13){\framebox(6,6){}}
\put(71,13){\framebox(6,6){}}
\put(63,21){\framebox(6,6){}}
\put(71,21){\framebox(6,6){}}
\put(62,32){\framebox(16,16){}}
\put(70,32){\line(0,1){16}}
\put(62,40){\line(1,0){16}}
\put(63,33){\framebox(6,6){}}
\put(71,33){\framebox(6,6){}}
\put(63,41){\framebox(6,6){}}
\put(71,41){\framebox(6,6){}}
\put(82,12){\framebox(16,16){}}
\put(90,12){\line(0,1){16}}
\put(82,20){\line(1,0){16}}
\put(83,13){\framebox(6,6){}}
\put(91,13){\framebox(6,6){}}
\put(83,21){\framebox(6,6){}}
\put(91,21){\framebox(6,6){}}
\put(82,32){\framebox(16,16){}}
\put(90,32){\line(0,1){16}}
\put(82,40){\line(1,0){16}}
\put(83,33){\framebox(6,6){}}
\put(91,33){\framebox(6,6){}}
\put(83,41){\framebox(6,6){}}
\put(91,41){\framebox(6,6){}}
\end{picture}
\vskip -30pt
\caption{Squares
$Q_{\ve}$ and ${Q'}_{\ve}$ for $k=1,2$ and $n=2$
}
\end{figure}
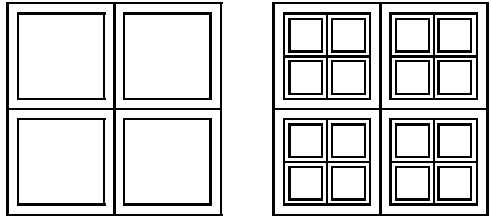

	We refer to the set
	$$
	\C = \bigcap_{k=1}^{\infty} \bigcup_{\ve\in\V^k_n} {Q}_{\ve}
	$$
	as the Cantor set corresponding to the sequence $s_k$.

	The number of the cubes $\{Q_{\ve}:\ \ve\in\V^k_n\}$ is $2^{nk}$ and the measure of each cube of generation $k$ is $2^ns_k^n$. Therefore 
	$$
		\mathcal{L}^n(\C) = \mathcal{L}^n\Big(\bigcap_{k=1}^{\infty} \bigcup_{\ve\in\V^k_n} {Q}_{\ve}\Big)  =\lim_{k\to\infty} 2^{nk}2^ns_k^n.
	$$
	Notice that the projection of the sets $\bigcup_{\ve\in\V^k_n} {Q}_{\ve}$ onto $\er^d$, $d<n$ is the same as the repeating the original construction with $s_k$ in $\er^d$. 
It follows that $\C(n)$ is a Cartesian product of corresponding one dimensional Cantor sets $\C(1)$.

	If $\C(4)$ is the Cantor set constructed in dimension $4$ corresponding to $s_k$, we refer to the set
	$$
		\C_m:=\bigcup_{n_1,n_2,n_3=-m}^m \bigl(\C+[2n_1,2n_2,2n_3,0]\bigr)
	$$
	as a $4$-dimensional `$m$-Cantor plate' corresponding to the sequence $s_k$ i.e. in each of the $(2m+1)^3$ cubes $Q((2n_1,2n_2,2n_3,0),1)$ we have a copy of $\C$, for $-m\leq n_1,n_2,n_3 \leq m$.

	We use two specific sets $\C_{A,K},\ \C_{B}\subset [-1,1]^n$. For $n=4$ we use the notation $\C_{A, K,m}$ and $\C_{B,m}$ for the corresponding $4$-dimensional $m$-Cantor plates in $Z_m$. 	
	We fix the parameters $\alpha>0$ and $K \in \en$ whose exact value is chosen later. 
	In fact at the end we choose $\alpha=\frac{2+p}{3-2p}$ (recall that $1\leq p<3/2$) and $K$ is so big that the measure of $\C_{A,K}$ is really close to $\mira([-1,1]^n)$. 
	We define
	\begin{equation}\label{defK}
		\tilde{r}_k(K) = \frac{2^{-k}}{1+ (K+1)^{-\alpha}}(1+ (K+k+1)^{-\alpha}).
	\end{equation}
	Note that $2^k\tilde{r}_k(K)$ is decreasing and $\tilde{r}_0(K)=1$. 
	By $\C_{A,K}$ we denote the Cantor set constructed using the sequence $\tilde{r}_k(K)$. 
	We use the construction predominantly in the 4 dimensional case but also sometimes in 3 dimensions. When necessary we denote the set constructed in $\rn$ by $\C_{A,K}(n)$, resp $\C_{B}(n)$. 
	We call the points calculated by the equation \eqref{GiveZ} for this sequence $\z_{K,\ve}$. Given a $\ve \in \V^k_n$ we call
	$$\tilde{Q}_{K, \ve} = Q(\z_{K, \ve},\tilde{r}_k(K) )\text{ and we call }
	\tilde{U}_{K,k}^n = \bigcup_{\ve\in\V^k_n} \tilde{Q}_{K, \ve}.
	$$
	The set $\tilde{U}_{K,k}^n$ equals to the Cartesian product $[\tilde{U}_{K,k}^1]^n$. 
	Clearly 
	\begin{equation}\label{ComparingSizes}
		\begin{aligned}
		\mathcal{L}^n(C_{A,K})&=\lim_{k\to\infty}\mathcal{L}^n(\tilde {U}_{K,k}) =
			\lim_{k\to\infty}2^{nk}(2\tilde{r}_k(K))^n \\
			&=\lim_{k\to\infty} 2^{n}\frac{\bigl(1+ \frac{1}{(K+k+1)^\alpha}\bigl)^n}{\bigl(1+ \frac{1}{(K+1)^\alpha}\bigr)^n} = \frac{2^n}{\bigl(1+ (K+1)^{-\alpha}\bigr)^n}
		\end{aligned}
	\end{equation}
	and this tends to $2^n = \mathcal{L}^n(Q(0,1))$ as $K \to \infty$. More precisely
	\eqn{rozdil}
	$$
	\mathcal{L}^n([-1,1]^n\setminus C_{A,K})=2^n-\frac{2^n}{\bigl(1+ (K+1)^{-\alpha}\bigr)^n}
	\leq C(n,\alpha)(K+1)^{-\alpha}. 
	$$
	
	If by $\tilde{I}$ we denote one of the intervals of $\tilde{U}_{K,k}^1\setminus \tilde{U}_{K,k+1}^1$, then we have
	\begin{equation}\label{DistancesEstimate1}
	\begin{aligned}
		\diam \tilde{I} &= \tfrac{1}{2}\tilde{r}_k(K) - \tilde{r}_{k+1}(K) \\
		&= \frac{2^{-k-1}}{1+ (K+1)^{-\alpha}}((K+k+1)^{-\alpha} - (K+k+2)^{-\alpha})\\
		&\approx C2^{-k}(K+k)^{-\alpha -1}
	\end{aligned}
	\end{equation}
	and the constant $C$ does not depend on $K$.

  We fix a parameter $\beta>2$ whose exact value we specify later. In fact this is an absolute constant and we fix it big enough to that there is enough room around $\C_B$ to perform our construction. 
	For $k\in \en_0$ we write
\begin{equation}\label{defrKk}
{r}_k = \frac{1}{2^{k}}2^{-\beta k }
\end{equation}
and we call the Cantor set constructed using the sequence $r_k$ as $\C_{B}$. 
	We denote the points calculated by the equation \eqref{GiveZ} for this sequence $z_{\ve}$. Given a $\ve \in \V^k_n$ we call
	$$
	{Q}_{\ve} = Q(z_{\ve},r_k )\text{ and we call }{U}_{k}^n = \bigcup_{\ve\in\V^k_n} {Q}_{\ve}.
	$$
	Again ${U}_{k}^n = [{U}_{k}^1]^n$. If ${I}$ is an interval of ${U}_{k}^1\setminus {U}_{k+1}^1$ then
	\begin{equation}\label{DistancesEstimate2}
		\begin{aligned}
			\diam I &= \tfrac{1}{2}r_k - r_{k+1} \\
			&= 2^{-k(\beta+1)}(\tfrac{1}{2} - 2^{-\beta})\\
			&\approx C2^{-k(\beta+1)}
		\end{aligned}
	\end{equation}
	because we consider $\beta > 2$.


\subsection{A standard homeomorphism $J_K^n$ that maps $\C_{A,K}$ onto $\C_{B}$.}\label{defHJ}
	The construction of the standard `frame-to-frame'  map can be found in \cite[Chapter 4.3]{HK}. An $\infty$-norm radial map maps cubes onto cubes. It is standard that we can map $\tilde{Q}_{K,\ve}'\setminus \tilde{Q}_{K,\ve}$ onto $Q_{\ve}'\setminus {Q}_{\ve}$ for each $\ve \in \V^k_n$ for each $k$ by a ($\infty$-norm) radial map centered at $\z_{K,\ve}$ and $z_{\ve}$. This radial map is injective and since the cubes $Q_{ \ve}'$ are pairwise disjoint (up to their possible common boundary) the map is injective from each $\tilde{U}_{K,k}^n \setminus \tilde{U}_{K,k}^n$. We conduct this on $\tilde{U}_{K,k}^n \setminus \tilde{U}_{K,k}^n$ for each $k$. Take a $\ve\in \V^{\en}_n$, and call $\ve(k)\in\V^k_n$ the element that is equal to the first $k$ components of $\ve$. The map we are constructing sends $\Q_{K,\ve(k)}$ onto $Q_{ \ve(k)}$ for each $k$. Therefore it becomes immediately obvious that each $\z_{K, \ve}\in \C_{A,K}$ is mapped onto $z_{\ve}$ and that the mapping is continuous on $\C_{A,K}$. It is standard that this map is continuous on each $Q(0,1)\setminus \tilde{U}_{K,k}^n$ for each $k$ and extends as a homeomorphism onto $Q(0,1)$ sending $\C_{A,K}$ onto $\C_{B}$. We denote this map by $J_K^n$. The same construction can be conducted in the opposite direction, mapping $\C_{B}$ onto $\C_{A,K}$ and we call it $H_K^n$. Especially we denote $J_K^1 = q_K$ and $H_K^1 = t_K$; to be explicit
	\begin{equation}\label{linear}
	\begin{aligned}
	&q_K\text{ is the continuous extension of the map that is linear on each interval }\\
	&\text{in }\tilde{U}_{K,k-1}^1\setminus\tilde{U}_{K,k}^1\text{ sending it onto the corresponding interval in }{U}_{k-1}^1\setminus {U}_{k}^1\\
	\end{aligned}
	\end{equation}
	(for $k\in \en$) and the constructed function $t_K$ is its inverse. 
	
	Further, for $l\in\en$ we define a continuous map $[J_K^n]_{l}$ such that 
	$\lim_{l\to\infty}[J_K^n]_{l}=J_K^n$. We define 
	$$
	[J_K^n]_{l}=J_K^n\text{ on }Q(0,1)\setminus \tilde{U}_{K,l}
	$$
	and it maps each $\tilde{Q}_{K,\ve} = Q(\z_{K,\ve},\tilde{r}_l(K) )$, $\ve \in \V^l_n$, linearly onto $Q_{\ve}$, i.e. 
\eqn{basic}
	$$
	[J_K^n]_{l}(Q(\z_{K,\ve},\tilde{r}_l(K) )) = Q(z_{\ve},r_l ) = Q_{\ve}\text{ for every }\ve \in \V^l_n.
	$$
 The same construction applied in reverse gives $\big([J_K^n]_{l}\big)^{-1} = [H_K^n]_{l}$. 
 That is $[H_K^n]_{l}=H_K^n$ on $Q(0,1)\setminus {U}_{l}$ and it maps $Q_{\ve}$ linearly on $\tilde{Q}_{K,\ve}$ for each 
 $\ve \in \V^l_n$.

\subsection{A specific homeomorphism that maps $\C_{B}(4)$ onto $\C_{A,K}(4)$ dependent on a parameter $N$.}\label{secondmap}

	In this section we define a map $G_{K,N,m,\eta}$ of course depending on $\alpha$ and $\beta$ that maps $\C_{B}(4)$ onto $\C_{A,K}(4)$. In order to do so we use mappings of type $[H_K^3]_{l}$ on $3$-dimensional hyperplanes perpendicular to $e_i$, the $i$-th canonical vector. We define
	\begin{equation}\label{defT1}
	T_i(x_1,x_2,x_3, x_4):= (x_{j_1},x_{j_2}, x_{j_3})\text{ where }j_1<j_2<j_3\text{ and }j_l\neq i\text{ for all }l.
	\end{equation}
	Futher we define
	\begin{equation}\label{defT2}
	T^i(x_{j_1},x_{j_2}, x_{j_3}):= (x_1,\dots 0,\dots, x_4)\text{ where the }0\text{ is on the }i\text{-th place}.
	\end{equation}
	Then
	$$
	H^{3,i}_K(x) =T^i\bigl(H^3_K(T_i(x))\bigr)
	\text{ and }[H_K^{3,i}]_{l}(x) =T^i\bigl([H_K^{3}]_{l}(T_i(x))\bigr).
	$$
	In this case the map $[H^{3,i}_0]_k$ corresponds to the definition of $H^{3,i}_k$ from \cite[Section 3]{CHT}.

We divide ${Q}_{\ve}'\setminus {Q}_{\ve}$, $\ve\in \V_4^k$, into parts where we are farthest from $z_{\ve}$ in the $i$-th direction, i.e. 
$$
S_{\ve, i}:=\bigl\{x\in {Q}_{\ve}'\setminus {Q}_{\ve}:\ \|x-z_{\ve}\|_{\infty}=|x_i-(z_{\ve})_i|\bigr\}. 
$$
	It was proved in \cite[Section 3]{CHT} that the map defined in each frame ${Q}_{\ve}'\setminus {Q}_{\ve}$ (of the 4-dimensional cubes) for $\ve\in \V_4^k$ and $x\in S_{\ve, i}$ by the convex combination of $[H^{3,i}_0]_{3k} + t(x_i)e_i$ and $[H^{3,i}_0]_{3k+3}+ t(x_i)e_i $ is a homeomorphism which maps each frame ${Q}_{\ve}'\setminus {Q}_{\ve}$ onto $\tilde{Q}_{0,\ve}'\setminus \tilde{Q}_{0,\ve}$. In fact the difference $3k$ and $3k+3$ used there is immaterial. We define our map analogously, we find a continuous function ${\zeta}_{K,k}$ so that 
	$$
		{\zeta}_{K,k}(s) =\begin{cases}
			1  & s\in {U}_{k+1}^1\\
			0  & s\in \er \setminus {U}_{k}^1\\
			\text{linear on intervals in} & {U}_{k}^1\setminus {U}_{k+1}^1.\\
		\end{cases}
	$$

	Fix $l_1,l_2\in\en$ with $l_2>l_1 \geq k$. We define a mapping ${G}_{K,l_1,l_2,k}$ on $U_{k}^4\setminus U_{k+1}^4$, which for each $\w \in \V^{k+1}_4$ and $x\in S_{\w, i}\subset {Q}_{\w}' \setminus {Q}_{\w}$ 
	is defined as 
	\begin{equation}\label{SpeciallyForStanda}
		{G}_{K,l_1,l_2,k}(x) := \zeta_{K,k}(x_i)[H^{3,i}_K]_{l_1}(x) +[1-\zeta_{K,k}(x_i)][H^{3,i}_K]_{l_2}(x) + t_K(x_i)e_i.
	\end{equation}
	We claim that for any $l_2>l_1 \geq k$, the mapping ${G}_{K,l_1,l_2,k}$ is a homeomorphism. 
		
	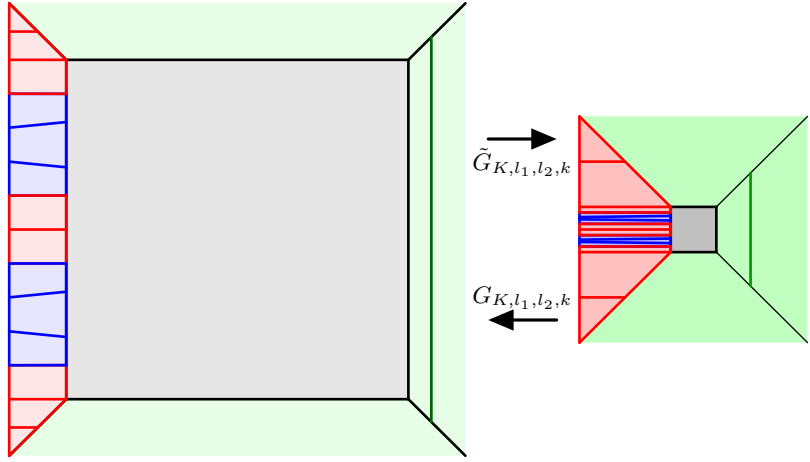
\begin{figure}
		\centering
		\begin{tikzpicture}[line cap=round,line join=round,>=triangle 45,x=1.5cm,y=1.5cm]
			\clip(-0.2,-0.2) rectangle (7.2,4.2);
			\fill[line width=1.pt,fill=black,fill opacity=0.10000000149011612] (0.5,3.5) -- (0.5,0.5) -- (3.5,0.5) -- (3.5,3.5) -- cycle;
			\fill[line width=1.pt,color=qqqqff,fill=qqqqff,fill opacity=0.10000000149011612] (0.,3.2) -- (0.5,3.1995062528319886) -- (0.5,2.3) -- (0.,2.3) -- cycle;
			\fill[line width=1.pt,color=ffqqqq,fill=ffqqqq,fill opacity=0.10000000149011612] (0.,3.2) -- (0.,4.) -- (0.5,3.5) -- (0.5,3.1995062528319886) -- cycle;
			\fill[line width=1.pt,color=ffqqqq,fill=ffqqqq,fill opacity=0.10000000149011612] (0.,2.3) -- (0.,1.7) -- (0.5,1.7) -- (0.5,2.3) -- cycle;
			\fill[line width=1.pt,color=ffqqqq,fill=ffqqqq,fill opacity=0.10000000149011612] (0.,0.8) -- (0.,0.) -- (0.5,0.5) -- (0.5,0.8008218716001494) -- cycle;
			\fill[line width=1.pt,color=qqqqff,fill=qqqqff,fill opacity=0.10000000149011612] (0.,1.7) -- (0.,0.8) -- (0.5,0.8008218716001494) -- (0.5,1.7) -- cycle;
			\fill[line width=1.pt,color=qqffqq,fill=qqffqq,fill opacity=0.10000000149011612] (0.,4.) -- (0.5,3.5) -- (3.5,3.5) -- (4.,4.) -- cycle;
			\fill[line width=1.pt,color=qqffqq,fill=qqffqq,fill opacity=0.10000000149011612] (3.5,0.5) -- (4.,0.) -- (0.,0.) -- (0.5,0.5) -- cycle;
			\fill[line width=1.pt,color=qqffqq,fill=qqffqq,fill opacity=0.10000000149011612] (4.,4.) -- (4.,0.) -- (3.5,0.5) -- (3.5,3.5) -- cycle;
			\fill[line width=1.pt,fill=black,fill opacity=0.25] (5.8,2.2) -- (5.8,1.8) -- (6.2,1.8) -- (6.2,2.2) -- cycle;
			\fill[line width=1.pt,color=qqffqq,fill=qqffqq,fill opacity=0.25] (5.,3.) -- (5.8,2.2) -- (6.2,2.2) -- (7.,3.) -- cycle;
			\fill[line width=1.pt,color=qqffqq,fill=qqffqq,fill opacity=0.25] (5.,1.) -- (5.8,1.8) -- (6.2,1.8) -- (7.,1.) -- cycle;
			\fill[line width=1.pt,color=qqffqq,fill=qqffqq,fill opacity=0.25] (7.,3.) -- (7.,1.) -- (6.2,1.8) -- (6.2,2.2) -- cycle;
			\fill[line width=1.pt,color=qqqqff,fill=qqqqff,fill opacity=0.25] (5.8,2.15) -- (5.,2.15) -- (5.,2.05) -- (5.8,2.05) -- cycle;
			\fill[line width=1.pt,color=qqqqff,fill=qqqqff,fill opacity=0.25] (5.8,1.85) -- (5.8,1.95) -- (5.,1.95) -- (5.,1.85) -- cycle;
			\fill[line width=1.pt,color=ffqqqq,fill=ffqqqq,fill opacity=0.25] (5.,2.05) -- (5.,1.95) -- (5.8,1.95) -- (5.8,2.05) -- cycle;
			\fill[line width=1.pt,color=ffqqqq,fill=ffqqqq,fill opacity=0.25] (5.,3.) -- (5.8,2.2) -- (5.8,2.15) -- (5.,2.15) -- cycle;
			\fill[line width=1.pt,color=ffqqqq,fill=ffqqqq,fill opacity=0.25] (5.,1.85) -- (5.,1.) -- (5.8,1.8) -- (5.8,1.85) -- cycle;
			\draw [line width=1.pt] (0.5,3.5)-- (0.5,0.5);
			\draw [line width=1.pt] (0.5,0.5)-- (3.5,0.5);
			\draw [line width=1.pt] (3.5,0.5)-- (3.5,3.5);
			\draw [line width=1.pt] (3.5,3.5)-- (0.5,3.5);
			\draw [line width=1.pt] (0.,4.)-- (0.5,3.5);
			\draw [line width=1.pt] (0.,0.)-- (0.5,0.5);
			\draw [line width=1.pt] (3.5,3.5)-- (4.,4.);
			\draw [line width=1.pt] (3.5,0.5)-- (4.,0.);
			
			\draw [line width=1.pt,color=qqqqff] (0., 2.9)-- (0.5,2.95);
			\draw [line width=1.pt,color=qqqqff] (0., 2.6)-- (0.5,2.55);
			\draw [line width=1.pt,color=qqqqff] (0., 1.4)-- (0.5,1.45);
			\draw [line width=1.pt,color=qqqqff] (0., 1.1)-- (0.5,1.05);
			
			\draw [line width=1.pt,color=qqqqff] (5., 2.09)-- (5.8,2.08);
			\draw [line width=1.pt,color=qqqqff] (5., 2.11)-- (5.8,2.12);
			\draw [line width=1.pt,color=qqqqff] (5., 1.89)-- (5.8,1.88);
			\draw [line width=1.pt,color=qqqqff] (5., 1.91)-- (5.8,1.92);

			\draw [line width=1.pt,color=ffqqqq] (0., 3.75)-- (0.25,3.75);
			\draw [line width=1.pt,color=ffqqqq] (0., 3.5)-- (0.5,3.5);
			\draw [line width=1.pt,color=ffqqqq] (0., 2)-- (0.5,2);
			\draw [line width=1.pt,color=ffqqqq] (0., 0.5)-- (0.5,0.5);
			\draw [line width=1.pt,color=ffqqqq] (0., 0.25)-- (0.25,0.25);
			
			\draw [line width=0.5pt] (6.2, 2.2)-- (7,3);
			\draw [line width=0.5pt] (6.2, 1.8)-- (7,1);
			
			\draw [line width=1.pt,color=ffqqqq] (5., 2.6)-- (5.4,2.6);
			\draw [line width=1.pt,color=ffqqqq] (5., 2.2)-- (5.8,2.2);
			\draw [line width=1.pt,color=ffqqqq] (5., 2)-- (5.8,2);
			\draw [line width=1.pt,color=ffqqqq] (5., 1.8)-- (5.8,1.8);
			\draw [line width=1.pt,color=ffqqqq] (5., 1.4)-- (5.4,1.4);
			
			\draw [line width=1.pt,color=qqqqff] (0.,3.2)-- (0.5,3.1995062528319886);
			\draw [line width=1.pt,color=qqqqff] (0.5,3.1995062528319886)-- (0.5,2.3);
			\draw [line width=1.pt,color=qqqqff] (0.5,2.3)-- (0.,2.3);
			\draw [line width=1.pt,color=qqqqff] (0.,2.3)-- (0.,3.2);
			\draw [line width=1.pt,color=ffqqqq] (0.,3.2)-- (0.,4.);
			\draw [line width=1.pt,color=ffqqqq] (0.,4.)-- (0.5,3.5);
			\draw [line width=1.pt,color=ffqqqq] (0.5,3.5)-- (0.5,3.1995062528319886);
			\draw [line width=1.pt,color=ffqqqq] (0.5,3.1995062528319886)-- (0.,3.2);
			\draw [line width=1.pt,color=ffqqqq] (0.,2.3)-- (0.,1.7);
			\draw [line width=1.pt,color=ffqqqq] (0.,1.7)-- (0.5,1.7);
			\draw [line width=1.pt,color=ffqqqq] (0.5,1.7)-- (0.5,2.3);
			\draw [line width=1.pt,color=ffqqqq] (0.5,2.3)-- (0.,2.3);
			\draw [line width=1.pt,color=ffqqqq] (0.,0.8)-- (0.,0.);
			\draw [line width=1.pt,color=ffqqqq] (0.,0.)-- (0.5,0.5);
			\draw [line width=1.pt,color=ffqqqq] (0.5,0.5)-- (0.5,0.8008218716001494);
			\draw [line width=1.pt,color=ffqqqq] (0.5,0.8008218716001494)-- (0.,0.8);
			\draw [line width=1.pt,color=qqqqff] (0.,1.7)-- (0.,0.8);
			\draw [line width=1.pt,color=qqqqff] (0.,0.8)-- (0.5,0.8008218716001494);
			\draw [line width=1.pt,color=qqqqff] (0.5,0.8008218716001494)-- (0.5,1.7);
			\draw [line width=1.pt,color=qqqqff] (0.5,1.7)-- (0.,1.7);
			\draw [line width=1.pt] (5.8,2.2)-- (5.8,1.8);
			\draw [line width=1.pt] (5.8,1.8)-- (6.2,1.8);
			\draw [line width=1.pt] (6.2,1.8)-- (6.2,2.2);
			\draw [line width=1.pt] (6.2,2.2)-- (5.8,2.2);
			\draw [line width=1.pt,color=qqqqff] (5.8,2.15)-- (5.,2.15);
			\draw [line width=1.pt,color=qqqqff] (5.,2.15)-- (5.,2.05);
			\draw [line width=1.pt,color=qqqqff] (5.,2.05)-- (5.8,2.05);
			\draw [line width=1.pt,color=qqqqff] (5.8,2.05)-- (5.8,2.15);
			\draw [line width=1.pt,color=qqqqff] (5.8,1.85)-- (5.8,1.95);
			\draw [line width=1.pt,color=qqqqff] (5.8,1.95)-- (5.,1.95);
			\draw [line width=1.pt,color=qqqqff] (5.,1.95)-- (5.,1.85);
			\draw [line width=1.pt,color=qqqqff] (5.,1.85)-- (5.8,1.85);
			\draw [line width=1.pt,color=ffqqqq] (5.,2.05)-- (5.,1.95);
			\draw [line width=1.pt,color=ffqqqq] (5.,1.95)-- (5.8,1.95);
			\draw [line width=1.pt,color=ffqqqq] (5.8,1.95)-- (5.8,2.05);
			\draw [line width=1.pt,color=ffqqqq] (5.8,2.05)-- (5.,2.05);
			\draw [line width=1.pt,color=ffqqqq] (5.,3.)-- (5.8,2.2);
			\draw [line width=1.pt,color=ffqqqq] (5.8,2.2)-- (5.8,2.15);
			\draw [line width=1.pt,color=ffqqqq] (5.8,2.15)-- (5.,2.15);
			\draw [line width=1.pt,color=ffqqqq] (5.,2.15)-- (5.,3.);
			\draw [line width=1.pt,color=ffqqqq] (5.,1.85)-- (5.,1.);
			\draw [line width=1.pt,color=ffqqqq] (5.,1.)-- (5.8,1.8);
			\draw [line width=1.pt,color=ffqqqq] (5.8,1.8)-- (5.8,1.85);
			\draw [line width=1.pt,color=ffqqqq] (5.8,1.85)-- (5.,1.85);
			\draw [line width=1.pt,color=qqwuqq] (3.7012410443889086,3.7012410443889086)-- (3.7012410443889094,0.2987589556110911);
			\draw [line width=1.pt,color=qqzzqq] (6.5,1.5)-- (6.5,2.5);
			\draw [->,line width=1.2pt] (4.2,2.8) -- (4.8,2.8);
			\draw [->,line width=1.2pt] (4.8,1.2) -- (4.2,1.2);
			\begin{scriptsize}
			\draw[color=black] (4.5,2.6) node {$\tilde{G}_{K,l_1,l_2,k}$};
			\draw[color=black] (4.5,1.4) node {${G}_{K,l_1,l_2,k}$};
			\end{scriptsize}
		\end{tikzpicture}
		\caption{A scheme of the maps $\tilde{G}_{K,l_1,l_2,k}$ and ${G}_{K,l_1,l_2,k}$. Hyperplane parts get mapped to hyperplane parts as in the dark green line. Injectivity on the red parts mapped to the red parts is thanks to the injectivity of the frame-to-frame maps. Injectivity on the blue parts can be seen easily from the convexity of the cube. The fact that the green parts have image disjoint from the blue/red part was proved in \cite[Section 3]{CHT}.  Continuity on the hyperplane parts is thanks to continuity of the frame-to-frame mappings. Continuity of the maps at the boundary of the green and red/blue parts was proved in \cite[Section 3]{CHT}. Red lines parallel to $e_1$ are mapped to red lines parallel to $e_1$.}\label{Fig:Places}
	\end{figure}

	The injectivity of ${G}_{K,l_1,l_2,k}$ is easily proved. In fact it suffices to check injectivity on a single frame and injectivity follows on all other frames by self similarity of the construction. The situation on a given frame is illustrated in Figure~\ref{Fig:Places}. We have 3 different considerations, which are represented in Figure~\ref{Fig:Places} by the blue, the red and the green parts. In $S_{\ve, i}$, the part of the frame farthest from $z_{\ve}$ in the $e_i$ direction, we send hyperplane parts perpendicular to $e_i$ onto hyperplane parts perpendicular to $e_i$. Distinct hyperplane parts are mapped onto distinct hyperplane parts because $t_K$ is an increasing function.
	On the part of the hyperplane inside $U^3_{l_1}$ (blue in Figure~\ref{Fig:Places}), injectivity on each hyperplane is easy to see from the convexity of the cube. Injectivity on the hyperplane outside $U^3_{l_1}$ (red in Figure~\ref{Fig:Places}) follows immediately from the fact that in this part
	$$
		H^{3,i}_K(x)+t_K(x_i)e_i = [H^{3,i}_K]_{l_1}(x)+t_K(x_i)e_i =[H^{3,i}_K]_{l_2}(x)+t_K(x_i)e_i.
	$$
	
	In the case considered in \cite{CHT}, the fact was proved that $S_{i,\w}\subset Q_{\w}' \setminus Q_{\w}$ has disjoint image from $S_{j,\w}$ for $j\neq i$. Continuity of the map at the common boundary of the two sets was also proved. The entire argument applies here with the only difference being slightly changed indices, i.e. now we have general $K\in\en$ and not only $K=0$ and we have $l_2>l_1\geq k$ instead of $3k+k>3k\geq k$. Arguments are however the same 
	and we refer the curious reader to \cite[Lemma~3.2]{CHT} especially Step~1 of the proof.
	
	Now we fix an important parameter $N\in\en$. In fact at the end we put $N=2K$ but we prefer to use a different notation for it so it is easy to track. Its role is that we squeeze more in some directions and hence in some key estimates we obtain the additional factor $2^{-2\beta N}$ (see e.g. \eqref{SQ1B} and \eqref{SQ3D1} below). 
Its geometric role is that in Theorem	\ref{TheBigLebowski} $v)$ we  need that $\int_{Q(0,1)\setminus E} |D_1f_1|^p$ is really small, so especially for $p=1$ the length of $f_1(L\cap (Q(0,1)\setminus E))$ is small for every line segment $L$ parallel to $e_1$ that intersects $E$. To achieve that we construct $G_{K,N,m,\eta}$ so that $G_{K,N,m,\eta}(L)$ is close to being a segment parallel to $e_1$ axis, especially it is not prolonged too much. This is achieved exactly by introducing $N$ in the next paragraph. At the end we build our $f_1$ from $G_{K,N,m,\eta}\circ \hat{F}_{\beta, m}\circ \tilde{G}_{K,N,m,\eta}$ and we need also that $\hat{F}_{\beta, m}$ and $\tilde{G}_{K,N,m,\eta}$ map things close to segment parallel to $e_1$ to something close to a segment parallel to $e_1$. Therefore we use $N$ also in the next subsection to define $\tilde{G}_{K,N,m,\eta}$. 
	
	We set $l_1=3k+N$ and $l_2=3k+3+N$. We have that  $G_{K,3k+N,3k+3+N,k}$ is a homeomorphism  on the whole frame $Q_{\w}' \setminus Q_{\w}$, $\w \in \V^{k+1}_4$. The union of these frames is the set ${U}_{k}^4\setminus {U}_{k+1}^4$.

	\begin{figure}
		\centering
			\definecolor{ubqqys}{rgb}{0.29411764705882354,0.,0.5098039215686274}
			\definecolor{qqffqq}{rgb}{0.,1.,0.}
			\definecolor{qqqqff}{rgb}{0.,0.,1.}
			\begin{tikzpicture}[line cap=round,line join=round,>=triangle 45,x=0.5cm,y=0.5cm]
			\clip(-4.731744203146705,-10.58514825572656) rectangle (21.040947818145312,5.602201716916772);
			\fill[line width=0.7pt,color=qqqqff,fill=qqqqff,fill opacity=0.30000001192092896] (-4.,-0.47) -- (0.,-0.47) -- (0.,-0.53) -- (-4.,-0.53) -- cycle;
			\fill[line width=0.7pt,color=qqffqq,fill=qqffqq,fill opacity=0.10000000149011612] (-4.,0.47) -- (-4.,-0.47) -- (0.,-0.47) -- (0.,0.47) -- cycle;
			\fill[line width=0.7pt,color=qqffqq,fill=qqffqq,fill opacity=0.10000000149011612] (0.,1.) -- (-4.,1.) -- (-4.,0.53) -- (0.,0.53) -- cycle;
			\fill[line width=0.7pt,color=qqffqq,fill=qqffqq,fill opacity=0.10000000149011612] (0.,-1.) -- (-4.,-1.) -- (-4.,-0.53) -- (0.,-0.53) -- cycle;
			\fill[line width=0.7pt,color=ubqqys,fill=ubqqys,fill opacity=0.10000000149011612] (-4.,-1.) -- (20.,-1.) -- (20.,-5.) -- (-4.,-5.) -- cycle;
			\fill[line width=0.7pt,color=qqqqff,fill=qqqqff,fill opacity=0.30000001192092896] (0.47,5.) -- (0.47,1.) -- (0.53,1.) -- (0.53,5.) -- cycle;
			\fill[line width=0.7pt,color=qqqqff,fill=qqqqff,fill opacity=0.30000001192092896] (1.47,5.) -- (1.47,1.) -- (1.53,1.) -- (1.53,5.) -- cycle;
			\fill[line width=0.7pt,color=qqffqq,fill=qqffqq,fill opacity=0.10000000149011612] (1.47,5.) -- (0.53,5.) -- (0.53,1.) -- (1.47,1.) -- cycle;
			\fill[line width=0.7pt,color=ubqqys,fill=ubqqys,fill opacity=0.30000001192092896] (-4.,-10.) -- (20.,-10.) -- (20.,-10.2) -- (-4.,-10.2) -- cycle;
			\fill[line width=0.7pt,color=qqqqff,fill=qqqqff,fill opacity=0.30000001192092896] (-4.,-8.47) -- (-1.,-8.47) -- (0.,-8.1) -- (0.,-8.9) -- (-1.,-8.53) -- (-4.,-8.53) -- cycle;
			\fill[line width=0.7pt,color=qqqqff,fill=qqqqff,fill opacity=0.30000001192092896] (0.,-9.1) -- (-1.,-9.47) -- (-4.,-9.47) -- (-4.,-9.53) -- (-1.,-9.53) -- (0.,-9.9) -- cycle;
			\fill[line width=0.7pt,color=qqffqq,fill=qqffqq,fill opacity=0.10000000149011612] (0.,-8.9) -- (-1.,-8.53) -- (-4.,-8.53) -- (-4.,-9.47) -- (-1.,-9.47) -- (0.,-9.1) -- cycle;
			\fill[line width=0.7pt,color=qqffqq,fill=qqffqq,fill opacity=0.10000000149011612] (0.,-9.9) -- (0.,-10.) -- (-4.,-10.) -- (-4.,-9.53) -- (-1.,-9.53) -- cycle;
			\fill[line width=0.7pt,color=qqffqq,fill=qqffqq,fill opacity=0.10000000149011612] (0.,-8.) -- (-4.,-8.) -- (-4.,-8.47) -- (-1.,-8.47) -- (0.,-8.1) -- cycle;
			\fill[line width=0.7pt,color=qqffqq,fill=qqffqq,fill opacity=0.10000000149011612] (0.,-7.8) -- (0.,-8.) -- (0.1,-8.) -- (0.1,-7.99) -- (0.47,-7.98) -- (0.47,-7.8) -- cycle;
			\fill[line width=0.7pt,color=qqqqff,fill=qqqqff,fill opacity=0.30000001192092896] (0.47,-7.8) -- (0.53,-7.8) -- (0.53,-7.98) -- (0.9,-7.99) -- (0.9,-8.) -- (0.1,-8.) -- (0.1,-7.99) -- (0.47,-7.98) -- cycle;
			\fill[line width=0.7pt,color=qqqqff,fill=qqqqff,fill opacity=0.30000001192092896] (1.47,-7.8) -- (1.53,-7.8) -- (1.53,-7.98) -- (1.9,-7.99) -- (1.9,-8.) -- (1.1,-8.) -- (1.1,-7.99) -- (1.47,-7.98) -- cycle;
			\fill[line width=0.7pt,color=qqffqq,fill=qqffqq,fill opacity=0.10000000149011612] (0.53,-7.8) -- (1.47,-7.8) -- (1.47,-7.98) -- (1.1,-7.99) -- (1.1,-8.) -- (0.9,-8.) -- (0.9,-7.99) -- (0.53,-7.98) -- cycle;
			\fill[line width=0.7pt,color=qqffqq,fill=qqffqq,fill opacity=0.10000000149011612] (1.53,-7.8) -- (1.53,-7.98) -- (1.9,-7.99) -- (1.9,-8.) -- (2.,-8.) -- (2.,-7.8) -- cycle;
			\fill[line width=0.7pt,color=qqffqq,fill=qqffqq,fill opacity=0.10000000149011612] (0.,1.) -- (0.,5.) -- (0.47,5.) -- (0.47,1.) -- cycle;
			\fill[line width=0.7pt,color=qqffqq,fill=qqffqq,fill opacity=0.10000000149011612] (2.,1.) -- (2.,5.) -- (1.53,5.) -- (1.53,1.) -- cycle;
			\draw [line width=0.7pt] (-4.,-7.8)-- (-4.,-10.2);
			\draw [line width=0.7pt] (-4.,-10.2)-- (20.,-10.2);
			\draw [line width=0.7pt] (20.,-10.2)-- (20.,-7.8);
			\draw [line width=0.7pt] (20.,-7.8)-- (-4.,-7.8);
			\draw [line width=0.7pt] (0.,1.)-- (0.,-1.);
			\draw [line width=0.7pt] (0.,-1.)-- (16.,-1.);
			\draw [line width=0.7pt] (16.,-1.)-- (16.,1.);
			\draw [line width=0.7pt] (16.,1.)-- (0.,1.);
			\draw [line width=0.7pt] (2.,1.)-- (2.,-1.);
			\draw [line width=0.7pt] (4.,1.)-- (4.,-1.);
			\draw [line width=0.7pt] (6.,1.)-- (6.,-1.);
			\draw [line width=0.7pt] (8.,1.)-- (8.,-1.);
			\draw [line width=0.7pt] (10.,1.)-- (10.,-1.);
			\draw [line width=0.7pt] (12.,1.)-- (12.,-1.);
			\draw [line width=0.7pt] (14.,1.)-- (14.,-1.);
			\draw [line width=0.7pt] (-4.,5.)-- (20.,5.);
			\draw [line width=0.7pt] (20.,5.)-- (20.,-5.);
			\draw [line width=0.7pt] (20.,-5.)-- (-4.,-5.);
			\draw [line width=0.7pt] (-4.,-5.)-- (-4.,5.);
			\draw [line width=0.7pt] (0.,-8.)-- (0.,-10.);
			\draw [line width=0.7pt] (0.,-10.)-- (16.,-10.);
			\draw [line width=0.7pt] (16.,-10.)-- (16.,-8.);
			\draw [line width=0.7pt] (16.,-8.)-- (0.,-8.);
			\draw [line width=0.7pt] (2.,-8.)-- (2.,-10.001252097531708);
			\draw [line width=0.7pt] (4.,-8.)-- (4.,-10.);
			\draw [line width=0.7pt] (6.,-8.)-- (6.,-10.);
			\draw [line width=0.7pt] (8.,-8.)-- (8.,-10.);
			\draw [line width=0.7pt] (10.,-8.)-- (10.,-10.);
			\draw [line width=0.7pt] (12.,-8.)-- (12.,-10.);
			\draw [line width=0.7pt] (14.,-8.)-- (14.,-10.00058127117801);
			\draw [line width=0.7pt,color=qqqqff] (-4.,0.53)-- (0.,0.53);
			\draw [line width=0.7pt,color=qqqqff] (0.,0.47)-- (-4.,0.47);
			\draw [line width=0.7pt,color=qqqqff] (-4.,-0.47)-- (0.,-0.47);
			\draw [line width=0.7pt,color=qqqqff] (0.,-0.53)-- (-4.,-0.53);
			\draw [line width=0.7pt,color=qqqqff] (0.47,5.)-- (0.47,1.);
			\draw [line width=0.7pt,color=qqqqff] (0.53,1.)-- (0.53,5.);
			\draw [line width=0.7pt,color=qqqqff] (1.47,5.)-- (1.47,1.);
			\draw [line width=0.7pt,color=qqqqff] (1.53,1.)-- (1.53,5.);
			\draw [line width=0.7pt,color=qqqqff] (-4.,-8.47)-- (-1.,-8.47);
			\draw [line width=0.7pt,color=qqqqff] (-1.,-8.47)-- (0.,-8.1);
			\draw [line width=0.7pt,color=qqqqff] (0.,-8.9)-- (-1.,-8.53);
			\draw [line width=0.7pt,color=qqqqff] (-1.,-8.53)-- (-4.,-8.53);
			\draw [line width=0.7pt,color=qqqqff] (0.,-9.1)-- (-1.,-9.47);
			\draw [line width=0.7pt,color=qqqqff] (-1.,-9.47)-- (-4.,-9.47);
			\draw [line width=0.7pt,color=qqqqff] (-4.,-9.53)-- (-1.,-9.53);
			\draw [line width=0.7pt,color=qqqqff] (-1.,-9.53)-- (0.,-9.9);
			\draw [line width=0.7pt,color=qqqqff] (0.53,-7.8)-- (0.53,-7.98);
			\draw [line width=0.7pt,color=qqqqff] (0.53,-7.98)-- (0.9,-7.99);
			\draw [line width=0.7pt,color=qqqqff] (0.9,-7.99)-- (0.9,-8.);
			\draw [line width=0.7pt,color=qqqqff] (0.1,-8.)-- (0.1,-7.99);
			\draw [line width=0.7pt,color=qqqqff] (0.1,-7.99)-- (0.47,-7.98);
			\draw [line width=0.7pt,color=qqqqff] (0.47,-7.98)-- (0.47,-7.8);
			\draw [line width=0.7pt,color=qqqqff] (1.53,-7.8)-- (1.53,-7.98);
			\draw [line width=0.7pt,color=qqqqff] (1.53,-7.98)-- (1.9,-7.99);
			\draw [line width=0.7pt,color=qqqqff] (1.1,-7.99)-- (1.47,-7.98);
			\draw [line width=0.7pt,color=qqqqff] (1.47,-7.98)-- (1.47,-7.8);
			\draw [line width=0.7pt,color=qqqqff] (1.1,-8.)-- (1.1,-7.99);
			\draw [line width=0.7pt,color=qqqqff] (1.9,-8.)-- (1.9,-7.99);
			\draw [->,line width=0.7pt] (8.,-5.5) -- (8.,-7.5);
			\begin{scriptsize}
			\draw[color=black] (9.5,-6.5) node {$G_{K,N,m,\eta}$};
			\end{scriptsize}
			\end{tikzpicture}

			\begin{tikzpicture}[line cap=round,line join=round,>=triangle 45,x=6cm,y=6cm]
			\clip(0,-8.05) rectangle (2,-7.6);
			\fill[line width=0.7pt,color=qqqqff,fill=qqqqff,fill opacity=0.30000001192092896] (-4.,-0.47) -- (0.,-0.47) -- (0.,-0.53) -- (-4.,-0.53) -- cycle;
			\fill[line width=0.7pt,color=qqffqq,fill=qqffqq,fill opacity=0.10000000149011612] (-4.,0.47) -- (-4.,-0.47) -- (0.,-0.47) -- (0.,0.47) -- cycle;
			\fill[line width=0.7pt,color=qqffqq,fill=qqffqq,fill opacity=0.10000000149011612] (0.,1.) -- (-4.,1.) -- (-4.,0.53) -- (0.,0.53) -- cycle;
			\fill[line width=0.7pt,color=qqffqq,fill=qqffqq,fill opacity=0.10000000149011612] (0.,-1.) -- (-4.,-1.) -- (-4.,-0.53) -- (0.,-0.53) -- cycle;
			\fill[line width=0.7pt,color=ubqqys,fill=ubqqys,fill opacity=0.10000000149011612] (-4.,-1.) -- (20.,-1.) -- (20.,-5.) -- (-4.,-5.) -- cycle;
			\fill[line width=0.7pt,color=qqqqff,fill=qqqqff,fill opacity=0.30000001192092896] (0.47,5.) -- (0.47,1.) -- (0.53,1.) -- (0.53,5.) -- cycle;
			\fill[line width=0.7pt,color=qqqqff,fill=qqqqff,fill opacity=0.30000001192092896] (1.47,5.) -- (1.47,1.) -- (1.53,1.) -- (1.53,5.) -- cycle;
			\fill[line width=0.7pt,color=qqffqq,fill=qqffqq,fill opacity=0.10000000149011612] (1.47,5.) -- (0.53,5.) -- (0.53,1.) -- (1.47,1.) -- cycle;
			\fill[line width=0.7pt,color=ubqqys,fill=ubqqys,fill opacity=0.30000001192092896] (-4.,-10.) -- (20.,-10.) -- (20.,-10.2) -- (-4.,-10.2) -- cycle;
			\fill[line width=0.7pt,color=qqqqff,fill=qqqqff,fill opacity=0.30000001192092896] (-4.,-8.47) -- (-1.,-8.47) -- (0.,-8.1) -- (0.,-8.9) -- (-1.,-8.53) -- (-4.,-8.53) -- cycle;
			\fill[line width=0.7pt,color=qqqqff,fill=qqqqff,fill opacity=0.30000001192092896] (0.,-9.1) -- (-1.,-9.47) -- (-4.,-9.47) -- (-4.,-9.53) -- (-1.,-9.53) -- (0.,-9.9) -- cycle;
			\fill[line width=0.7pt,color=qqffqq,fill=qqffqq,fill opacity=0.10000000149011612] (0.,-8.9) -- (-1.,-8.53) -- (-4.,-8.53) -- (-4.,-9.47) -- (-1.,-9.47) -- (0.,-9.1) -- cycle;
			\fill[line width=0.7pt,color=qqffqq,fill=qqffqq,fill opacity=0.10000000149011612] (0.,-9.9) -- (0.,-10.) -- (-4.,-10.) -- (-4.,-9.53) -- (-1.,-9.53) -- cycle;
			\fill[line width=0.7pt,color=qqffqq,fill=qqffqq,fill opacity=0.10000000149011612] (0.,-8.) -- (-4.,-8.) -- (-4.,-8.47) -- (-1.,-8.47) -- (0.,-8.1) -- cycle;
			\fill[line width=0.7pt,color=qqffqq,fill=qqffqq,fill opacity=0.10000000149011612] (0.,-7.8) -- (0.,-8.) -- (0.1,-8.) -- (0.1,-7.99) -- (0.47,-7.98) -- (0.47,-7.8) -- cycle;
			\fill[line width=0.7pt,color=qqqqff,fill=qqqqff,fill opacity=0.30000001192092896] (0.47,-7.8) -- (0.53,-7.8) -- (0.53,-7.98) -- (0.9,-7.99) -- (0.9,-8.) -- (0.1,-8.) -- (0.1,-7.99) -- (0.47,-7.98) -- cycle;
			\fill[line width=0.7pt,color=qqqqff,fill=qqqqff,fill opacity=0.30000001192092896] (1.47,-7.8) -- (1.53,-7.8) -- (1.53,-7.98) -- (1.9,-7.99) -- (1.9,-8.) -- (1.1,-8.) -- (1.1,-7.99) -- (1.47,-7.98) -- cycle;
			\fill[line width=0.7pt,color=qqffqq,fill=qqffqq,fill opacity=0.10000000149011612] (0.53,-7.8) -- (1.47,-7.8) -- (1.47,-7.98) -- (1.1,-7.99) -- (1.1,-8.) -- (0.9,-8.) -- (0.9,-7.99) -- (0.53,-7.98) -- cycle;
			\fill[line width=0.7pt,color=qqffqq,fill=qqffqq,fill opacity=0.10000000149011612] (1.53,-7.8) -- (1.53,-7.98) -- (1.9,-7.99) -- (1.9,-8.) -- (2.,-8.) -- (2.,-7.8) -- cycle;
			\fill[line width=0.7pt,color=qqffqq,fill=qqffqq,fill opacity=0.10000000149011612] (0.,1.) -- (0.,5.) -- (0.47,5.) -- (0.47,1.) -- cycle;
			\fill[line width=0.7pt,color=qqffqq,fill=qqffqq,fill opacity=0.10000000149011612] (2.,1.) -- (2.,5.) -- (1.53,5.) -- (1.53,1.) -- cycle;
			\draw [line width=0.7pt] (-4.,-7.8)-- (-4.,-10.2);
			\draw [line width=0.7pt] (-4.,-10.2)-- (20.,-10.2);
			\draw [line width=0.7pt] (20.,-10.2)-- (20.,-7.8);
			\draw [line width=0.7pt] (20.,-7.8)-- (-4.,-7.8);
			\draw [line width=0.7pt] (0.,1.)-- (0.,-1.);
			\draw [line width=0.7pt] (0.,-1.)-- (16.,-1.);
			\draw [line width=0.7pt] (16.,-1.)-- (16.,1.);
			\draw [line width=0.7pt] (16.,1.)-- (0.,1.);
			\draw [line width=0.7pt] (2.,1.)-- (2.,-1.);
			\draw [line width=0.7pt] (4.,1.)-- (4.,-1.);
			\draw [line width=0.7pt] (6.,1.)-- (6.,-1.);
			\draw [line width=0.7pt] (8.,1.)-- (8.,-1.);
			\draw [line width=0.7pt] (10.,1.)-- (10.,-1.);
			\draw [line width=0.7pt] (12.,1.)-- (12.,-1.);
			\draw [line width=0.7pt] (14.,1.)-- (14.,-1.);
			\draw [line width=0.7pt] (-4.,5.)-- (20.,5.);
			\draw [line width=0.7pt] (20.,5.)-- (20.,-5.);
			\draw [line width=0.7pt] (20.,-5.)-- (-4.,-5.);
			\draw [line width=0.7pt] (-4.,-5.)-- (-4.,5.);
			\draw [line width=0.7pt] (0.,-8.)-- (0.,-10.);
			\draw [line width=0.7pt] (0.,-10.)-- (16.,-10.);
			\draw [line width=0.7pt] (16.,-10.)-- (16.,-8.);
			\draw [line width=0.7pt] (16.,-8.)-- (0.,-8.);
			\draw [line width=0.7pt] (2.,-8.)-- (2.,-10.001252097531708);
			\draw [line width=0.7pt] (4.,-8.)-- (4.,-10.);
			\draw [line width=0.7pt] (6.,-8.)-- (6.,-10.);
			\draw [line width=0.7pt] (8.,-8.)-- (8.,-10.);
			\draw [line width=0.7pt] (10.,-8.)-- (10.,-10.);
			\draw [line width=0.7pt] (12.,-8.)-- (12.,-10.);
			\draw [line width=0.7pt] (14.,-8.)-- (14.,-10.00058127117801);
			\draw [line width=0.7pt,color=qqqqff] (-4.,0.53)-- (0.,0.53);
			\draw [line width=0.7pt,color=qqqqff] (0.,0.47)-- (-4.,0.47);
			\draw [line width=0.7pt,color=qqqqff] (-4.,-0.47)-- (0.,-0.47);
			\draw [line width=0.7pt,color=qqqqff] (0.,-0.53)-- (-4.,-0.53);
			\draw [line width=0.7pt,color=qqqqff] (0.47,5.)-- (0.47,1.);
			\draw [line width=0.7pt,color=qqqqff] (0.53,1.)-- (0.53,5.);
			\draw [line width=0.7pt,color=qqqqff] (1.47,5.)-- (1.47,1.);
			\draw [line width=0.7pt,color=qqqqff] (1.53,1.)-- (1.53,5.);
			\draw [line width=0.7pt,color=qqqqff] (-4.,-8.47)-- (-1.,-8.47);
			\draw [line width=0.7pt,color=qqqqff] (-1.,-8.47)-- (0.,-8.1);
			\draw [line width=0.7pt,color=qqqqff] (0.,-8.9)-- (-1.,-8.53);
			\draw [line width=0.7pt,color=qqqqff] (-1.,-8.53)-- (-4.,-8.53);
			\draw [line width=0.7pt,color=qqqqff] (0.,-9.1)-- (-1.,-9.47);
			\draw [line width=0.7pt,color=qqqqff] (-1.,-9.47)-- (-4.,-9.47);
			\draw [line width=0.7pt,color=qqqqff] (-4.,-9.53)-- (-1.,-9.53);
			\draw [line width=0.7pt,color=qqqqff] (-1.,-9.53)-- (0.,-9.9);
			\draw [line width=0.7pt,color=qqqqff] (0.53,-7.8)-- (0.53,-7.98);
			\draw [line width=0.7pt,color=qqqqff] (0.53,-7.98)-- (0.9,-7.99);
			\draw [line width=0.7pt,color=qqqqff] (0.9,-7.99)-- (0.9,-8.);
			\draw [line width=0.7pt,color=qqqqff] (0.1,-8.)-- (0.1,-7.99);
			\draw [line width=0.7pt,color=qqqqff] (1.1,-8.)-- (1.1,-7.99);
			\draw [line width=0.7pt,color=qqqqff] (1.9,-8.)-- (1.9,-7.99);
			\draw [line width=0.7pt,color=qqqqff] (0.1,-7.99)-- (0.47,-7.98);
			\draw [line width=0.7pt,color=qqqqff] (0.47,-7.98)-- (0.47,-7.8);
			\draw [line width=0.7pt,color=qqqqff] (1.53,-7.8)-- (1.53,-7.98);
			\draw [line width=0.7pt,color=qqqqff] (1.53,-7.98)-- (1.9,-7.99);
			\draw [line width=0.7pt,color=qqqqff] (1.1,-7.99)-- (1.47,-7.98);
			\draw [line width=0.7pt,color=qqqqff] (1.47,-7.98)-- (1.47,-7.8);
			\draw [->,line width=0.7pt] (8.,-5.5) -- (8.,-7.5);
			\begin{scriptsize}
			\draw[color=black] (9.5,-6.5) node {$G_{K,N,m,\eta}$};
			\end{scriptsize}
			\end{tikzpicture}
			
			A detail of the image near the boundary at the top of $R_{m,\eta}$.
		\caption{A scheme of the definition of $G_{K,N,m,\eta}$. The grey shaded area is where we use the mapping $G_{K,3k+N, 3k+3+N, k}$. The top and bottom (purple) parts are squeezed from size $13$ to size $\eta$. In the preimage the blue area represents $[-2m-5,-2m-1]\times {U}_{N}^3$, and $({U}_{N}^3+n)\times [1,14]$ respectively.}\label{Fig:Stretching}
	\end{figure}

	Now we define the map that this subsection is dedicated to; its scheme is in Figure~\ref{Fig:Stretching}. Let $\eta >0$ and call $n = (n_1, n_2, n_3, n_4)$, where $n_1,n_2,n_3$ are even numbers between $-2m$ and $+2m$ and $n_4 = 0$. We define $G_{K,N,m,\eta}$ on $Z_m$ as follows; when $y = x+n$ for $x\in Q(0,1)$,
	\eqn{ASeason0}
	$$
			G_{K,N,m,\eta}(y) =G_{K,N,m,\eta}(x + n)
			= \begin{cases}
				G_{K,3k+N,3k+3+N,k}(x) + n \ & x \in {U}_{k}^4 \setminus {U}_{k+1}^4,\\
				(t_K(x_1),t_K(x_2),t_K(x_3),t_K(x_4)) + n \ &  x\in \C_{B}.\\
		\end{cases}
	$$
	
	On $R_{m,13}\setminus Z_m =\bigl([-2m-5,2m+5]^3\times[-14,14]\bigr)\setminus \bigl([-2m-1,2m+1]^3\times[-1,1]\bigr)$ we define our $G_{K,N,m,\eta}$ as an interpolation between values on $\partial [-2m-1,2m+1]^3\times[-1,1]$ (which is some form of $[H^{3,i}_K]_N$) and between identity on outer boundary. Moreover, we need to get that $G$ is stretching (and thus $DG$ is big) only on $\bigl([-2m-2,2m+2]^3\times[-2,2]\bigr)\setminus \bigl([-2m-1,2m+1]^3\times[-1,1]\bigr)$ and it is Lipschitz outside of this region (see Fig. \ref{Fig:Stretching}). 
	
	When $y = x+n$, $x_1,x_2,x_3\in [-1, 1]$, and $x_4 \in [-14,-1]\cup[1,14]$ we put
	\begin{equation}\label{StretchingBitAroundTop}
		\begin{aligned}
			G_{K,N,m,\eta}(y) &=G_{K,N,m,\eta}(x+n)\\
			& = \begin{cases}
				[H^{3,4}_K]_{N}(x) +n + \frac{\eta}{13}(x_4-1)e_4 +e_4 \ & x_4 \in [1,\tfrac{3}{2}],\\
				2(2-x_4)[H^{3,4}_K]_{N}(x)+ 2(x_4-\frac{3}{2})(x_1,x_2,x_3,0)+\ & \\
				\quad +n + \frac{\eta}{13}(x_4-1)e_4 +e_4 \ & x_4 \in [\tfrac{3}{2},2],\\
				[H^{3,4}_K]_{N}(x)+ n + \frac{\eta}{13}(x_4+1)e_4 -e_4 \ & x_4 \in [-\tfrac{3}{2},-1],\\
				2(2+x_4)[H^{3,4}_K]_{N}(x)+2(-x_4-\frac{3}{2})(x_1,x_2,x_3,0)+ \ & \\
				\quad+ n + \frac{\eta}{13}(x_4+1)e_4 -e_4 \ & x_4 \in [-2,-\tfrac{3}{2}],\\
				(x_1,x_2,x_3,0)+n + \frac{\eta}{13}(x_4-1)e_4 +e_4 \ & x_4 \in [2,14],\\
				(x_1,x_2,x_3,0)+n + \frac{\eta}{13}(x_4+1)e_4 -e_4 \ & x_4 \in [-14,-2],\\
			\end{cases}
		\end{aligned}
	\end{equation}
	It is easy to see from \eqref{ASeason0}, \eqref{SpeciallyForStanda} and \eqref{StretchingBitAroundTop} that those mapping agree on the boundary $[-2m-1,2m+1]^3\times\{-1,1\}$, i.e. they are both equal to the shifted $[H^{3,4}_K]_{N}(x)$. 
	For $x_4\in [-2,-1]\cup[1,2]$ we  just interpolate between identity and $[H^{3,4}_K]_{N}(x)$ in the first three coordinates and in the last coordinate we squeeze the slab of thickness $1$ to slab of thickness $\frac{\eta}{13}$. 
	Finally for $x_4\in [-14,-2]\cup[2,14]$ we  have identity in the first three coordinates and in the last coordinate we squeeze the slab of thickness $12$ to slab of thickness $\frac{12}{13}\eta$.

On other parts of the boundary we do a similar interpolation between values on $\partial [-2m-1,2m+1]^3\times[-1,1]$ and the identity on the outer boundary:
	Assume that there is exactly one index $i=1,2,3$ such that $y_i\in [-2m-5,-2m-1]\cup[2m+1, 2m+5]$ and $y_j\in [-20-1,2m+1]$ for other $j\in\{1,2,3\}\setminus\{i\}$. 
Then we put $x_i = y_i$ and $n_i = 0$, for $j\in\{1,2,3\}\setminus\{i\}$ we find $x_j\in[-1,1]$ so that $y_j = x_j+n_j$ and $x_4 \in [-1, 1]$. We define 
	\begin{equation}\label{StretchingBitAroundTop2}
		\begin{aligned}
			G_{K,N,m,\eta}(y) & =G_{K,N,m,\eta}(x+n)\\
			& = \begin{cases}
				x + n  \ & x_i \in [2m+5, 2m+2],\\
				n +x_ie_i 
				+ (x_i-2m-1)(x-x_ie_i)\ &\\
				\quad + (2m+2 - x_i)[H^{3,i}_K]_{N}(x) \ & x_i \in [2m+1, 2m+2],\\
			  n +x_ie_i 
				+ (-x_i-2m-1)(x-x_ie_i)\ &\\
				\quad + (2m+2+x_i)[H^{3,i}_K]_{N}(x) \ & x_i \in  [-2m-2,-2m-1],\\
				x+n  \ & x_i \in [-2m-5, -2m-2].\\
			\end{cases}
		\end{aligned}
	\end{equation}
	
In the remaining part around the edges of our block $[-2m-5,2m+5]^3\times [-14,14]$ we just do a simple linear squeezing. That is for $y\in [-2m-5,2m+5]^3\times [-14,14]$ so that $y\notin [-2m-1,2m+1]^3\times[-14,14]$ and $y\notin [2m-5,2m+5]\times[-2m-1,2m+1]^2\times [-1,1]$
we define	
	\begin{equation}\label{StretchingBitAroundSides}
		G_{K,N,m,\eta}^j(y) =y_j \text{ for } j=1,2,3
	\end{equation}
	and in the last coordinate we linearly stretch
	\begin{equation}\label{StretchingBitAroundSides2}
		G_{K,N,m,\eta}^4(y) = \begin{cases}
			y_4& y_4 \in [-1,1]\\
			1 + \frac{\eta}{13}(y_4-1) \ & y_4 \in [1,14],\\
			-1 + \frac{\eta}{13}(y_4+1) \ & y_4 \in [-14,-1].\\	
		\end{cases}
	\end{equation}
Similarly we define it for the permutation of the first three coordinates, i.e. $y\notin [-2m-1,2m+1]^3\times[-14,14]$ and ($y\notin [-2m-1,2m+1]\times[2m-5,2m+5]\times[-2m-1,2m+1]\times [-1,1]$ or $y\notin [-2m-1,2m+1]^2\times [2m-5,2m+5]\times [-1,1]$). 
	

\subsection{A specific homeomorphism that maps $\C_{A,K}(4)$ onto $\C_{B}(4)$ dependent on a parameter $N$.}\label{firstmap}

	Now we define a map that squeezes the Cantor set; its scheme is in Figure~\ref{Fig:Squeeze}. Let $\eta >0$. The map this section is dedicated to is called $\tilde{G}_{K,N,m,\eta}$ (of course it depends on $\alpha$ and $\beta$) that maps $\C_{A,K}$ onto $\C_{B}$. In order to do so we use mappings of type $[J_K^3]_{m}$ on $3$-dimensional hyperplanes perpendicular to $e_i$, the $i$-th canonical vector. Again we use the maps $T_i$ and $T^i$ defined in Section~\ref{secondmap}. We call
	\begin{equation}
	J^{3,i}_K(x) =T^i(J^3_K(T_i(x))).
	\end{equation}
	Of course we also define
\begin{equation}\label{defT3}
	[J_K^{3,i}]_{m}(x) =T^i([J_K^{3}]_{m}(T_i(x))).
	\end{equation}
	In order to define $\tilde{G}_{K,N,m,\eta}$ we define the continuous function
	$$
		\tilde{\zeta}_{K,k}(s) =\begin{cases}
			1  & s\in \tilde{U}_{K,k+1}^1\\
			0  & s\in \er \setminus \tilde{U}_{K,k}^1\\
			\text{linear on intervals in} & \tilde{U}_{K,k}^1\setminus \tilde{U}_{K,k+1}^1.\\
		\end{cases}
	$$
	\begin{figure}
		\centering
		\begin{tikzpicture}[line cap=round,line join=round,>=triangle 45,x=0.5cm,y=0.5cm]
			\clip(-4.5,-14.5) rectangle (20.5,1.7);
			\fill[line width=0.pt,fill=black,fill opacity=0.3] (0.,1.) -- (0.,-1.) -- (16.,-1.) -- (16.,1.) -- cycle;
			\fill[line width=0.pt,fill=black,fill opacity=0.3] (0.,-8.) -- (0.,-10.) -- (16.,-10.) -- (16.,-8.) -- cycle;
			\fill[line width=0.pt,color=ffqqqq,fill=ffqqqq,fill opacity=0.5] (-4.,-1.) -- (-4.,-1.1) -- (20.,-1.1) -- (20.,-1.) -- cycle;
			\fill[line width=0.pt,color=qqqqff,fill=qqqqff,fill opacity=0.5] (-4.,0.95) -- (0.,0.95) -- (0.,0.55) -- (-4.,0.55) -- cycle;
			\fill[line width=0.pt,color=qqqqff,fill=qqqqff,fill opacity=0.5] (-4.,0.45) -- (-4.,0.05) -- (0.,0.05) -- (0.,0.45) -- cycle;
			\fill[line width=0.pt,color=qqqqff,fill=qqqqff,fill opacity=0.5] (-4.,-0.05) -- (-4.,-0.45) -- (0.,-0.45) -- (0.,-0.05) -- cycle;
			\fill[line width=0.pt,color=qqqqff,fill=qqqqff,fill opacity=0.5] (-4.,-0.55) -- (-4.,-0.95) -- (0.,-0.95) -- (0.,-0.55) -- cycle;
			\fill[line width=0.pt,color=qqffqq,fill=qqffqq,fill opacity=0.5] (-4.,0.55) -- (0.,0.55) -- (0.,0.45) -- (-4.,0.45) -- cycle;
			\fill[line width=0.pt,color=qqffqq,fill=qqffqq,fill opacity=0.5] (-4.,0.05) -- (-4.,-0.05) -- (0.,-0.05) -- (0.,0.05) -- cycle;
			\fill[line width=0.pt,color=qqffqq,fill=qqffqq,fill opacity=0.5] (-4.,-0.45) -- (-4.,-0.55) -- (0.,-0.55) -- (0.,-0.45) -- cycle;
			\fill[line width=0.pt,color=qqffqq,fill=qqffqq,fill opacity=0.5] (-4.,-0.95) -- (-4.,-1.) -- (0.,-1.) -- (0.,-0.95) -- cycle;
			\fill[line width=0.pt,color=ffqqqq,fill=ffqqqq,fill opacity=0.5] (-4.,-10.) -- (20.,-10.) -- (20.,-13.9) -- (-4.,-13.9) -- cycle;
			\fill[line width=0.pt,color=qqffqq,fill=qqffqq,fill opacity=0.5] (-4.,1.) -- (0.,1.) -- (0.,0.95) -- (-4.,0.95) -- cycle;
			\fill[line width=0.pt,color=qqqqff,fill=qqqqff,fill opacity=0.5] (-4.,-8.05) -- (-3.,-8.45) -- (0.,-8.45) -- (0.,-8.55) -- (-3.,-8.55) -- (-4.,-8.95) -- cycle;
			\fill[line width=0.pt,color=qqqqff,fill=qqqqff,fill opacity=0.5] (-4.,-9.05) -- (-3.,-9.45) -- (0.,-9.45) -- (0.,-9.55) -- (-3.,-9.55) -- (-4.,-9.95) -- cycle;
			\fill[line width=0.pt,color=ffqqqq,fill=ffqqqq,fill opacity=0.5] (-4.,-4.1) -- (-4.,-8.) -- (20.,-8.) -- (20.,-4.1) -- cycle;
			\fill[line width=0.pt,color=ffqqqq,fill=ffqqqq,fill opacity=0.5] (-4.,1.1) -- (-4.,1.) -- (20.,1.) -- (20.,1.1) -- cycle;
			\fill[line width=0.pt,color=qqffqq,fill=qqffqq,fill opacity=0.5] (-4.,-8.) -- (-4.,-8.05) -- (-3.,-8.45) -- (0.,-8.45) -- (0.,-8.) -- cycle;
			\fill[line width=0.pt,color=qqffqq,fill=qqffqq,fill opacity=0.5] (-3.,-8.55) -- (-4.,-8.95) -- (-4.,-9.05) -- (-3.,-9.45) -- (0.,-9.45) -- (0.,-8.55) -- cycle;
			\fill[line width=0.pt,color=qqffqq,fill=qqffqq,fill opacity=0.5] (-4.,-9.95) -- (-3.,-9.55) -- (0.,-9.55) -- (0.,-10.) -- (-4.,-10.) -- cycle;
			\fill[line width=0.pt,color=ffffff,fill=ffffff,fill opacity=1.0] (0.,1.2) -- (0.,1.) -- (2.,1.) -- (2.,1.2) -- cycle;
			\fill[line width=0.pt,color=qqffqq,fill=qqffqq,fill opacity=0.5] (0.,1.2) -- (0.,1.) -- (0.05,1.) -- (0.05,1.2) -- cycle;
			\fill[line width=0.pt,color=qqffqq,fill=qqffqq,fill opacity=0.5] (0.95,1.2) -- (0.95,1.) -- (1.05,1.) -- (1.05,1.2) -- cycle;
			\fill[line width=0.pt,color=qqffqq,fill=qqffqq,fill opacity=0.5] (1.95,1.2) -- (1.95,1.) -- (2.,1.) -- (2.,1.2) -- cycle;
			\fill[line width=0.pt,color=qqqqff,fill=qqqqff,fill opacity=0.5] (0.05,1.2) -- (0.95,1.2) -- (0.95,1.) -- (0.05,1.) -- cycle;
			\fill[line width=0.pt,color=qqqqff,fill=qqqqff,fill opacity=0.5] (1.05,1.2) -- (1.05,1.) -- (1.95,1.) -- (1.95,1.2) -- cycle;
			\fill[line width=0.pt,color=ffffff,fill=ffffff,fill opacity=1.0] (0.,-4.1) -- (0.,-8.) -- (2.,-8.) -- (2.,-4.1) -- cycle;
			\fill[line width=0.pt,color=qqffqq,fill=qqffqq,fill opacity=0.5] (0.,-4.) -- (0.,-8.) -- (0.47,-8.) -- (0.47,-4.1) -- (0.05,-4.) -- cycle;
			\fill[line width=0.pt,color=qqffqq,fill=qqffqq,fill opacity=0.5] (0.95,-4.) -- (0.53,-4.1) -- (0.53,-8.) -- (1.47,-8.) -- (1.47,-4.1) -- (1.05,-4.) -- cycle;
			\fill[line width=0.pt,color=qqffqq,fill=qqffqq,fill opacity=0.5] (1.53,-4.1) -- (1.53,-8.) -- (2.,-8.) -- (2.,-4.) -- (1.95,-4.) -- cycle;
			\fill[line width=0.pt,color=qqqqff,fill=qqqqff,fill opacity=0.5] (0.05,-4.) -- (0.95,-4.) -- (0.53,-4.1) -- (0.53,-8.) -- (0.47,-8.) -- (0.47,-4.1) -- cycle;
			\fill[line width=0.pt,color=qqqqff,fill=qqqqff,fill opacity=0.5] (1.05,-4.) -- (1.95,-4.) -- (1.53,-4.1) -- (1.53,-8.) -- (1.47,-8.) -- (1.47,-4.1) -- cycle;
			\draw [line width=0.7pt] (-4.,-4.)-- (-4.,-14.);
			\draw [line width=0.7pt] (-4.,-14.)-- (20.,-14.);
			\draw [line width=0.7pt] (20.,-14.)-- (20.,-4.);
			\draw [line width=0.7pt] (20.,-4.)-- (-4.,-4.);
			\draw [line width=0.7pt] (-4.,1.2)-- (20.,1.2);
			\draw [line width=0.7pt] (20.,1.2)-- (20.,-1.2);
			\draw [line width=0.7pt] (20.,-1.2)-- (-4.,-1.2);
			\draw [line width=0.7pt] (-4.,-1.2)-- (-4.,1.2);
			\draw [line width=0.7pt] (0.,1.)-- (0.,-1.);
			\draw [line width=0.7pt] (0.,-1.)-- (16.,-1.);
			\draw [line width=0.7pt] (16.,-1.)-- (16.,1.);
			\draw [line width=0.7pt] (16.,1.)-- (0.,1.);
			\draw [line width=0.7pt] (2.,1.)-- (2.,-1.);
			\draw [line width=0.7pt] (4.,1.)-- (4.,-1.);
			\draw [line width=0.7pt] (6.,1.)-- (6.,-1.);
			\draw [line width=0.7pt] (8.,1.)-- (8.,-1.);
			\draw [line width=0.7pt] (10.,1.)-- (10.,-1.);
			\draw [line width=0.7pt] (12.,1.)-- (12.,-1.);
			\draw [line width=0.7pt] (14.,1.)-- (14.,-1.);
			\draw [line width=0.7pt] (0.,-8.)-- (0.,-10.);
			\draw [line width=0.7pt] (0.,-10.)-- (16.,-10.);
			\draw [line width=0.7pt] (16.,-10.)-- (16.,-8.);
			\draw [line width=0.7pt] (16.,-8.)-- (0.,-8.);
			\draw [line width=0.7pt] (2.,-8.)-- (2.,-10.001252097531708);
			\draw [line width=0.7pt] (4.,-8.)-- (4.,-10.);
			\draw [line width=0.7pt] (6.,-8.)-- (6.,-10.);
			\draw [line width=0.7pt] (8.,-8.)-- (8.,-10.);
			\draw [line width=0.7pt] (10.,-8.)-- (10.,-10.);
			\draw [line width=0.7pt] (12.,-8.)-- (12.,-10.);
			\draw [line width=0.7pt] (14.,-8.)-- (14.,-10.00058127117801);
			\draw [->,line width=1.pt] (8.,-1.5) -- (8.,-3.5);
			\begin{scriptsize}
			\draw[color=black] (9.5,-2.26) node {$\tilde{G}_{K,N,m,\eta}$};
			\end{scriptsize}
		\end{tikzpicture}
		
		\begin{tikzpicture}[line cap=round,line join=round,>=triangle 45,x=3.0cm,y=3.0cm]
		\clip(-0.2,-4.3) rectangle (2.2,-3.9);
			\fill[line width=0.pt,color=qqqqff,fill=qqqqff,fill opacity=0.5] (-4.,-8.05) -- (-3.,-8.45) -- (0.,-8.45) -- (0.,-8.55) -- (-3.,-8.55) -- (-4.,-8.95) -- cycle;
			\fill[line width=0.pt,color=qqqqff,fill=qqqqff,fill opacity=0.5] (-4.,-9.05) -- (-3.,-9.45) -- (0.,-9.45) -- (0.,-9.55) -- (-3.,-9.55) -- (-4.,-9.95) -- cycle;
			\fill[line width=0.pt,color=ffqqqq,fill=ffqqqq,fill opacity=0.5] (-4.,-4.1) -- (-4.,-8.) -- (20.,-8.) -- (20.,-4.1) -- cycle;
			\fill[line width=0.pt,color=ffqqqq,fill=ffqqqq,fill opacity=0.5] (-4.,1.1) -- (-4.,1.) -- (20.,1.) -- (20.,1.1) -- cycle;
			\fill[line width=0.pt,color=qqffqq,fill=qqffqq,fill opacity=0.5] (-4.,-8.) -- (-4.,-8.05) -- (-3.,-8.45) -- (0.,-8.45) -- (0.,-8.) -- cycle;
			\fill[line width=0.pt,color=qqffqq,fill=qqffqq,fill opacity=0.5] (-3.,-8.55) -- (-4.,-8.95) -- (-4.,-9.05) -- (-3.,-9.45) -- (0.,-9.45) -- (0.,-8.55) -- cycle;
			\fill[line width=0.pt,color=qqffqq,fill=qqffqq,fill opacity=0.5] (-4.,-9.95) -- (-3.,-9.55) -- (0.,-9.55) -- (0.,-10.) -- (-4.,-10.) -- cycle;
			\fill[line width=0.pt,color=ffffff,fill=ffffff,fill opacity=1.0] (0.,1.2) -- (0.,1.) -- (2.,1.) -- (2.,1.2) -- cycle;
			\fill[line width=0.pt,color=qqffqq,fill=qqffqq,fill opacity=0.5] (0.,1.2) -- (0.,1.) -- (0.05,1.) -- (0.05,1.2) -- cycle;
			\fill[line width=0.pt,color=qqffqq,fill=qqffqq,fill opacity=0.5] (0.95,1.2) -- (0.95,1.) -- (1.05,1.) -- (1.05,1.2) -- cycle;
			\fill[line width=0.pt,color=qqffqq,fill=qqffqq,fill opacity=0.5] (1.95,1.2) -- (1.95,1.) -- (2.,1.) -- (2.,1.2) -- cycle;
			\fill[line width=0.pt,color=qqqqff,fill=qqqqff,fill opacity=0.5] (0.05,1.2) -- (0.95,1.2) -- (0.95,1.) -- (0.05,1.) -- cycle;
			\fill[line width=0.pt,color=qqqqff,fill=qqqqff,fill opacity=0.5] (1.05,1.2) -- (1.05,1.) -- (1.95,1.) -- (1.95,1.2) -- cycle;
			\fill[line width=0.pt,color=ffffff,fill=ffffff,fill opacity=1.0] (0.,-4.1) -- (0.,-8.) -- (2.,-8.) -- (2.,-4.1) -- cycle;
			\fill[line width=0.pt,color=qqffqq,fill=qqffqq,fill opacity=0.5] (0.,-4.) -- (0.,-8.) -- (0.47,-8.) -- (0.47,-4.1) -- (0.05,-4.) -- cycle;
			\fill[line width=0.pt,color=qqffqq,fill=qqffqq,fill opacity=0.5] (0.95,-4.) -- (0.53,-4.1) -- (0.53,-8.) -- (1.47,-8.) -- (1.47,-4.1) -- (1.05,-4.) -- cycle;
			\fill[line width=0.pt,color=qqffqq,fill=qqffqq,fill opacity=0.5] (1.53,-4.1) -- (1.53,-8.) -- (2.,-8.) -- (2.,-4.) -- (1.95,-4.) -- cycle;
			\fill[line width=0.pt,color=qqqqff,fill=qqqqff,fill opacity=0.5] (0.05,-4.) -- (0.95,-4.) -- (0.53,-4.1) -- (0.53,-8.) -- (0.47,-8.) -- (0.47,-4.1) -- cycle;
			\fill[line width=0.pt,color=qqqqff,fill=qqqqff,fill opacity=0.5] (1.05,-4.) -- (1.95,-4.) -- 	(1.53,-4.1) -- (1.53,-8.) -- (1.47,-8.) -- (1.47,-4.1) -- cycle;
			\draw [line width=0.2pt] (-4.,-4.)-- (-4.,-14.);
			\draw [line width=0.2pt] (-4.,-14.)-- (20.,-14.);
			\draw [line width=0.2pt] (20.,-14.)-- (20.,-4.);
			\draw [line width=0.2pt] (20.,-4.)-- (-4.,-4.);
			\draw [line width=0.2pt] (-4.,1.2)-- (20.,1.2);
			\draw [line width=0.2pt] (20.,1.2)-- (20.,-1.2);	
			\draw [line width=0.2pt] (20.,-1.2)-- (-4.,-1.2);
			\draw [line width=0.2pt] (-4.,-1.2)-- (-4.,1.2);
			\draw [line width=0.2pt] (0.,1.)-- (0.,-1.);
			\draw [line width=0.2pt] (0.,-1.)-- (16.,-1.);
			\draw [line width=0.2pt] (16.,-1.)-- (16.,1.);
			\draw [line width=0.2pt] (16.,1.)-- (0.,1.);
			\draw [line width=0.2pt] (2.,1.)-- (2.,-1.);
			\draw [line width=0.2pt] (4.,1.)-- (4.,-1.);
			\draw [line width=0.2pt] (6.,1.)-- (6.,-1.);
			\draw [line width=0.2pt] (8.,1.)-- (8.,-1.);
			\draw [line width=0.2pt] (10.,1.)-- (10.,-1.);
			\draw [line width=0.2pt] (12.,1.)-- (12.,-1.);
			\draw [line width=0.2pt] (14.,1.)-- (14.,-1.);
			\draw [line width=0.2pt] (0.,-8.)-- (0.,-10.);
			\draw [line width=0.2pt] (0.,-10.)-- (16.,-10.);
			\draw [line width=0.2pt] (16.,-10.)-- (16.,-8.);
			\draw [line width=0.2pt] (16.,-8.)-- (0.,-8.);
			\draw [line width=0.2pt] (2.,-8.)-- (2.,-10.001252097531708);
			\draw [line width=0.2pt] (4.,-8.)-- (4.,-10.);
			\draw [line width=0.2pt] (6.,-8.)-- (6.,-10.);
			\draw [line width=0.2pt] (8.,-8.)-- (8.,-10.);
			\draw [line width=0.2pt] (10.,-8.)-- (10.,-10.);
			\draw [line width=0.2pt] (12.,-8.)-- (12.,-10.);
		\end{tikzpicture}	
		
		A detail of the image near the boundary at the top of $R_{m,13}$.
		\caption{A scheme of the definition of $\tilde{G}_{K,N,m,\eta}$. The grey shaded area is where we use the mapping $\tilde{G}_{K,4k+2N, 4k+4+2N, k}$. In the preimage the blue area is $[-2m-5,-2m-1]\times \tilde{U}_{K,2N}^3$ resp $(\tilde{U}_{K,2N}^3+n)\times [1,1+\eta]$. The top and bottom (red) parts are stretched from size $\eta/2$ to size $13-\eta/2$.}\label{Fig:Squeeze}
	\end{figure}
	
	We define a mapping $\tilde{G}_{K,l_1,l_2,k}$, which for each $\w \in \V^{k+1}$ is defined on the part of $\tilde{Q}_{K,\w}' \setminus \tilde{Q}_{K,\w}$ farthest from $\z_{K,\w}$ in the direction $e_i$ by the expression
	\begin{equation}\label{TurnTurnTurn}
		\tilde{G}_{K,l_1,l_2,k}(x) := \tilde{\zeta}_{K,k}(x_i)[J^{3,i}_K]_{l_1}(x) + [1-\tilde{\zeta}_{K,k}(x_i)] [J^{3,i}_K]_{l_2}(x) + q_K(x_i)e_i.
	\end{equation}
	We claim that for any $l_2>l_1 \geq k$, the mapping $\tilde{G}_{K,l_1,l_2,k}$ is a homeomorphism sending $\tilde{U}_{K,k}^4\setminus \tilde{U}_{K,k+1}^4$ onto ${U}_{k}^4\setminus {U}_{k+1}^4$. This is clear from the fact that $\tilde{G}_{K,l_1,l_2,k} = {G}_{K,l_1,l_2,k}^{-1}$ and this is a homeomorphism by Section~\ref{secondmap}.

	From the reasoning of previous paragraph we have that
\begin{equation}\label{defN}
\tilde{G}_{K,4k+2N,4k+4+2N,k}\text{ is a homeomorphism}.
\end{equation}
Recall that the parameter $N$ was fixed in the previous subsection. 
	
	Call $\eta > 0$ and $n = (n_1,n_2,n_3,0)$, where $n_1,n_2,n_3$ are even numbers between $-2m$ and $2m$. We define $\tilde{G}_{K,N,m,\eta}$ on $[-2m-1, 2m+1]^3\times[-1,1]$ as follows;
	when $y = x+n$ $x\in Q(0,1)$,
	\begin{equation}\label{ASeason}
		\begin{aligned}
			\tilde{G}_{K,N,m,\eta}(y) &=\tilde{G}_{K,N,m,\eta}(x + n)\\
			& = \begin{cases}
				\tilde{G}_{K,4k+2N,4k+4+2N,k}(x) + n \ & x \in \tilde{U}_{K,k}^4 \setminus \tilde{U}_{K,k+1}^4,\\
	 			(q_K(x_1),q_K(x_2),q_K(x_3),q_K(x_4)) + n \ &  x\in \C_{A,K}.\\
			\end{cases}
		\end{aligned}
	\end{equation}
	Note that the last line together with \eqref{ASeason0} and $q_K=(t_K)^{-1}$ imply that 
	$$
	G_{K,N,m,\eta}\circ \tilde{G}_{K,N,m,\eta}(x) = x\text{ on }\C_{A,K,m}.
	$$ 
	When $y = x+n$, $x_1,x_2,x_3\in [-1, 1]$, and $x_4 \in [-1-\tfrac{\eta}{2},-1]\cup[1,1+\tfrac{\eta}{2}]$ we put
	\begin{equation}\label{OhDear1A}
		\begin{aligned}
			\tilde{G}_{K,N,m,\eta}(y) &=\tilde{G}_{K,N,m,\eta}(x+n)\\
			& = \begin{cases}
				[J^{3,4}_K]_{2N}(x) + n + \frac{26-{\eta}}{\eta}(x_4-1)e_4 +e_4 \ & x_4 \in [1,1+\tfrac{\eta}{2}],\\
				[J^{3,4}_K]_{2N}(x) + n + \frac{26-{\eta}}{\eta}(x_4+1)e_4 -e_4 \ & x_4 \in [-1-\tfrac{\eta}{2},-1].\\
			\end{cases}
		\end{aligned}
	\end{equation}
	It is easy to see from \eqref{ASeason} and \eqref{OhDear1A} that those mapping agree on the boundary $[-2m-1,2m+1]^3\times\{-1,1\}$, i.e. they are both equal to the shifted $[J^{3,4}_K]_{2N}(x)$. Moreover, the definition \eqref{OhDear1A} gives us that we stretch a slab of thickness $\frac{\eta}{2}$ onto slab of thickness $13-\frac{\eta}{2}$ in the $x_4$-coordinate (see Figure \ref{Fig:Squeeze})\textbf{}.
	Further when $y = x+n$, $x_1,x_2,x_3\in [-1, 1]$, and $x_4 \in [-1-{\eta},-1 - \tfrac{\eta}{2}]\cup[1+\tfrac{\eta}{2},1+{\eta}]$ we define
	\begin{equation}\label{OhDear1B}
		\begin{aligned}
			\tilde{G}_{K,N,m,\eta}(y) &=\tilde{G}_{K,N,m,\eta}(x+n) \\
			& = \begin{cases}
				 n+(x_4+13-\eta) e_4 \ & \\
				\quad +\tfrac{2}{\eta}(x_4-1-\tfrac{\eta}{2})(x_1,x_2,x_3,0)\ & \\
				\quad+\tfrac{2}{\eta}(1+{\eta}-x_4)[J^{3,4}_K]_{2N}(x) \ & x_4 \in [1+\tfrac{\eta}{2}, 1+\eta],\\
				n+(x_4-13+\eta) e_4\ & \\
				\quad + \tfrac{2}{\eta}(-1-\tfrac{\eta}{2}-x_4)(x_1,x_2,x_3,0) \ & \\ \quad+\tfrac{2}{\eta}(1+{\eta}+x_4)[J^{3,4}_K]_{2N}(x)  \ & x_4 \in [-1-\eta,-1-\tfrac{\eta}{2}].\\
			\end{cases}
		\end{aligned}
	\end{equation}
	If fact the definition \eqref{OhDear1B} is just a linear interpolation between $[J^{3,4}_K]_{2N}(x)$ on $[-2m-1,2m+1]^3\times\{1+\frac{\eta}{2}\}$ and identity on $[-2m-1,2m+1]^3\times\{1+\eta\}$ (see blue parts in Figure \ref{Fig:Squeeze}).
	
	On other parts of the boundary we do a similar interpolation between values on $\partial [-2m-1,2m+1]^3\times[-1,1]$ and the identity on the outer boundary:
	When there is exactly one index $i=1,2,3$ such that $y_i\in [-2m-5,-2m-1]\cup[2m+1, 2m+5]$ (then put $x_i = y_i$ and $n_i = 0$) but for $j\neq i$ $y_j = x_j+n_j$ and $x_4 \in [-1, 1]$ we put
	\begin{equation}\label{OhDear2}
		\begin{aligned}
			\tilde{G}_{K,N,m,\eta}(y) & =\tilde{G}_{K,N,m,\eta}(x+n)\\
			& = \begin{cases}
				[J^{3,i}_K]_{2N}(x) + n +x_ie_i \ & x_i \in [2m+1, 2m+4],\\
				n +x_ie_i 
				+ (x_i-2m-4)(x-x_ie_i)\ &\\
				\quad + (2m+5 - x_i)[J^{3,i}_K]_{2N}(x) \ & x_i \in [2m+4, 2m+5],\\
				[J^{3,i}_K]_{2N}(x) + n +x_ie_i \ & x_i \in  [-2m-4,-2m-1],\\
				n +x_ie_i + (-x_i-2m-4)(x-x_ie_i) \ & \\
				\quad + (2m+5 + x_i)[J^{3,i}_K]_{2N}(x) \ & x_i \in [-2m-5, -2m-4].\\
			\end{cases}
		\end{aligned}
	\end{equation}
	When there are at least two indices $i=1,2,3$ such that $y_i\in [-2m-1-\eta,-2m-1]\cup[2m+1, 2m+1+\eta]$ (or at least one index $i=1,2,3$ such that $y_i\in [-2m-1-\eta,-2m-1]\cup[2m+1, 2m+1+\eta]$ and $y_4 \in [-1-\eta,-1]\cup[1,1+\eta]$) we define the $j$-th coordinate function of $\tilde{G}_{N,K,m,\eta}$ as follows
	\begin{equation}\label{DeathToGo1}
		\tilde{G}_{K,N,m,\eta}^j(y) =y_j \text{ for } j=1,2,3
	\end{equation}
	and we just linearly stretch the corners as
	\begin{equation}\label{DeathToGo2}
		\tilde{G}_{K,N,m,\eta}^4(y) = \begin{cases}
			y_4& y_4 \in [-1,1]\\
			1 + \frac{26-\eta}{\eta}(y_4-1) \ & y_4 \in [1,1+\tfrac{\eta}{2}],\\
			y_4 + 13 -\eta \ & y_4 \in [1+\tfrac{\eta}{2}, 1+\eta],\\
			-1 + \frac{26-\eta}{\eta}(y_4-1) \ & y_4 \in [-1-\frac{\eta}{2},-1],\\
			y_4 - 13 +\eta \ & y_4 \in [-1-\eta, -1-\tfrac{\eta}{2}].\\	
		\end{cases}
	\end{equation}

\section{A mapping equaling a reflection on $\C_{B}$ with identity on the boundary} \label{CarCrash}



Recall that $R_{m,\eta}=[-2m-5,2m+5]^{3}\times[-1-\eta,1+\eta]$. 
In Section~\ref{Cantor} we define maps $\tilde{G}_{K,N,m,\eta}$ and $G_{K,N,m,\eta}$ such that 
$$
G_{K,N,m,\eta}\circ \tilde{G}_{K,N,m,\eta}(x) = x\text{ on }\partial R_{m,\eta}\text{ and on }\C_{A,K,m}, 
$$
$$
\tilde{G}_{K,N,m,\eta}(R_{m,\eta}) = R_{m,13}\text{ and }
G_{K,N,m,\eta}(R_{m,13}) = R_{m,\eta}.
$$ 
The main aim of this section is the construction of a mapping $\hat{F}_{\beta, m}$ from $R_{m,13}$ onto $R_{m,13}$ with $\hat{F}_{\beta, m}(x) = x$ on $\partial R_{m,13}$ so that $\hat{F}_{\beta, m}$ behaves like a reflection $x\to[x_1,x_2,x_3,-x_4]$ on $\C_{B,m} = \bigcup_{n_1,n_2,n_3=-m}^m\C_{B}+(2n_1,2n_2,2n_3,0)$.  
Then we have
$$
G_{K,N,m,\eta} \circ \hat{F}_{\beta, m} \circ \tilde{G}_{K,N,m,\eta}(x) = x \text{ on } \partial R_{m,\eta}
$$
and
$$
J_{G_{K,N,m,\eta} \circ \hat{F}_{\beta, m} \circ \tilde{G}_{K,N,m,\eta}}=-1<0 \text{ on } \C_{A,K,m}.
$$

The basic building block for the construction of $\hat{F}_{\beta, m}$ is the mapping $F_{\beta}$ defined in \cite[Theorem 5.1]{CHT}, which is the subject of the following theorem. We define the set
$$
A_{i,k,l}(\beta) :=\bigl\{x \in\er^4: x_i\in [-1,1]\setminus {U}^1_{k} ,\ x_o\in {U}_{l}^1 \text{ for all } o\in\{1,2,3,4\}\setminus\{i\} \bigr\}.
$$
Also, by $\K_{B}$ we denote the set of lines intersecting $\C_{B}$ parallel to coordinate axes, i.e.
$$  
\begin{array}{ll}
\mathcal K_B =& (\mathcal C_B^1 \times \mathcal C_B^1 \times \mathcal C_B^1  \times [-1,1])\cup (\mathcal C_B^1 \times \mathcal C_B^1  \times [-1,1] \times \mathcal C_B^1 ) \\
& (\mathcal C_B^1 \times [-1,1]  \times \mathcal C_B^1 \times \mathcal C_B^1 )\cup ( [-1,1]\times \mathcal C_B^1 \times \mathcal C_B^1  \times \C_B^1).
\end{array}
$$
\prt{Theorem}
\begin{proclaim}\label{BloodSweatAndTears}\cite[Theorem 5.1]{CHT}
	There is $\beta_0>0$ such that for all $\beta>\beta_0$ there exists a mapping ${ F_{\beta}} : (-1,1)^{4} \to (-1,1)^{4}$, which is a sense-preserving bi-Lipschitz extension of the map
	\begin{equation}\label{OldOld}
	F_{\color{black} \beta }(x_1,x_2,x_3,x_4) = (x_1,x_2,x_3, -x_4) \quad x\in \mathcal{K}_{B}.		
	\end{equation}
	Moreover there exists a constant $M_0 \in \mathbb{N}$ such that for each $k,l\in\en$ satisfying $M_0<k\leq l$ and for every line parallel to $e_i$, $L$, we have that $F(L\cap A_{i,k-M_0-1, l+M_0}{ (\beta ) })$ is a line segment parallel to $e_i$ which lies in the set $A_{i, k-1, l}{ (\beta ) }$. Moreover, the derivative along $L$ satisfies
	\begin{equation}\label{StraightStraight}
	D_{i}F_{{\beta} }(x) =
	\begin{cases}
	e_i & \textrm{if $i=1,2,3$}\\
	-e_{i} & \textrm{if $i=4$}.
	\end{cases}
	\end{equation}
	for every $x\in L\cap A_{i,k-M_0-1, l+M_0}{ (\beta ) }$.
\end{proclaim}	

To construct $\hat{F}_{\beta, m}$ we extend $F_{\beta}$ periodically and then tweak it to get identity on the boundary $\partial R_{m,13}$. Moreover, we slightly extend its behaviour from $Z_m$ to its neighborhood so that we get the key identity \eqref{Oasis} not only on $\mathcal{K}_{B}$ but also on lines through the Cantor set in $[-2m-3,2m+3]^3\times[-3,3]$. 
We call
\eqn{defkbm}
$$
\K_{B,m} = \bigcup_{n_1,n_2,n_3=-m}^m \Bigl(\C_{B} +(2n_1, 2n_2,2n_3,0)\Bigr)+\bigcup_{i=1}^4 \er e_i.
$$

\begin{thm}\label{reflect}
	Let $F_{\beta}$ be the map from Theorem~\ref{BloodSweatAndTears}. There exists a bi-Lipschitz map $\hat{F}_{\beta, m}: R_{m,13}\to R_{m,13}$, such that
	\begin{equation}\label{PeriodicExtension}
	\hat{F}_{\beta, m}(x+n) = F_{\beta}(x)+n \text{ for }x\in Q(0,1)\text{ and }
	\end{equation}
	for $n=(2n_1,2n_2,2n_3,0)$, $n_1,n_2,n_3\in\{-m,\hdots,m\}$. Moreover, 
	$\hat{F}_{\beta, m}(x) = x$ on $\partial R_{m,13}$. Further 
	\begin{equation}\label{Oasis}
		\hat{F}_{\beta, m}(x_1,x_2,x_3,x_4)=(x_1,x_2,x_3,-x_4)
	\end{equation}
	on $\K_{B,m}\cap [-2m-3,2m+3]^3\times[-3,3]$. Also
	\begin{equation}\label{NotGay1}
		D_1\hat{F}_{\beta, m}(x+(0,2n_2, 2n_3, 0)) = e_1 \text{ for } x\in ([-2m-3,-2m-1]\cup [2m+1,2m+3] )\times U_1^3,
	\end{equation}
	\begin{equation}\label{NotGay2}
	D_2\hat{F}_{\beta, m}(x+(2n_1,0, 2n_3, 0)) = e_2 \text{ for } x\in U_1^1\times ([-2m-3,-2m-1]\cup [2m+1,2m+3] )\times U_1^2,
	\end{equation}
	\begin{equation}\label{NotGay3}
	D_3\hat{F}_{\beta, m}(x+(2n_1,2n_2, 0, 0)) = e_3 \text{ for } x\in U_1^2\times ([-2m-3,-2m-1]\cup [2m+1,2m+3] )\times U_1^1
	\end{equation}
	and
	\begin{equation}\label{NotGay4}
	D_4\hat{F}_{\beta, m}(x+(2n_1,2n_2, 2n_3, 0)) = -e_4 \text{ for } x\in U_2^3\times([-3,3]\setminus[-1,1]).
	\end{equation}
	
\end{thm}

Let us first prove the following estimate, which is a simple result of the bi-Lipschitz nature of $F_{\beta}$ and \eqref{StraightStraight}.

\prt{Proposition}
\begin{proclaim}\label{EmotionallyDistant}There is $M\in\en$, $M\geq M_0$, such that 
	the mapping $F_{\beta}$ from Theorem~\ref{BloodSweatAndTears} maps $[-1,1] \times [{U}^3_{l}\setminus {U}^3_{l+1}]$ into $[-1-C(\beta)r_l,1+C(\beta)r_l]\times[{U}^3_{l-M}\setminus {U}^3_{l+1+M}]$ for every $l\in\en$ and similarly for all permutations of the coordinates. Likewise for every $k,l\in\en$ our $F_{\beta}$ maps $(U_{k}\setminus U_{k+1}) \times [{U}^3_{l}\setminus {U}^3_{l+1}]$ into $(U_{k-M}\setminus U_{k+M+1}) \times [{U}^3_{l-M}\setminus {U}^3_{l+M+1}]$ and maps ${U}^4_{l}\setminus {U}^4_{l+1}$ into ${U}^4_{l-M}\setminus {U}^4_{l+M+1}$.
\end{proclaim}
\begin{proof}
Denote by $\C_{B}(3)$ the Cantor set constructed in $\er^3$.  
	From \eqref{OldOld} we have $F_{\beta}(x) = (x_1,x_2,x_3, -x_4)$ on $[-1,1]\times \C_{B}(3)$ and $F_{\beta}$ is bi-Lipschitz. Hence the neighbourhoods of $[-1,1]\times \C_{B}(3)$ are mapped onto neighbourhoods of $[-1,1]\times \C_{B,K}(3)$ (and similarly for all permutations of the coordinates). Now ${U}^3_{l}$ is a neighbourhood of $\C_{B}(3)$, further there exists a $C_1$ such that for each $l$ we have that
	$$
	\begin{aligned}[]
	[-1,1] \times [\C_{B}(3)+B_3(0,C_1^{-1}2^{-l(\beta+1)})] &\subset [-1,1] \times [{U}^3_{l}]  \text{ and }\\
	[-1,1] \times [{U}^3_{l}] &\subset[-1,1]\times [\C_{B}(3)+B_3(0,C_12^{-l(\beta+1)})].
	\end{aligned}
	$$
	Then call $C$ the bi-Lipschitz constant of $F_{\beta}$ and choose $M$ so that $CC_1^2<2^{M(\beta+1)}$. Further call $\delta_l = CC_12^{-l(\beta+1)}$.  We have 
	$$	
	\begin{aligned}[]
	[-1+\delta_l,1-\delta_l] \times [{U}^3_{l+M}] &\subset [-1+\delta_l,1-\delta_l]\times [\C_{B}(3)+B_3(0,C_12^{-(l-M)(\beta+1)})]\\
	&\subset [-1+\delta_l,1-\delta_l]\times [\C_{B}(3)+B_3(0,C^{-1}C_1^{-1}2^{-l(\beta+1)})]\\
	&\subset F_{\beta}([-1,1] \times {U}^3_{l})\\
	&\subset [-1-\delta_l,1+\delta_l]\times [\C_{B}(3)+B_3(0,CC_12^{-l(\beta+1)})]\\
	& \subset[-1-\delta_l,1+\delta_l]\times [\C_{B}(3)+B_3(0,C_1^{-1}2^{-(l-M)(\beta+1)})]\\
	&\subset [-1-\delta_l,1+\delta_l] \times [{U}^3_{l-M}].
	\end{aligned}
	$$
	This fact yields the claim immediately since clearly the choice of $M$ does not depend on $l$.
	
	The argument that $F_{\beta}$ maps ${U}^4_{l}\setminus {U}^4_{l+1}$ into ${U}^4_{l-M}\setminus {U}^4_{l+M+1}$ is similar. The map $F_{\beta}$ sends $\C_{B}$ onto $\C_{B}$ and ${U}^4_{l}$ are neighbourhoods of $\C_{B}$. The rest of the argument remains analogous to the above.
\end{proof}	

Let us recall some useful notation and results from \cite[Section 4]{CHT}. Let $t\in(0,1]$ and let  $v=(\tfrac{t}{4},\tfrac{t}{8},\tfrac{t}{16},1)$ be vector. We define a projection $P_v:\er^4\to \er^{3}\times \{0\}$  in the direction of $-v$  as follows
$$
P_v(x) = x- \frac{x_4}{v_4}  v.
$$
It was shown in \cite[Lemma 4.2]{CHT} that for $t=1$ the mapping $P_v(x)$ is one to one on the Cantor set $\mathcal{C}_{B}$ and thus the whole construction of $F_{\beta}$ from Theorem \ref{BloodSweatAndTears} is possible.
It follows from the proof in \cite[Theorem 5.1]{CHT} that there are many admissible choices of $v$ and specifically any of the vectors $(\tfrac{t}{4},\tfrac{t}{8},\tfrac{t}{16},1)$ for $t\in(0,1]$ has the same properties once we choose $\beta$ sufficiently large.
Thus we may assume that $\max\{|v_1|,|v_2|,|v_3|\}$ is as small as we like. During the proof of Theorem~\ref{reflect}, we fix one single vector $v = v_t$, for $t\leq 1$ and work with that vector.

\prt{Definition}
\begin{proclaim}\label{SpagMap}
	Given a Lipschitz function $g:\er^{3}\times \{0\}\to[-3,3]$ we define the spaghetti strand map $F_{g,v}:\er^4\to\er^4$ by
	\begin{equation}\label{spaghetti}
	F_{g,v}(x) = x + g(P_v(x)) v.
	\end{equation}
\end{proclaim}

The equality $D_iF_{\beta}(x) = \pm e_i$ from \eqref{StraightStraight} was attained as follows. A Lipschitz function $g$ was defined on the hyperplane $\er^3\times\{0\}$. This function was constant on lines parallel to $e_i$, $i=1,2,3$ and linear on lines parallel to $P_v(\er e_4)$ (see the proof of Lemma~5.7, specifically Steps 3, 4, 5 equations (5.27), (5.30) pages 789-796 of \cite{CHT}). The function allows a Lipschitz periodic extension which means we are able to prove the translation equation~\eqref{PeriodicExtension} below.

\begin{proof}[Proof of Theorem~\ref{reflect}]
	Let us recall that $\beta$  and $r_k$ are from the construction of $\mathcal{C}_{B}$ \eqref{defrKk}.
	We assume that $\beta> \max\{6,\beta_0\}$ where $\beta_0$ comes from Theorem \ref{BloodSweatAndTears}. 
	\step{1}{Fixing a vector $\ve$}{RefS0}
	
	Now we prove that we can choose $t$ so that for $v=(\tfrac{t}{4},\tfrac{t}{8},\tfrac{t}{16},1)$ and for each $\ve \in \V_4^1$ we have
	\begin{equation}\label{GetInTheHole}
	P_{v}\bigl(Q(z_{\ve}, r_1)+ [-3,3]e_4\bigr) \subset Q_3(0,\tfrac{3}{4})\times\{0\}.
	\end{equation}
	Since we have $P_{v}(z_{\ve}) \to (\pm\tfrac{1}{2},\pm\tfrac{1}{2},\pm\tfrac{1}{2},0)$ as $t\to 0^+$, $r_1 = 2^{-1-\beta} < 2^{-7}$ (because $\beta > 6$) and $\diam(P_{v}([-3,3]e_4))\to 0$ as $t\to 0^+$ this is clearly possible. We choose and fix this $v$ with $t\leq 1$ so that $\max\{|v_1|,|v_2|,|v_3|\} \leq \tfrac{1}{14}$ and \eqref{GetInTheHole} holds. 
	
	\step{2}{The periodic extension of $F_{\beta}$ on $[-2m-1,2m+1]^3\times[-3,3]$}{RefS1}
	
	It was shown in \cite[Section 4 and Lemma 5.7]{CHT} that for $\beta\gg 1$ large enough we have that
	\begin{equation*}
	P_v\text{ is one-to-one on } (\C_{B}+ \bigcup_{i=1}^4\er e_i).
	\end{equation*}
	It follows rather quickly from the fact that $P_v (\C_{B}) \subset [-1,1]^3\times\{0\}$ that $P_v$ is one-to-one on $ \C_{B,m}$. By symmetry therefore $P_v$ is one-to-one on $\K_{B,m}$. Thus in the following we assume that $\beta$ is fixed and that $P_v$ is one-to-one on $\K_{B,m}$. Thus $v = v(t)$ and $\beta$ are absolutely fixed geometrical constants which do not change at all during the calculations. In doing so we have fixed the bi-Lipschitz constant of $F_{\beta}$ and the constant $M$ from Proposition~\ref{EmotionallyDistant}. None of these constants depend on $m$.
	
	From \cite[Section 4]{CHT} we know that
	$$
	g\text{ is Lipschitz }\Rightarrow F_{g,v}\text{ is bilipschitz}.
	$$
	Moreover, let $u = [-v_1,-v_2, -v_{3}, v_4]$ and let $g:\er^{3}\times\{0\}\to[-3,3]$ be a Lipschitz function satisfying
	\begin{equation}\label{defg}
	g(P_v(x)) = -x_4\text{ for every }x\in  \K_{B,m}\cap \bigl([-2m-3,2m+3]^3\times[-3,3]\bigr).
	\end{equation}
	
	This $g$ was constructed in \cite[Lemma 5.7]{CHT} for a single $\C_{B} + \bigcup_{i=1}^4\er e_i$. 
	Also by a simple cut-off argument we may assume that
	\begin{equation}\label{Bounded}
	-3\leq \min_{x\in \er^3\times\{0\}}g(x) \leq \max_{x\in \er^3\times\{0\}}g(x) \leq 3.
	\end{equation}
	In order to extend $g$ periodically by $g(x+n) = g(x)$ for all $n=(2n_1,2n_2,2n_3)$ with $-m\leq n_1,n_2,n_3\leq m$ we need that
	\begin{equation}\label{Butterfly}
	g(x) = g(x\pm2e_i) \text{ for all } x \text{ such that } x,x+2e_i\in \partial_3Q_3(0,1)\times \{0\}
	\end{equation}
	for each $i=1,2,3$. This was not proved explicitly in the original theorem, however $g$ was defined so that
	\begin{equation}\label{IronLionZion}
	g(P_v(x)) =- x_4 \text{ on }\C_{B} + \bigcup_{i=1}^4\er e_i.
	\end{equation}
	Therefore $g$ is constant on lines parallel to the $x_1, x_2$ and $ x_3$ axis intersecting $\C_{B}$ and so \eqref{Butterfly} holds at the intersection of these lines with $\partial_3 Q_3(0,1) \times\{0\}$ (in fact $g$ was defined so that \eqref{Butterfly} holds not only on these lines but also on their neighborhoods). The question of the value of $g$ on the projection of lines in direction $e_4$ is not an issue of the values of $g$ on $\partial_3Q_3(0,1)$ because by \eqref{GetInTheHole} we have
	$$
	P_v\big([\C_{B} + \er e_4] \cap ([-1,1]^3\times[-3,3])\big) \subset Q_3(0,\tfrac{3}{4})
	$$
	and therefore the projection of $(\C_{B} + \er e_4 + n) \cap (n+[-1,1]^3\times[-3,3])$ is a subset of $Q_3((2n_1,2n_2,2n_3),\tfrac{3}{4})$ and so far away from $\partial_3Q_3((2n_1,2n_2,2n_3),1)$. Since the exact values of $g$ are important only close to $\K_{B,m}$ (and that only while the values are in $[-3,3]$, see \eqref{Bounded}) and away from this set we define $g$ by an arbitrary Lipschitz extension, we may assume that $g$ has been extended to satisfy \eqref{Butterfly}.
	
	Therefore we define $g(x+n) = g(x)$ for $x\in [-1,1]^3\times\{0\}$ and $n=(2n_1,2n_2,2n_3,0)$ where 
	$n_1,n_2,n_3\in\{-m,\hdots,m\}$ and we have that $g$ is Lipschitz and defined on a subset of $[-2m-2, 2m+2]^3\times \{0\}$.
	
	This last fact together with \eqref{IronLionZion} means we have $g(P_v(x)) = -x_4$ for all $x \in [-2m-1,2m+1]\times \C_{B,m}(3)$, all $x \in \C_{B,m}(1)\times [-2m-1,2m+1]\times \C_{B,m}(2)$ and all $x \in \C_{B,m}(2)\times [-2m-1,2m+1]\times \C_{B,m}(1)$. By extending $g$ constant on lines a little further we easily get 
	\begin{equation}\label{NoelAndLiam}
	\begin{aligned}
	g(P_v(x)) = -x_4 \ \ \text{ for all } \ \ 
	&x \in \bigl([-2m-3,2m+3]\times \C_{B,m}(3) \bigr)\cup\\
	& \bigl(\C_{B,m}(1)\times [-2m-3,2m+3]\times \C_{B,m}(2)\bigr)\cup\\
	& \bigl(\C_{B,m}(2)\times [-2m-3,2m+3]\times \C_{B,m}(1)\bigr)
	\end{aligned}
	\end{equation}
	and (because $\max\{|v_1|,|v_2|,|v_3|\} \leq \tfrac{1}{14}$) the projection of that set is a subset of $[-2m-3-\tfrac{1}{14}, 2m+3+\tfrac{1}{14}]^3\times\{0\}$.
	
We extend our $g$ in a way that it is constant not only on lines through the Cantor set but also on their neighbourhoods, i.e. lines parallel to $e_1$ lying inside the set $([-2m-3,-2m-1]\cup[2m+1,2m+3])\times U_1^3$.  
It follows that $g(P_v(x)) = g(P_v(x+te_1))$ there. Note that this leads to a correct definition since it is easy to prove (see \cite[Proposition 5.3]{CHT}) that for each distinct pair of cubes $Q_{\ve}$ and $Q_{\w}$ in $U_1^3$ we have that
	$$
		\begin{aligned}
		P_v\big(([-2m-3,-2m-1]&\cup[2m+1,2m+3])\times Q_{\ve}\big)\cap \\
		& P_v\big(([-2m-3,-2m-1]\cup[2m+1,2m+3])\times Q_{\w}\big) = \emptyset.
		\end{aligned}
	$$
	Because $g$ is defined constant on these lines parallel to $e_1$ (especially $g(P_v(x)) = g(P_v(x+te_1))$) we have 
	$$
		F_{g,v}(x+te_1)=x+t e_1+v g(P_v(x+t e_1)) = x+t e_1+v g(P_v(x))=F_{g,v}(x)+te_1 
	$$
	and further, because the similar identity holds for  $F_{g,u}$, we have
	$$
		\begin{aligned}
		\hat{F}_{\beta, m}(x+(0,2n_2, 2n_3, 0)+te_1) &= F_{g,u}(F_{g,v}(x+(0,2n_2, 2n_3, 0)+te_1))\\
		&= F_{g,u}(F_{g,v}(x+(0,2n_2, 2n_3, 0)))+te_1,
		\end{aligned}
	$$
	when we put (same as in \cite[proof of Theorem 5.1]{CHT})
	\begin{equation}\label{FDeath}
	\hat{F}_{\beta, m}(x) = F_{g,u}\circ F_{g,v}: \er^4\to\er^4.
	\end{equation}
	Therefore we get
	$$
		D_1\hat{F}_{\beta, m}(x+(0,2n_2, 2n_3, 0)) = e_1
	$$
	which is precisely \eqref{NotGay1}. We get the equations \eqref{NotGay2} and \eqref{NotGay3} the same way; the only difference is a permutation of the coordinates.
	
	Recall that the condition $g(P_v(x)) = -x_4$ from \eqref{NoelAndLiam} is precisely the condition that guarantees 
	(see \cite[Lemma 4.5]{CHT}) 
	\begin{equation}\label{Digsy}
		\hat{F}_{\beta,m}(x_1,x_2,x_3,x_4) = (x_1,x_2,x_3, -x_4)
	\end{equation}
	on the set in \eqref{NoelAndLiam}.
	
	In \cite[Step 4 in the proof of Lemma 5.7]{CHT} the function $g$ was defined linear on lines parallel to $P_v(e_4)$ lying inside $P_v(U_{k+2}^3\times([-1,1\setminus U_k]))$. Since the set $P_v(U_2^3\times[-3,3])\subset Q_3(0,\tfrac{3}{4})$ we can define $g$ linear on lines parallel to $P_v(e_4)$ in $P_v(U_{2}^3\times([-3,3]\setminus [-1,1]))$. This means that for all $x\in U_2^3\times([-3,3]\setminus[-1,1])$ we have
	$$
	D_4\hat{F}_{\beta, m}(x+(2n_1,2n_2, 2n_3, 0)) = -e_4,
	$$
	which is \eqref{NotGay4}. Especially since $g(P_v(x)) = -x_4$ for all $x \in \C_B^3\times[-3,3]$ we can combine with \eqref{Digsy} to get \eqref{Oasis}. Therefore in the following we may assume that $g$ is Lipschitz and defined on $[-2m-3-\tfrac{1}{14}, 2m+3+\tfrac{1}{14}]^3\times\{0\}$.

	Then the map $\hat{F}_{\beta, m}$ satisfies (see \cite[Lemma 4.5]{CHT})
	\begin{equation}\label{NewNew}
	\hat{F}_{\beta, m}(x_1,x_2,x_3,x_4) = (x_1,x_2,x_3,-x_4) \text{ for } 
	x\in \K_{B,m}\cap [-2m-3,2m+3]^3\times[-3,3].
	\end{equation}
	Thus we see that $\hat{F}_{\beta, m}(x+n) =\hat{F}_{\beta, m}(x)+n$, which is a bi-Lipschitz map because $g$ is a Lipschitz function.
	
	\step{3}{Obtaining identity on the boundary}{RefS2}
	We have defined $\hat{F}_{\beta, m}$ on $[-2m-3,2m+3]^3\times[-3,3]$. Now we want to define $\hat{F}_{\beta, m}$ on $R_{m,13} \setminus [-2m-3,2m+3]^3\times[-3,3]$ so that $\hat{F}_{\beta, m}(x) = x$ on $\partial R_{m,13}$. Since by Step~\ref{RefS0} we have that $\max\{|v_1|,|v_2|,|v_3|\} \leq \tfrac{1}{14}$ it is easy to check that
	$$
	\begin{aligned}
	P_v([-2m-3,2m+3]^3\times[-3, 3]&\subset [2m-3-\tfrac{3}{14},2m+3+\tfrac{3}{14}]^3\times\{0\} \\
	\text{ and hence } F_{g,v}(x)&\in [-2m-3-\tfrac{3}{14},2m+3+\tfrac{3}{14}]^3\times[-6, 6]\text{ and }\\
	P_u([-2m-3-\tfrac{3}{14},2m+3+\tfrac{3}{14}]^3\times[-6, 6])
	&\subset [2m-3-\tfrac{9}{14},2m+3+\tfrac{9}{14}]^3\times\{0\}\\
	\end{aligned}
	$$
	Thus we can define
	\eqn{gjenula}
	$$
	g \equiv 0 \text{ on } \bigl( \er^3\setminus{[-2m-3-\tfrac{12}{14},2m+3+\tfrac{12}{14}]} \bigr) \times \{0\}
	$$
	without altering the behaviour of $\hat{F}_{\beta, m}$ on $[-2m-3,2m+3]^3\times[-3,3]$.
	Clearly there exists a Lipschitz extension of $g$ satisfying the above and \eqref{defg}. Now, since $g = 0$ on $([-2m-5,2m+5]^3\setminus [-2m-4,2m+4])\times\{0\}$ and $-14\leq x_4 \leq 14$ and $\tfrac{1}{14} \geq \max\{|v_1|, |v_2|, |v_3|\}$ we have that the image in $P_v$ of the hyperplane parts of $\partial R_{m,13}$ which correspond to $x_i=\pm( 2m+ 5)$, $i=1,2,3$ belong in $(\partial_3[-2m-5,2m+5]^3 + B_3(0,1))\times\{0\}$ and on this set we have $g = 0$. The same holds for $P_u$. But by the definition of $\hat{F}_{\beta, m}$, the fact that $g= 0$ on $P_v(x)$ and $g =0$ on $P_u(x)$ means that (see \eqref{SpagMap}) on those parts of $\partial R_{m,13}$ 
	$$\hat{F}_{\beta, m}(x)=F_{g,u}(F_{g,v}(x))=F_{g,u}(x - 0v) = x -0v -0u = x.$$
	
	Now, to get identity on the remaining parts of $\partial R_{m,13}$ where $x_4=\pm 14$  it suffices to alter the map $\hat{F}_{\beta, m}$ on $[-2m-5, 2m+5]^3\times([-14,-3]\cup[3,14])$. Spaghetti strand mappings defined in \eqref{spaghetti} take a line through
	$y\in \er^{3}\times \{0\}$ in the direction $v$ and move it in the direction of $v$ by the amount $g(y)$. Instead we define a ``rubber band'' mapping (see Figure \ref{Fig:Rubber})
	\begin{equation}\label{rubber}
	R_{g,v}(x) = x + r(x)g(P_v(x)) v,
	\end{equation}
	where $r$ is a fixed Lipschitz function with Lipschitz constant $\tfrac{1}{7}$ such that
	$$
	r(x)=\begin{cases}
	0 \qquad & x_4\geq 13,\\
	\tfrac{13-x_4}{7}\qquad & x_4\in [6,13],\\
	1 \qquad & x_4\in [-6,6],\\
	\tfrac{13+x_4}{7}\qquad & x_4\in [-13,-6],\\
	0\qquad & x_4\leq -13.\\
	\end{cases}
	$$
	Then it is easy to check that the fourth coordinate of $R_{g,v}(x)$, i.e.
	\eqn{increase}
	$$
	t\to t+r(t) g(y)\text{ is an increasing function for any }g(y)\in [-3,3].
	$$
	Now we define $\hat{F}_{\beta, m}:=R_{g,u}\circ R_{g,v}$. Since $R_{g,v}(x)=x$ whenever $|x_4|\geq 13$ it is easy to check that $\hat{F}_{\beta, m}(x)=x$ whenever $13\leq |x_4|\leq 14$.
	Moreover, as $R_{g,v}(x)=F_{g,v}(x)$ whenever $|x_4|\leq 3$ it is not difficult to check that the new definition of $\hat{F}_{\beta, m}$ using rubber band maps is equal to the definition of the previous definition of $\hat{F}_{\beta, m}$ using spaghetti strand maps on $[-2m-3,2m+3]^3\times[-3,3]$.
	
	
	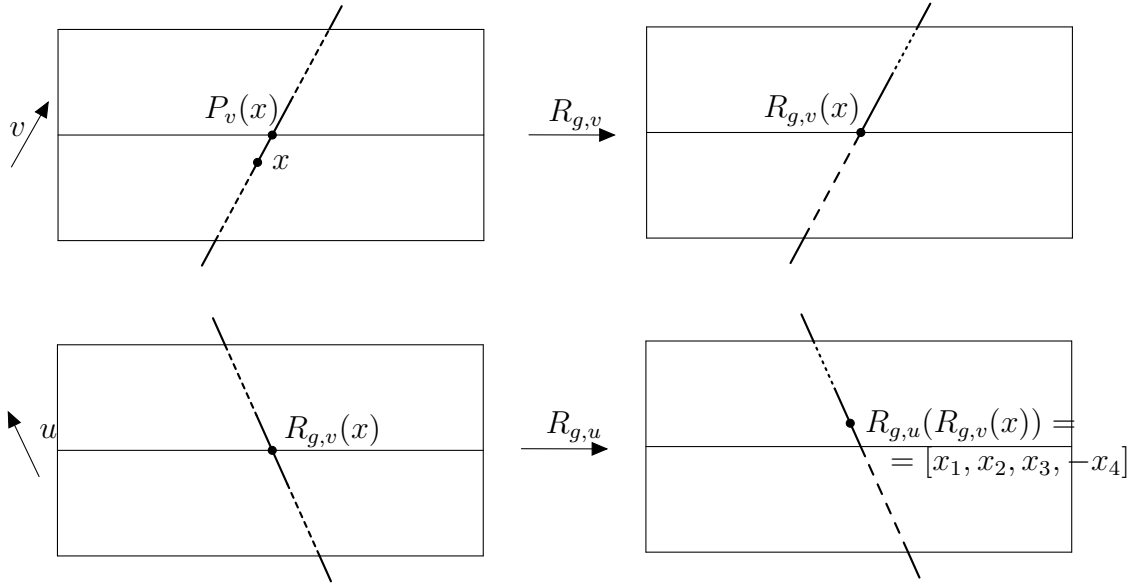
\begin{figure}[h!]
		\centering
		\begin{tikzpicture}[line cap=round,line join=round,>=triangle 45,x=0.7cm,y=0.7cm]
		\clip(-1.1858961803640047,-8.820130268357834) rectangle (21.770171126118576,4.287663299307549);
		\draw [line width=0.4pt] (0.,3.)-- (8.,3.);
		\draw [line width=0.4pt] (0.,-1.)-- (8.,-1.);
		\draw [line width=0.4pt] (0.,3.)-- (0.,-1.);
		\draw [line width=0.4pt] (8.,3.)-- (8.,-1.);
		\draw [line width=0.4pt] (0.,1.)-- (8.,1.);
		\draw [->,line width=0.4pt] (-0.8755867258606931,0.39487561787930525) -- (-0.1725709528934669,1.6446814364877058);
		\draw (-1.1169519612364,1.4464208946962798) node[anchor=north west] {$v$};
		\draw [line width=0.8pt] (2.7036062327779917,-1.4735462509222261)-- (2.9566653404136787,-1.);
		\draw [line width=0.8pt] (5.094231477979819,3.)-- (5.33004633181297,3.441277300737282);
		\draw [line width=0.8pt] (3.6754978760824972,0.34514207169709343)-- (4.365370288170033,1.6360914368906005);
		\draw [line width=0.8pt,dash pattern=on 2pt off 2pt] (2.9566653404136787,-1.)-- (3.6754978760824972,0.34514207169709343);
		\draw [line width=0.8pt,dash pattern=on 2pt off 2pt] (4.365370288170033,1.6360914368906005)-- (5.094231477979819,3.);
		\draw (3.8146907753277044,0.8141419244356437) node[anchor=north west] {$x$};
		\draw (2.540994122687917,1.9807501651796832) node[anchor=north west] {$P_v(x)$};
		\draw [line width=0.4pt] (11.058164097157666,3.04801839630791)-- (19.058164097157665,3.04801839630791);
		\draw [line width=0.4pt] (11.058164097157666,-0.9519816036920908)-- (19.058164097157665,-0.9519816036920908);
		\draw [line width=0.4pt] (11.058164097157666,3.04801839630791)-- (11.058164097157666,-0.9519816036920908);
		\draw [line width=0.4pt] (19.058164097157665,3.04801839630791)-- (19.058164097157665,-0.9519816036920908);
		\draw [line width=0.4pt] (11.058164097157666,1.0480183963079097)-- (19.058164097157665,1.0480183963079097);
		\draw [line width=0.8pt] (13.761770329935661,-1.4255278546143124)-- (14.014829437571343,-0.9519816036920906);
		\draw [line width=0.8pt] (16.152395575137483,3.04801839630791)-- (16.388210428970638,3.4892956970451907);
		\draw [line width=0.8pt] (14.98535698046653,0.8641540953889966)-- (15.66531129112931,2.136543845045086);
		\draw [line width=0.8pt,dash pattern=on 4pt off 4pt] (14.014829437571343,-0.9519816036920906)-- (14.98535698046653,0.8641540953889966);
		\draw [line width=0.8pt,dotted] (15.66531129112931,2.136543845045086)-- (16.152395575137483,3.04801839630791);
		\draw (13.04691266789846,1.997914800421385) node[anchor=north west] {$R_{g, v}(x)$};
		\draw [line width=0.4pt] (-0.011082000095416156,-2.9735417482944424)-- (7.988917999904585,-2.9735417482944424);
		\draw [line width=0.4pt] (-0.011082000095416156,-6.973541748294442)-- (7.988917999904585,-6.973541748294442);
		\draw [line width=0.4pt] (-0.011082000095416156,-2.9735417482944424)-- (-0.011082000095416156,-6.973541748294442);
		\draw [line width=0.4pt] (7.988917999904585,-2.9735417482944424)-- (7.988917999904585,-6.973541748294442);
		\draw [line width=0.4pt] (-0.011082000095416156,-4.973541748294442)-- (7.988917999904585,-4.973541748294442);
		\draw (4.034374466405897,-4.077481596905404) node[anchor=north west] {$R_{g,v}(x)$};
		\draw [->,line width=0.4pt] (-0.35602956873166075,-5.470464737595762) -- (-0.9458406990685659,-4.222259787347895);
		\draw (-0.5522109432621935,-4.280000652741894) node[anchor=north west] {$u$};
		\draw [->,line width=0.4pt] (8.847767372107022,1.0037411581953766) -- (10.48004,1.);
		\draw (9.004263996446008,1.8832692210161722) node[anchor=north west] {$R_{g,v}$};
		\draw [line width=0.8pt] (2.9112742747038753,-2.469872785036305)-- (3.135222484943014,-2.973541748294443);
		\draw [line width=0.8pt] (5.127942686960046,-7.455253467007065)-- (4.913757408282032,-6.973541748294443);
		\draw [line width=0.8pt,dash pattern=on 2pt off 2pt] (4.913757408282032,-6.973541748294443)-- (4.309173515372972,-5.613807092594763);
		\draw [line width=0.8pt,dash pattern=on 2pt off 2pt] (3.135222484943014,-2.973541748294443)-- (3.7314453471259554,-4.31447209351048);
		\draw [line width=0.8pt] (3.7314453471259554,-4.31447209351048)-- (4.309173515372972,-5.613807092594763);
		\draw [->,line width=0.4pt] (8.70577288259091,-4.942339580819469) -- (10.338045510483889,-4.946080739014845);
		\draw (8.95780820239388,-4.02836017500192) node[anchor=north west] {$R_{g,u}$};
		\draw [line width=0.4pt] (11.045135209174815,-2.902001204858222)-- (19.045135209174813,-2.902001204858222);
		\draw [line width=0.4pt] (11.045135209174815,-6.902001204858223)-- (19.045135209174813,-6.902001204858223);
		\draw [line width=0.4pt] (11.045135209174815,-2.902001204858222)-- (11.045135209174815,-6.902001204858223);
		\draw [line width=0.4pt] (19.045135209174813,-2.902001204858222)-- (19.045135209174813,-6.902001204858223);
		\draw [line width=0.4pt] (11.045135209174815,-4.902001204858222)-- (19.045135209174813,-4.902001204858222);
		\draw (14.95997670269463,-3.962836017500192) node[anchor=north west] {$R_{g,u}(R_{g,v}(x))=$};
		\draw [line width=0.8pt] (13.967491483974106,-2.3983322416000847)-- (14.191439694213242,-2.9020012048582227);
		\draw [line width=0.8pt] (16.184159896230277,-7.383712923570846)-- (15.96997461755226,-6.902001204858222);
		\draw [line width=0.8pt,dash pattern=on 4pt off 4pt] (15.96997461755226,-6.902001204858222)-- (15.162637185874408,-5.086265295299223);
		\draw [line width=0.8pt,dotted] (14.191439694213242,-2.9020012048582227)-- (14.593084330566324,-3.8053169195680203);
		\draw [line width=0.8pt] (14.593084330566324,-3.8053169195680203)-- (15.162637185874408,-5.086265295299223);
		\draw (15.428635243661441,-4.69007687608785) node[anchor=north west] {$=[x_1,x_2,x_3,-x_4]$};
		\begin{scriptsize}
		\draw [fill=black] (3.748646494873849,0.48202414052447407) circle (1.5pt);
		\draw [fill=black] (4.025448409196748,1.) circle (1.5pt);
		\draw [fill=black] (15.083612506354413,1.0480183963079097) circle (1.5pt);
		\draw [fill=black] (4.0244899466125235,-4.973541748294443) circle (1.5pt);
		\draw [fill=black] (14.884176024948353,-4.459994408465533) circle (1.5pt);
		\end{scriptsize}
		\end{tikzpicture}
		\caption{Two-dimensional projection of the `rubber band' mappings.}\label{Fig:Rubber}
	\end{figure}
	
	Since $g$ is a Lipschitz map and since $r$ is Lipschitz with coefficient $\tfrac{1}{7}$ (thus giving us \eqref{increase}) we see that $R_{g,v}$ and $R_{g,u}$ are bi-Lipschitz maps as were $F_{g,v}$ and $F_{g,u}$. Therefore $\hat{F}_{\beta, m}$ is a bi-Lipschitz map, whose constant depends on the Lipschitz constant of the function $g$, which is fixed by the choice of $v$ and $\beta$.
\end{proof}

We include the following lemma about the image of certain points in $\hat{F}_{\beta, m}$. Later, in the course of proving Theorem~\ref{TheBigLebowski}, it is necessary to apply derivative estimates of the composition $G_{K,N,m,\eta}\circ \hat{F}_{\beta, m}\circ \tilde{G}_{K,N,m,\eta}$. For certain points $x$ we will need to know that $|DG_{K,N,m,\eta}(\hat{F}_{\beta, m}\circ \tilde{G}_{K,N,m,\eta}(x))|$ is not too large. We do this by proving that these points are not mapped by $\hat{F}_{\beta, m}\circ \tilde{G}_{K,N,m,\eta}$ onto the set where the derivative of $G_{K,N,m,\eta}$ is large. This is the content of the following Lemma.

\begin{lemma}\label{ThatTimeOfTheMonth}
	Let $x\in \G$ where
	$$
	\begin{aligned}
	\G = R_{m,13} \setminus\Big[
	\bigcup_{n_1,n_2,n_3=-m}^m &\Big([-2m-3,2m+3]\times U^3_{1} +(0,2n_2,2n_3,0) \\
	&\cup U^1_{1}\times [-2m-3,2m+3]\times U^2_{1} +(2n_1,0,2n_3,0)\\
	&  \cup U^2_{1}\times [-2m-3,2m+3]\times U^1_{1} +(2n_1,2n_2,0,0)\\
	& \cup U^3_{1}\times[-3,3]+ (2n_1,2n_2,2n_3,0)  \Big)\Big]
	\end{aligned}
	$$
	then
	$$
	\hat{F}_{\beta, m}(x) \in R_{m,13} \setminus \mathcal{T}
	$$
	where
	$$
	\begin{aligned}
	\T = \bigcup_{n_1,n_2,n_3=-m}^m &\Big([-2m-2,2m+2]\times U^3_{M+1} +(0,2n_2,2n_3,0) \\
	&\cup U^1_{M+1}\times [-2m-2,2m+2]\times U^2_{M+1} +(2n_1,0,2n_3,0)\\
	&  \cup U^2_{M+1}\times [-2m-2,2m+2]\times U^1_{M+1} +(2n_1,2n_2,0,0)\\
	& \cup U^3_{M+1}\times[-2,2]+ (2n_1,2n_2,2n_3,0)  \Big).
	\end{aligned}
	$$
\end{lemma}
\begin{proof}
	We may assume that $M$ is so large that $C(\beta)r_M=C(\beta)2^{-(\beta+1)k} < r_1 = 2^{-\beta-1}$ where $C(\beta)$ is the constant from Proposition \ref{EmotionallyDistant}. Nevertheless $M$ is an absolute constant depending only of the construction of ${F}_{\beta}$.  Specifically it depends on $\beta$ and $v$, which are fixed and do not change at any point. Notice that the Proposition~\ref{EmotionallyDistant} holds not only for ${F}_{\beta}$ but also for ${F}_{\beta}^{-1}$ since both of these maps are bi-Lipschitz and send $\K_{B}$ onto $\K_{B}$. In that case we have
	$$
	F_{\beta}^{-1}([-1,1]\times U^3_{M+1}) \subset [-2,2] \times U^3_{1}
	$$
	and similarly for coordinates 2 and 3. Moreover the fact that $\hat{F}_{\beta, m}$ is constructed by translation of $F_{\beta}$ on these sets we have the same for $\hat{F}_{\beta, m}$, precisely that
	$$
	\hat{F}_{\beta, m}^{-1}([-2m-1,2m+1]\times U^3_{M+1}+(0,2n_2,2n_3,0)) \subset [-2m-2,2m+2] \times U^3_{1} +(0,2n_2,2n_3,0)
	$$
	and similarly for the permutations of the coordinates. In fact, by the extension in \eqref{NoelAndLiam}, we have
	$$
	\hat{F}_{\beta, m}^{-1}([-2m-2,2m+2]\times U^3_{M+1}+(0,2n_2,2n_3,0)) \subset [-2m-3,2m+3] \times U^3_{1} +(0,2n_2,2n_3,0)
	$$
	because $\hat{F}_{\beta, m}$ sends $[-2m-2,2m+2]\times \C_{B,m}(3)$ onto $[-2m-2,2m+2]\times \C_{B,m}(3)$ and so the reasoning from Proposition~\ref{EmotionallyDistant} applies here too. Similar inclusions hold for the permutations of the coordinates 2 and 3.
	
	The last necessary observation is that
	$$
	\hat{F}_{\beta, m}^{-1}\Big(U^3_{M+1}\times[-2,2]+ (2n_1,2n_2,2n_3,0)\Big) \subset U^3_{1}\times[-3,3]+ (2n_1,2n_2,2n_3,0)
	$$
	for each $n$. The reasoning here is essentially the same as before, and especially using the fact that $\hat{F}_{\beta, m}(x_1,x_2,x_3,x_4) = (x_1,x_2,x_3,-x_4)$ on $\C_{B,m}(3) \times [-3,3]$.
	
	This precisely means that if $x\in \G$ then $x\notin R_{m,13}\setminus \G \supset \hat{F}_{\beta, m}^{-1}(\T)$ and therefore $\hat{F}_{\beta, m}(x) \notin \T$, which was exactly our claim.
\end{proof}

\section{Estimates of the derivative of the stretching and squeezing mapping}\label{Germany}

Let us briefly recall the role of various constants involved in the estimates. The sequence $\tilde{r}_k(K)$ used in the definition of $\C_{A,K}$ was defined in \eqref{defK}. Analogously we define $\C_{B}$ with the help of sequence $r_k$ and parameter $\beta$ in \eqref{defrKk}. The parameter $N$ was introduced in the construction of $G_{K,N,m,\eta}$ in subsection \ref{secondmap}. Parameters $m$ and $\eta$ denote the size of the boxes (see \eqref{defm} for the definition of $R_{m,\eta}$).

	The following lemma on the size of the derivative of $J_K^n,$ $H_K^n, [J_K^n]_{k}$ and $[H_K^n]_{k}$ (defined in Section \ref{defHJ}) is standard and the proof can be found in \cite[proof of Theorem 4.10]{HK} in combination with \eqref{DistancesEstimate1}.

	\begin{lemma}\label{FrameToFrameDifferential}
		Let $n \ge 2$, $k\in \en$, $l\geq k$, $\ve\in \V_n^{k+1}$, let  $x\in   Q(\z_{K, \ve}, \tfrac{1}{2}\tilde{r}_k(K)) \setminus Q(\z_{K, \ve}, \tilde{r}_{k+1}(K))$ be a point such that
		$$
		|x_1 - (\z_{K, \ve})_1| > \max \{|x_2 - (\z_{K, \ve})_2|, |x_3 - (\z_{K, \ve})_3|,|x_4 - (\z_{K, \ve})_4|\}.
		$$
		Then it holds that
		\begin{equation}\label{Standard1}
			|D_1 J_K^n(x)| =  |D_1[J_K^n]_{l}(x)| \leq C 2^{-\beta k} (k+K)^{\alpha +1},
		\end{equation}
		further for $j\neq 1$
		\begin{equation}\label{Standard2}
			|D_j J_K^n(x)| =  |D_j[J_K^n]_{l}(x)| \leq C 2^{-\beta k}
		\end{equation}
		and
		\begin{equation}\label{Standard3}
			|D[J_K^n]_{k}(y)| \leq C 2^{-\beta k} \text{ for } y\in \tilde{U}_{K,k+1}^n .
		\end{equation}
		
		Now assume that $w\in  Q(z_{\ve}, \tfrac{1}{2}r_k) \setminus Q(z_{\ve}, r_{k+1})$ be a point such that
		$$
		|w_1 - (z_{\ve})_1| > \max \{|w_2 - (z_{\ve})_2|, |w_3 - (z_{\ve})_3|,|w_4 - (z_{ \ve})_4|\}.
		$$
		 Then it holds that
		\begin{equation}\label{Standard4}
			|D_1 H_K^n(w)| =  |D_1[H_K^n]_{l}(w)| \leq C 2^{\beta k} (k+K)^{-\alpha -1},
		\end{equation}
		further for $j\neq 1$
		\begin{equation}\label{Standard5}
			|D_j H_K^n(w)| =  |D_j[H_K^n]_{l}(w)| \leq C 2^{\beta k}
		\end{equation}
		and 
		\begin{equation}\label{Standard6}
			|D_i[H_K^n]_{k}(z)| \leq C 2^{\beta k}\text{ for }z\in {U}_{k+1}^n.
		\end{equation}
	\end{lemma}

\subsection{Estimates on the derivative of $\tilde{G}_{K,N,m,\eta}$ defined in Section \ref{firstmap}.}\label{DFM}
		
	The following proposition gives an estimate on $|D_1\tilde{G}_{K,N,m,\eta}|$ at points which are very close to the Cantor set in coordinates $(x_2,x_3,x_4)$ compared to the $x_1$ variable. Especially we use it to estimate the derivative of $\tilde{G}_{K,N,m,\eta}$ along lines that go through $\C_{A,K}$.
	
	\prt{Proposition}
	\begin{proclaim}\label{SQ1}
		Let $-2m\leq n_1,n_2,n_3 \leq 2m$ be even numbers, call $n = (n_1,n_2,n_3,0)$ and let $2N+4k \leq l$ with $k,l\in \en_0$. Let $x$ be a point such that $x_1\in \tilde{U}^1_{K,k}\setminus \tilde{U}^1_{K,k+1}$, $(x_2,x_3,x_4)\in \tilde{U}^3_{K,l}$. Then the following estimates hold
		\begin{equation}\label{SQ1A}
			D_1\tilde{G}_{K,N,m,\eta}^1(x+n) = \frac{\tfrac{1}{2}r_k -r_{k+1}}{\tfrac{1}{2}\tilde{r}_k(K) - \tilde{r}_{k+1}(K) } \leq C2^{-k\beta}(K+k)^{\alpha+1}
		\end{equation}
		and
		\begin{equation}\label{SQ1B}
			|D_1\tilde{G}_{K,N,m,\eta}(x+n)-D_1\tilde{G}_{K,N,m,\eta}^1(x+n)e_1| \leq C 2^{-4k\beta-2N\beta}(K+k)^{\alpha+1}.
		\end{equation}
		Further it holds that
		\begin{equation}\label{SQ1C}
			\tilde{G}_{K,N,m,\eta}(x+n) - n \in [{U}^1_{k}\setminus {U}^1_{k+1}] \times {U}^3_{4k+2N}.
		\end{equation}
		The same holds for each rotation, i.e. for $D_2\tilde{G}_{K,N,m,\eta}^2(x+n)$ when $x_2\in \tilde{U}^1_{K,k}\setminus \tilde{U}^1_{K,k+1}$ and $(x_1,x_3,x_4)\in \tilde{U}^3_{K,l}$ and similarly for coordinates $3$ and $4$.
	\end{proclaim}
	\begin{proof}
		The claim of rotational symmetry results directly from the definition of $\tilde{G}_{K,N,m,\eta}(x+n)$ in \eqref{ASeason} and \eqref{TurnTurnTurn}. Clearly we have $k<l$ and so $x\in \tilde{U}_{K,k}^4\setminus  \tilde{U}_{K,k+1}^4$. Thus there is exactly one $\ve\in \V_4^{k+1}$ such that $x\in  Q(\z_{K, \ve}, \tfrac{1}{2}\tilde{r}_{k}(K)) \setminus Q(\z_{K, \ve}, \tilde{r}_{k+1}(K))$. Since $k<l$ the point $x$ is farthest from $\z_{K, \ve}$ in the $x_1$ coordinate. Therefore we are on the set where (see \eqref{TurnTurnTurn} and \eqref{ASeason})
		\eqn{starr}
		$$
			\begin{aligned}
				\tilde{G}_{K,N,m,\eta}(x+n) &= \tilde{G}_{K,4k+2N,4k+4+2N,k}(x)+n\\ &=\tilde{\zeta}_{K,k}(x_1)[J^{3,1}_K]_{4k+2N}(x) + q_K(x_1)e_1 \\
				& \qquad+ [1-\tilde{\zeta}_{K,k}(x_1)] [J^{3,1}_K]_{4k+4+2N}(x) + n.
			\end{aligned}
		$$
		Both $[J^{3,1}_K]_{4k+2N}$ and $[J^{3,1}_K]_{4k+4+2N}$ are constant $0$ in the $x_1$ variable. Therefore
		\begin{equation}\label{SQ1P0}
			D_1\tilde{G}_{K,N,m,\eta}(x+n) = q_K'(x_1)e_1+\tilde{ \zeta}_{K,k}'(x_1)\big([J^{3,1}_K]_{4k+2N}(x) - [J^{3,1}_K]_{4k+4+2N}(x)\big).
		\end{equation}
		On the other hand by \eqref{linear}
		$$
			q_K'(x_1) = \frac{\tfrac{1}{2}r_k -r_{k+1}}{\tfrac{1}{2}\tilde{r}_k(K) - \tilde{r}_{k+1}(K) }
		$$
		thus proving \eqref{SQ1A}. Further, since $[J^{3,1}_K]_{l}$ maps each $\Q_{K,\ve}$ into corresponding $Q_{\ve}$ (for each $\ve \in \V_3^{l}$) and $\diam Q_{\ve}<C2^{-l-l\beta}$ we have
		\begin{equation}\label{SQ1P1}
			\|[J^{3,1}_K]_{4k+2N} - [J^{3,1}_K]_{4k+4+2N}\|_{\infty} \leq C 2^{-k} 2^{-4k\beta - 2N\beta}.
		\end{equation}
		On the other hand we have that each interval in $\tilde{U}^1_{K,k}\setminus \tilde{U}^1_{K,k+1}$ has length smaller than  
		$C2^{-k}(K+k)^{-\alpha-1}$ by \eqref{DistancesEstimate1}. 
		Therefore, for $x_1\in \tilde{U}^1_{K,k}\setminus \tilde{U}^1_{K,k+1}$
		\begin{equation}\label{SQ1P2}
			\tilde{ \zeta}_{K,k}'(x_1) \leq C2^{k}(K+k)^{\alpha+1}.
		\end{equation}
		Now \eqref{SQ1B} follows immediately from \eqref{SQ1P0}, \eqref{SQ1P1} and \eqref{SQ1P2}.
		
		Now \eqref{starr} implies \eqref{SQ1C}
		by noticing that $[J^{3,1}_K]_{4k+2N}$ sends $[\tilde{U}^1_{K,k}\setminus \tilde{U}^1_{K,k+1}] \times \tilde{U}^3_{K,l}$ onto $\{0\}\times {U}^3_{l}$. Of course $q_K$ sends $[\tilde{U}^1_{K,k}\setminus \tilde{U}^1_{K,k+1}]$ into $[{U}^1_{k}\setminus {U}^1_{k+1}]$.
	\end{proof}
		
	The following proposition gives an estimate on $|D_1\tilde{G}_{K,N,m,\eta}|$ at points which are closer to the Cantor set in coordinates $(x_2,x_3,x_4)$ than in the $x_1\in \tilde{U}^1_{K,k}\setminus \tilde{U}^1_{K,k+1}$ variable on lines that do not intersect the Cantor set (i.e. they go through $\tilde{U}^3_{K,l}\setminus \tilde{U}_{K,l+1}^3$) as the distance from the Cantor set in the first variable decreases.


	\prt{Proposition}
	\begin{proclaim}\label{SQ2}
		Let $-2m\leq n_1,n_2,n_3 \leq 2m$ be even numbers, call $n = (n_1,n_2,n_3,0)$ and let $k < l < 4 k+2N$ with $k,l\in \en_0$. Let $x$ be a point such that $x_1\in \tilde{U}^1_{K,k}\setminus \tilde{U}^1_{K,k+1}$, $(x_2,x_3,x_4)\in \tilde{U}^3_{K,l}\setminus \tilde{U}_{K,l+1}^3$. Then 
		\begin{equation}\label{SQ2A}
			D_1\tilde{G}_{K,N,m,\eta}(x+n) = \frac{\tfrac{1}{2}r_k -r_{k+1}}{\tfrac{1}{2}\tilde{r}_k(K) - \tilde{r}_{k+1}(K) }e_1
		\end{equation}
		and
		\begin{equation}\label{SQ2B}
			\tilde{G}_{K,N,m,\eta}(x+n) - n \in [{U}^1_{k}\setminus {U}^1_{k+1}] \times[{U}^3_{l}\setminus {U}^3_{l+1}].
		\end{equation}
		The same holds for each rotation, i.e. for $D_2\tilde{G}_{K,N,m,\eta}(x+n)$ when $x_2\in \tilde{U}^1_{K,k}\setminus \tilde{U}^1_{K,k+1}$ and $(x_1,x_3,x_4)\in \tilde{U}^3_{K,l}\setminus \tilde{U}_{K,l+1}^3$ and similarly for coordinates $3$ and $4$.
	\end{proclaim}
	
	\begin{proof}
		Again we have \eqref{starr}.  
		Since $l<4k+2N$ we have that (see \eqref{defT3}, \eqref{defT1}, \eqref{defT2} for the notation)
		$$
		[J_K^{3,1}]_{4k+2N}(x) =T^1([J_K^{3}]_{4k+2N}(T_1(x))) = [J_K^{3,1}]_{4k+4+2N}(x) =T^1([J_K^{3}]_{4k+4+2N}(T_1(x)))
		$$
		because $[J_K^{3}]_{o} = J_K^{3}$ for all $o\in \mathbb N$ on $\er^3\setminus \tilde{U}_{K,o}^3$. Therefore we have
		\begin{equation}\label{SQ2P1}
			\tilde{G}_{K,N,m,\eta}(x+n)= \tilde{G}_{K,4k+N,4k+4+N,k}(x)+n =J^{3,1}_K(x) + q_K(x_1)e_1.
		\end{equation}
		Obviously $D_1J^{3,1}_K(x) = 0$ (see the definition in Section~\ref{firstmap}) and as in Proposition~\ref{SQ1} we have by \eqref{linear}
		$$
			q_K'(x_1) = \frac{\tfrac{1}{2}r_k -r_{k+1}}{\tfrac{1}{2}\tilde{r}_k(K) - \tilde{r}_{k+1}(K) }.
		$$
		Applying the facts of the previous sentence to \eqref{SQ2P1} gives \eqref{SQ2A} immediately.
		
		We prove \eqref{SQ2B} the same way we proved \eqref{SQ1C}. The only nuance is the observation that $[J^{3}_K]_{4k+2N}$ sends $\tilde{U}^3_{K,l}\setminus \tilde{U}^3_{K,l+1}$ onto ${U}^3_{l}\setminus {U}^3_{l+1}$ for all $k\leq l < 4 k+2N$.
	\end{proof}

	The following proposition gives an estimate on $|D_1\tilde{G}_{K,N,m,\eta}|$ at points which are farthest away from the Cantor set in the second coordinate (index $l$), with respect to their distance from the Cantor set in the first coordinate (index $k$).
	
	\prt{Proposition}
	\begin{proclaim}\label{SQ3}
		Let $-2m\leq n_1,n_2,n_3 \leq 2m$ be even numbers, call $n = (n_1,n_2,n_3,0)$ and let $k,l\in \en_0$. Let $\ve \in \V_4^{l+1}$ and let $x\in  Q(\z_{K, \ve}, \tfrac{1}{2}\tilde{r}_l(K)) \setminus Q(\z_{K, \ve}, \tilde{r}_{l+1}(K))$ be a point such that
		$$
		\max\{|x_1 - (\z_{K, \ve})_1|,|x_3 - (\z_{K, \ve})_3|,|x_4 - (\z_{K, \ve})_4|\}  < |x_2 - (\z_{K, \ve})_2| .
		$$
		Further assume that $l\leq k \leq 4 l+4 +2N$, $x_1\in \tilde{U}^1_{K,k}\setminus \tilde{U}^1_{K,k+1}$ and $x_2\in \tilde{U}^1_{K,l}\setminus \tilde{U}^1_{K,l+1}$ and $(x_3,x_4)\in \tilde{U}^2_{K,k}$. 
		Then the following estimates holds, 
		\begin{equation}\label{SQ3A1}
			|D_1\tilde{G}_{K,N,m,\eta}(x+n)| \leq C2^{-k\beta}(K+k)^{\alpha+1}, 
		\end{equation}
		\begin{equation}
			|D\tilde{G}_{K,N,m,\eta}(x+n)| \leq C2^{-l\beta}(K+l)^{\alpha+1}, 
		\end{equation}
		and
		\begin{equation}\label{SQ3A2}
			\tilde{G}_{K,N,m,\eta}(x+n) - n \in [{U}^1_{k}\setminus {U}^1_{k+1}] \times[{U}^1_{l}\setminus {U}^1_{l+1}]\times {U}^2_{k}\text{ for }k\leq 4l+2N.
		\end{equation}		
		In case of $4l+2N<k\leq 4l+4+2N$ we get $[{U}^1_{k}\setminus {U}^1_{4l+4+2N}] \times[{U}^1_{l}\setminus {U}^1_{l+1}]\times {U}^2_{k}$ in the last inclusion. 
		
		Secondly, assume that $l\in\en_0$, $x_1\in \tilde{U}^1_{K,4l+4+2N}$, $x_2\in \tilde{U}^1_{K,l}\setminus \tilde{U}^1_{K,l+1}$ and $(x_3,x_4)\in \tilde{U}^2_{K, 4 l+4 +2N}$. Then the following estimate holds, 
		\begin{equation}\label{SQ3D1}
			|D_1\tilde{G}_{K,N,m,\eta}(x+n)| \leq C2^{-4l\beta-4\beta-2N\beta} 
		\end{equation}
		and
		\begin{equation}\label{SQ3D2}
			\tilde{G}_{K,N,m,\eta}(x+n) - n \in {U}^1_{K,4l+4+2N} \times[{U}^1_{l}\setminus \tilde{U}^1_{l+1}]\times {U}^2_{4l+4+2N}.
		\end{equation}
		These estimates are rotational in the sense that they also hold when we swap the roles of indices $x_2$, $x_3$ and $x_4$. Further the same estimate holds for $D_i\tilde{G}_{K,N,m,\eta}(x+n)$ for $i=2,3,4$ and corresponding permutation of $x$ coordinates.
	\end{proclaim}
	\begin{proof}
		To prove \eqref{SQ3A1} it suffices to notice the following facts. We are farthest from the Cantor set in the direction $x_2$ and so (see \eqref{TurnTurnTurn} and paragraph before that)
		$$
			\begin{aligned}
				\tilde{G}_{K,N,m,\eta}(x+n) &= \tilde{G}_{K,4k+2N,4k+4+2N,k}(x)+n\\ &=\tilde{\zeta}_{K,l}(x_2)[J^{3,2}_K]_{4l+2N}(x) + q_K(x_2)e_2 \\
				& \qquad+ [1-\tilde{\zeta}_{K,l}(x_2)] [J^{3,2}_K]_{4l+4+2N}(x) + n.
			\end{aligned}
		$$
		But we have $x_1\in \tilde{U}^1_{K,k}\setminus \tilde{U}^1_{K,k+1}$ and $k\leq 4l+2N$ and therefore we have
		$$
			[J^{3,2}_K]_{4l+2N}(x) = [J^{3,2}_K]_{4l+4+2N}(x) =  J^{3,2}_K(x)
		$$
		which implies 
		$$
		\tilde{G}_{K,N,m,\eta}(x+n)=q_K(x_2)e_2+J^{3,2}_K(x). 
		$$
	Now the estimates \eqref{Standard1} and \eqref{Standard2} estimate $D_1\tilde{G}_{K,N,m,\eta}$ as we desire in \eqref{SQ3A1}. The estimates of $D_3$ and $D_4$ are the same or even slightly better as $(x_3,x_4)\in \tilde{U}^2_{K,k}$ and the righthand side of \eqref{SQ3A1} is decreasing in $k$. Finally we can estimate  
$$|D_2\tilde{G}_{K,N,m,\eta}|\leq |q'_K(x_2)|\leq C (K+l)^{\alpha+1}2^{-l\beta}.$$

		The inclusion \eqref{SQ3A2} follows immediately from the fact that $ \tilde{G}_{K,4k+2N,4k+4+2N,k}(x)$ maps $\tilde{U}^4_{K,l} \setminus \tilde{U}^4_{K,l+1}$ onto ${U}^4_{l}\setminus {U}^4_{l+1}$ and on parts of hyperplanes perpendicular to $x_2$ furthest from the center of the nearest cube in direction $x_2$ we apply the map $[J^{3,2}_{K}]_{4l+2N}$ (given that $k \leq 4l+2N$) and $[J^{3,2}_{K}]_{4l+2N}$ maps each $\tilde{U}^3_{K,o} \setminus \tilde{U}^3_{K,o+1}$ onto ${U}^3_{o} \setminus {U}^3_{o+1}$ for $l\leq o\leq 4l+2N$. Especially $[J^{3,2}_{K}]_{4l+2N}$ maps $\tilde{U}^3_{K,k} \setminus \tilde{U}^3_{K,k+1}$ onto ${U}^3_{k} \setminus {U}^3_{k+1}$.
		
		The proof of \eqref{SQ3D1} is also similar but by applying \eqref{Standard3} instead of \eqref{Standard1}. Finally \eqref{SQ3D2} is proved by noticing that both $[J^3_{K}]_{4l+2N}$ and $[J^3_{K}]_{4l+4+2N}$ map $\tilde{U}^3_{K,4l+2N} $ onto ${U}^3_{4l+2N}$ and the rest of the argument remains the same as in the previous.
	\end{proof}
	
		\prt{Proposition}
	\begin{proclaim}\label{SQ5}
		Let $-2m\leq n_1,n_2,n_3 \leq 2m$ be even numbers, call $n = (n_1,n_2,n_3,0)$ and let $x\in \C_{A,K}$. Then the classic differential of $D\tilde{G}_{K,N,m,\eta}(x+n) $ exists and equals 0. Further for fixed $K,\alpha, \eta$ the map $\tilde{G}_{K,N,m,\eta}$ is Lipschitz with Lipschitz constant independent of $N,m$ and $\beta$. The map $\tilde{G}_{K,N,m,\eta}$ is locally bi-Lipschitz on $R_{m,\eta} \setminus \C_{A, K,m}$.
	\end{proclaim}
	\begin{proof}
		It is an easy observation that for each $k$ our $\tilde{G}_{K,4k+2N,4k+4+2N,k}(x)$ maps $\tilde{U}^4_{K,k}$ onto ${U}^4_{k}$ with $\Q_{K,\ve(k)}$ being mapped onto $Q_{\ve(k)}$. Now
		$$
			\frac{\diam(Q_{\ve(k)})}{\diam(\Q_{K,\ve(k)})} = \frac{r_k}{\tilde{r}_{k}(K)}  \leq C2^{-k\beta}
		$$
		and that tends to $0$ as $k\to \infty$ and therefore $D\tilde{G}_{K,N,m,\eta}(x+n) =0$ for each $x\in \C_{A,K}$. The parameter $\eta$ influences the mapping $\tilde{G}_{K,N,m,\eta}$ only outside $Z_m$. 
		For fixed $K$ and $\alpha$ the estimates in Propositions~\ref{SQ1}, \ref{SQ2}, \ref{SQ3} are decreasing in $N$ (the value of $\beta$ is a fixed constant not dependent on any of the other parameters) and we consider $N\in \en_0$ and $\beta > \beta_0\gg 1$.
		
		On each $\tilde{U}^4_{K,k} \setminus \tilde{U}^4_{K,k+1}$ we have that $\tilde{G}_{K,4k+2N,4k+4+2N,k}(x)$ is bi-Lipschitz and the derivative is smallest on sets of type $[\tilde{U}^1_{K,k} \setminus \tilde{U}^1_{K,k+1}] \times \tilde{U}^3_{K,4k+4+2N}$. Here we send 3-dimensional cubes on the hyperplane roughly of size $2^{-(4k+4+2N)}$ onto cubes roughly of size $2^{-(4k+4+2N)(\beta+1)}$ and we are linear on each cube of $ \tilde{U}^3_{K,4k+4+2N}$. Since we are bi-Lipschitz on each $\tilde{U}^4_{K,k} \setminus \tilde{U}^4_{K,k+1}$ we are locally bi-Lipschitz on $R_{m,\eta} \setminus \C_{A, K,m}$.
	\end{proof}	

\subsection{Estimates on the derivative of $G_{K,N,m,\eta}$.}\label{DSM}

	The following proposition gives an estimate on $|D_1G_{K,N,m,\eta}|$ at points which are very close to the Cantor set in coordinates $(x_2,x_3,x_4)$ compared to the $x_1$ variable.

	\prt{Proposition}
	\begin{proclaim}\label{ST1}
		Let $-2m\leq n_1,n_2,n_3 \leq 2m$ be even numbers, call $n = (n_1,n_2,n_3,0)$ and let $3k +N\leq l$ with $k,l\in \en_0$. Let $x$ be a point such that $x_1\in {U}^1_{k}\setminus {U}^1_{k+1}$, $(x_2,x_3,x_4)\in {U}^3_{l}$. Then the following estimates hold
\begin{equation}\label{ST1A}
\left|D_1G_{K,N,m,\eta}(x+n) - \frac{\tfrac{1}{2}\tilde{r}_k(K) -\tilde{r}_{k+1}(K)}{\tfrac{1}{2}r_k-r_{k+1}}e_1\right|\leq C 2^{-3k-N}2^{k(\beta+1)}, \,\,\, D_jG^1=0\,\,\text{for}\, j=2,3,4
\end{equation}
		and
		\begin{equation}\label{ST1B}
			|DG_{K,N,m,\eta}(x+n)| \leq C 2^{(3k+3+N)\beta}.
		\end{equation}
		The same holds for each rotation, i.e. for $D_2{G}^2_{K,m,\eta}(x+n)$ when $x_2\in {U}^1_{k}\setminus {U}^1_{k+1}$ and $(x_1,x_3,x_4)\in {U}^3_{l}$ and similarly for coordinates $3$ and $4$.
	\end{proclaim}
	
	\begin{proof} 
	Analogously to \eqref{starr} we obtain from \eqref{SpeciallyForStanda} that 
		$$
				G_{K,N,m,\eta}(x+n) =\zeta_{K,k}(x_1) [H^{3,1}_K]_{3k+N}(x) + t_K(x_1)e_1 
				 + [1-\zeta_{K,k}(x_1)] [H^{3,1}_K]_{3k+3+N}(x) + n.
		$$		
		The reasoning for the first part of \eqref{ST1A} is now exactly the same as in the proof of Proposition~\ref{SQ1}, where we use $|\zeta_{K,k}'|\leq 2^k 2^{\beta k}$ and 
		$$
		\bigl\|[H^{3,1}_K]_{3k+N}(x)-[H^{3,1}_K]_{3k+3+N}(x) \bigr\|_{\infty}\leq 2^{-3k-N}.
		$$
		 The other part $D_jG^1=0$ is easy to see as $G^1_{K,m,\eta}=t_k(x_1)e_1$. 
		
		The reasoning for \eqref{ST1B} is exactly the same as in the proof of Proposition~\ref{SQ3} especially \eqref{SQ3D1}. The only difference is that we work with $H$ instead of $J$ and indexes of type $3k+N$, and $3k+3+N$ instead of $4k+2N$ and $4k+4+2N$. Otherwise the arguments are identical.
	\end{proof}
	
	The following proposition gives an estimate on $|D_1G_{K,N,m,\eta}|$ at points which are closer to the Cantor set in coordinates $(x_2,x_3,x_4)$ than in the $x_1$ variable on lines parallel to $e_1$ that do not intersect the Cantor set (i.e. they go through ${U}^3_{l}\setminus {U}_{l+1}^3$) as distance from the Cantor set in the first variable decreases.
	
	\prt{Proposition}
	\begin{proclaim}\label{ST2}
		Let $-2m\leq n_1,n_2,n_3 \leq 2m$ be even numbers, call $n = (n_1,n_2,n_3,0)$ and let $k \leq l < 3k+N$ with $k,l\in \en_0$. Let $x$ be a point such that $x_1\in {U}^1_{k}\setminus {U}^1_{k+1}$, $(x_2,x_3,x_4)\in {U}^3_{l}\setminus {U}_{l+1}^3$. Then the following estimates hold
		\begin{equation}\label{ST2A}
			D_1G_{K,N,m,\eta}(x+n) = \frac{\tfrac{1}{2}\tilde{r}_k(K) -\tilde{r}_{k+1}(K)}{\tfrac{1}{2}r_k  - r_{k+1} }e_1.
		\end{equation}
		The same holds for each rotation, i.e. for $D_2G_{K,N,m,\eta}(x+n)$ when $x_2\in {U}^1_{k}\setminus {U}^1_{k+1}$ and $(x_1,x_3,x_4)\in {U}^3_{l}\setminus {U}_{l+1}^3$ and similarly for coordinates $3$ and $4$.
		Further, under the same assumptions,
		\begin{equation}\label{ST2B}
		|DG_{K,N,m,\eta}(x+n)| <C2^{\beta l}.
		\end{equation}
	\end{proclaim}
	\begin{proof}
		For \eqref{ST2A}, the reasoning is exactly the same as in the proof of Proposition~\ref{SQ2}. The only difference is that we work with $H$ instead of $J$ and indexes of type $3k+N$, and $3k+3+N$ instead of $4k+2N$ and $4k+4+2N$. Otherwise the argument is identical.
		
		The estimate \eqref{ST2B} is a direct application of result of \eqref{Standard5} because on hyperplanes $G_{K,N,m,\eta}$ is a combination of $[H^{3,1}_K]_{3k+N}$ and $[H^{3,1}_K]_{3k+3+N}(x)$ and $(x_2,x_3,x_4)\in {U}^3_{l}\setminus {U}_{l+1}^3$.
	\end{proof}
	
	\prt{Proposition}
	\begin{proclaim}\label{ST3}
		Let $-2m\leq n_1,n_2,n_3 \leq 2m$ be even numbers, call $n = (n_1,n_2,n_3,0)$ and let $k,l\in \en_0$. Assume that $l \leq k < 3 l+N$, $x_1\in {U}^1_{k}\setminus {U}^1_{k+1}$, and $(x_2,x_3,x_4)\in {U}^3_{l}\setminus {U}^3_{l+1}$. Then the following estimate holds,
		\begin{equation}\label{ST3A}
			|DG_{K,N,m,\eta}(x+n)| \leq C2^{k\beta}.
		\end{equation}
		Assume alternatively that $x_1\in {U}^1_{3l+N}$ and $(x_2,x_3,x_4)\in {U}^3_{l}\setminus {U}^3_{l+1}$ then
		\begin{equation}\label{ST3B}
			|DG_{K,N,m,\eta}(x+n)| \leq C2^{(3l+N)\beta}.
		\end{equation}
		This estimate is rotational in the sense that it also holds when we swap the roles of indices $x_1, x_2$, $x_3$ and $x_4$.
	\end{proclaim}
	\begin{proof}
		The proof of this claim follows the proof of Proposition~\ref{SQ3}, especially the proof of \eqref{SQ3D1} but in this case we use \eqref{Standard6}. The only difference that in the proof of Proposition~\ref{SQ3} the estimates of $D_3$ and $D_4$ are the same or better than the estimate of $D_1$ as $(x_3,x_4)\in \tilde{U}^2_{K,k}$ there. Here they are better than the estimate of $D_1$ as 
		$(x_3,x_4)\in {U}^2_{l}\setminus {U}^2_{l+1}$, $l\leq k$ and $2^{k\beta}$ is increasing in $k$. 
	\end{proof}

		\begin{prop}\label{ST4}
		Let $-2m\leq n_1,n_2,n_3 \leq 2m$ be even numbers, call $n = (n_1,n_2,n_3,0)$ and let $l\in \en_0$. Assume that $l\leq N-1$, $x_4\in [-2,-1]\cup[1,2]$, and $(x_1, x_2,x_3)\in {U}^3_{l}\setminus {U}^3_{l+1}$. Then the following estimate holds,
		\begin{equation}\label{ST4A}
		|DG_{K,N,m,\eta}(x+n)| \leq C2^{l\beta}.
		\end{equation}
		Assume alternatively that $(x_1, x_2,x_3)\in {U}^3_{N}$ then
		\begin{equation}\label{ST4B}
		|DG_{K,N,m,\eta}(x+n)| \leq C2^{N\beta}.
		\end{equation}
		This estimate is rotational in the sense that it also holds when we swap the roles of indices $x_1, x_2$, $x_3$ and $x_4$.
	\end{prop}
	\begin{proof}
		The proof of \eqref{ST4A} is a simple application of \eqref{Standard5} (see also \eqref{SpeciallyForStanda} and \eqref{ASeason0}) and the proof of \eqref{ST4B} is a simple application of \eqref{Standard6}.
	\end{proof}

	\prt{Proposition}
	\begin{proclaim}\label{ST0}
		The map $G_{K,N,m,\eta}$ is $C\cdot  2^{M\beta}$-Lipschitz on $R_{m, 13}\setminus \T$, where $\T$ is the set from Lemma~\ref{ThatTimeOfTheMonth}. The map $G_{K,N,m,\eta}$ is locally bi-Lipschitz on $R_{m,13} \setminus \C_{B,m}$.
	\end{proclaim}
	\begin{proof}
		The proof that $G_{K,N,m,\eta}$ is locally bi-Lipschitz on $R_{m,13} \setminus \C_{B,K,m}$ follows the same argument as given in the proof of Proposition~\ref{SQ5} but using Propositions~\ref{ST1}, \ref{ST2}, \ref{ST3} instead of the propositions in Section~\ref{DFM}. 
		
		Note that $|DG_{K,N,m,\eta}|\leq C$ when $|x_4|\in [2,14]$ in \eqref{StretchingBitAroundTop} or when $|x_i|\in[2m+2,2m+5]$ in \eqref{StretchingBitAroundTop2}. Therefore it is a simple application of \eqref{Standard5} to prove that $G_{K,N,m,\eta}$ is $C\cdot 2^{M\beta}$Lipschitz on $\R_{m,13} \setminus \T$.
	\end{proof}

	Now we improve on Proposition~\ref{EmotionallyDistant} and we estimate the position of the composition with $\tilde{G}_{K,N,m,\eta}$ from Section \ref{firstmap}.
	
	\begin{lemma}\label{WrongPlaceWrongTime}
		Let $-2m \leq n_1,n_2,n_3 \leq 2m$ be even numbers and $n=(n_1,n_2,n_3,0)$. There exists an $M^* = M^*(\beta)\geq M$ (where $M$ is from Proposition \ref{EmotionallyDistant}) such that for all $k,l,\in \en_0$ such that $l \geq k + M^*$ and all $x_1\in U_{k}^1\setminus U_{k+1}^1$ and $(x_2,x_3,x_4) \in U^3_{l}$, we have
		\begin{equation}\label{BetterX1}
			\big[\hat{F}_{\beta, m}\circ \tilde{G}_{K,N,m,\eta}(x+n) -n\big]_1 \in U^{1}_{k-1}\setminus U^{1}_{k+2}.
		\end{equation}
	\end{lemma}
	\begin{proof} 
		The image of $x$ in $\tilde{G}_{K,N,m,\eta}$ is in the set $({U}_{k}^1\setminus {U}_{k+1}^1)\times {U}^3_{l}$ (see Fig. \ref{Fig:Places} and \eqref{SQ2B}). The set $[-1,1] \times {U}^3_{l}$ is a neighbourhood of $[-1,1]\times C_{B}(3)$ and the diameter of each cube in ${U}^3_{l}$ is less than $C2^{-l(\beta+1)}$. Call the linear map $L:(x_1,x_2,x_3,x_4)\to (x_1,x_2,x_3,-x_4)$. Now on $[-1,1]\times C_{B}(3)$ we have \eqref{NewNew} and therefore because $\hat{F}_{\beta, m}$ is Lipschitz we have that
		$$
			\big|\big[\hat{F}_{\beta, m}\circ \tilde{G}_{K,N,m,\eta}(x+n)\big]_1 -  \big[L\circ\tilde{G}_{K,N,m,\eta}(x+n)\big]_1\big|< C2^{-l(\beta+1)}.
		$$
		Now $l \geq k+M^*$ and given that $M^*$ is large enough we have that the right hand side is much smaller than the size of the intervals in $U^{1}_{k}$, which have diameter $\approx 2^{-k(\beta +1)}$. This immediately yields \eqref{BetterX1}. 		
			\end{proof}

\section{Estimates of the derivative of compositions}\label{PardubiceJeJenVesniceUHradce}
See the beginning of the previous section for the role of various parameters. In this section we further use fixed constants $M$ from Proposition \ref{EmotionallyDistant} 
(see also Theorem \ref{BloodSweatAndTears} as $M\geq M_0$) and $M^*$ from Lemma \ref{WrongPlaceWrongTime}.

\subsection{The composition of $\hat{F}_{\beta,m} \circ \tilde{G}_{K,N,m,\eta}$.}\label{Composition}
	Consider a segment very close to being parallel to $e_i$ (as is the image of segments parallel to $e_i$ in $\tilde{G}_{K,N,m,\eta}$ by \eqref{SQ1A} and \eqref{SQ1B}). The image of this curve in $P_{v}$ onto $\er^3\times\{0\}$ is very close to a segment parallel to $e_i$ since $g$ is constant on segments parallel to $e_i$ near the projection of $\K_{B,m}$ and since it is Lipschitz, the function is very close to being constant on the projection of the image of the segment. Therefore the image of segments very close to being parallel to $e_i$ in $\hat{F}_{\beta, m}$ is a curve, which has parametrization whose derivative has very small components in all directions except the $i$ component. That is the subject of Proposition~\ref{C1}.
	\prt{Proposition}
	\begin{proclaim}\label{C1}
		Let $-2m\leq n_1,n_2,n_3 \leq 2m$ be even numbers, call $n = (n_1,n_2,n_3,0)$ and let $l, k\in \en$. Assume that $2N>M^*$ and $M\leq M^*\leq l$, $ 2N+4k\leq l$ and let $x$ be a point such that $x_1\in \tilde{U}^1_{K,k}\setminus \tilde{U}^1_{K,k+1}$, $(x_2,x_3,x_4)\in \tilde{U}^3_{K,l}$. Then
		\begin{equation}\label{C1A}
			\hat{F}_{\beta, m}\circ \tilde{G}_{K,N,m,\eta}(x+n) - n \in [{U}^1_{k-1}\setminus {U}^1_{k+2}]\times[{U}^3_{4k+2N-M}\setminus{U}_{4k+2N+M+1}^3]
		\end{equation}
		and the following estimate holds
		\begin{equation}\label{C1B}
			\Big|D_1(\hat{F}_{\beta, m}\circ\tilde{G}_{K,N,m,\eta}(x+n)) - \frac{\tfrac{1}{2}r_k -r_{k+1}}{\tfrac{1}{2}\tilde{r}_k(K) - \tilde{r}_{k+1}(K) } e_1 \Big|  \leq C2^{-4k\beta-2N\beta}(K+k)^{\alpha+1}.
		\end{equation}
		All of the above holds for each rotation, i.e. for $D_2\hat{F}_{\beta, m}\circ \tilde{G}_{K,N,m,\eta}(x+n)$ when $x_2\in \tilde{U}^1_{K,k}\setminus \tilde{U}^1_{K,k+1}$ and $(x_1,x_3,x_4)\in \tilde{U}^3_{K,l}$ and similarly for coordinates $3$ and $4$.
	\end{proclaim}
	\begin{proof}
		The claim \eqref{C1A} is derived easily from 
		\eqref{SQ1C}, Proposition~\ref{EmotionallyDistant} and \eqref{BetterX1} of Lemma~\ref{WrongPlaceWrongTime} (note that $2N>M^*$ easily implies that $k+M^* \leq4k+M^*< 4k+2N <  l$). 
		
		Now the $i$-th coordinate of the derivative equals
		$$
		D_1(\hat{F}_{\beta, m}\circ\tilde{G}_{K,N,m,\eta}(x))^i=
		\sum_{j=1}^4 D_j(\hat{F}_{\beta, m})^i(\tilde{G}_{K,N,m,\eta}(x)) D_1(\tilde{G}_{K,N,m,\eta})^j(x).
		$$
		By \eqref{StraightStraight} we know that $D_1\hat{F}_{\beta, m}(\tilde{G}_{K,N,m,\eta}(x))=e_1$ (recall that $M_0\leq M$) and hence the first term on the righthand side (corresponding to $j=1$) is nonzero only for $i=1$. By \eqref{SQ1A} we know that this first term for $i=1$ equals to 
		$\frac{\frac{1}{2}r_k -r_{k+1}}{\frac{1}{2}\tilde{r}_k(K) - \tilde{r}_{k+1}(K) }$. Other term for all $i\in\{1,2,3,4\}$ can be estimated by \eqref{SQ1B} and the fact that $\hat{F}_{\beta,m}$ is Lipschitz and we get \eqref{C1B}.  
	\end{proof}

	The following proposition gives an estimate on $|D_1\hat{F}_{\beta, m}\circ\tilde{G}_{K,N,m,\eta}|$ at points which are closer to the Cantor set in coordinates $(x_2,x_3,x_4)$ than in the $x_1\in \tilde{U}^1_{K,k}\setminus \tilde{U}^1_{K,k+1}$ variable on lines that do not intersect the Cantor set (i.e. they go through $\tilde{U}^3_{K,l}\setminus \tilde{U}_{K,l+1}^3$) as distance from the Cantor set in the first variable decreases. The key here is that $\tilde{G}_{K,N,m,\eta}$ sends segments parallel to $e_1$ onto segments parallel to $e_1$ and so does $\hat{F}_{\beta, m}$.
	
	\prt{Proposition}
	\begin{proclaim}\label{C2}
		Let $-2m\leq n_1,n_2,n_3 \leq 2m$ be even numbers, call $n = (n_1,n_2,n_3,0)$ and let $k,l\in \en$. Assume that $M\leq M^* \leq l$ and $k+M +M^*+1 < l< 2N+4k$ and let $x$ be a point such that $x_1\in \tilde{U}^1_{K,k}\setminus \tilde{U}^1_{K,k+1}$, $(x_2,x_3,x_4)\in \tilde{U}^3_{K,l}\setminus \tilde{U}_{K,l+1}^3$. Then
		\begin{equation}\label{C2A}
			\hat{F}_{\beta, m}\circ \tilde{G}_{K,N,m,\eta}(x+n) - n \in [{U}^1_{k-1}\setminus {U}^1_{k+2}]\times[{U}^3_{l-M}\setminus {U}_{l+M+1}^3].
		\end{equation}
		and
		\begin{equation}\label{C2B}
			D_1\hat{F}_{\beta, m}\circ\tilde{G}_{K,N,m,\eta}(x+n) = \frac{\tfrac{1}{2}r_k -r_{k+1}}{\tfrac{1}{2}\tilde{r}_k(K) - \tilde{r}_{k+1}(K) }e_1.
		\end{equation}
		The same holds for each rotation, i.e. for $D_2\tilde{G}_{K,N,m,\eta}(x+n)$ when $x_2\in \tilde{U}^1_{K,k}\setminus \tilde{U}^1_{K,k+1}$ and $(x_1,x_3,x_4)\in \tilde{U}^3_{K,l}\setminus \tilde{U}_{K,l+1}^3$ and similarly for coordinate $3$. The difference in the $4^{\text{th}}$ coordinate is that in \eqref{BA}, \eqref{BB} and \eqref{BC} the derivative a positive multiple of $-e_4$, but the estimates are otherwise the same.
	\end{proclaim}
	\begin{proof}
		The proof is similar to Proposition~\ref{C1} with the difference that we use Proposition~\ref{SQ2} specifically \eqref{SQ2A} to get that segments parallel to $e_1$ are mapped to segments parallel to $e_1$ by $\tilde{G}_{K,N,m,\eta}$ and the same is true for $\hat{F}_{\beta, m}$ by \eqref{StraightStraight}, thus \eqref{C2B} is proved. The reasoning for \eqref{C2A} is identical to that of \eqref{C1A}.
	\end{proof}

	\prt{Proposition}
	\begin{proclaim}\label{C3}
		Let $-2m\leq n_1,n_2,n_3 \leq 2m$ be even numbers, call $n = (n_1,n_2,n_3,0)$ and let $k,l\in \en$. Assume that $1\leq l-M-M^*-1 \leq k \leq 3l-2M-2+N$ and let $x$ be a point such that $x_1\in \tilde{U}^1_{K,k}\setminus \tilde{U}^1_{K,k+1}$, $(x_2,x_3,x_4)\in \tilde{U}^3_{K,l}\setminus \tilde{U}_{K,l+1}^3$. Then
		\begin{equation}\label{C3A}
			\hat{F}_{\beta, m}\circ \tilde{G}_{K,N,m,\eta}(x+n) - n \in [{U}^1_{k-M}\setminus {U}^1_{k+M+1}]\times [{U}^3_{l-M}\setminus {U}^3_{l+M+1}].
		\end{equation}
		In the case that, $x_1\in \tilde{U}^1_{K,3l-2M-2+N}$, and $(x_2, x_3,x_4)\in \tilde{U}^3_{K,l}\setminus \tilde{U}_{K,l+1}^3$ then
		\begin{equation}\label{C3B}
			\hat{F}_{\beta, m}\circ \tilde{G}_{K,N,m,\eta}(x+n) - n \in  [{U}^4_{l-M}\setminus{U}^4_{l+M+1}].
		\end{equation}
		These statements are rotational in the sense that they also hold when we swap the roles of the indices.
	\end{proclaim}
	\begin{proof}
		The claim \eqref{C3A} follows from Proposition \ref{SQ3} and Proposition~\ref{EmotionallyDistant}. The inclusion $\eqref{C3B}$ is exactly the second claim of Proposition~\ref{EmotionallyDistant}.
	\end{proof}
	
\subsection{Derivative estimates of $G_{K,N,m,\eta} \circ \hat{F}_{\beta, m} \circ \tilde{G}_{K,N,m,\eta}$.}\label{BigComposition}

	\prt{Proposition}
	\begin{proclaim}\label{B}
		Let $-2m\leq n_1,n_2,n_3 \leq 2m$ be even numbers, call $n = (n_1,n_2,n_3,0)$ and let $k,l\in \en$. Assume that $N\geq 3M^*+3M+6$. Firstly let $M^*\leq l$ and  $k \leq \tfrac{1}{4}(l-2N)$ and let $x$ be a point such that $x_1\in \tilde{U}^1_{K,k}\setminus \tilde{U}^1_{K,k+1}$, $(x_2,x_3,x_4)\in \tilde{U}^3_{K,l}$. Then 
		\begin{equation}\label{BA}
				|D_1G_{K,N,m,\eta}\circ \hat{F}_{\beta, m}\circ \tilde{G}_{K,N,m,\eta}(x+n)| \leq C 2^{\beta +2}+C 2^{\beta}2^{-3k-N}(K+k)^{\alpha+1}  .
		\end{equation}
		Secondly, let $M^* \leq l$ and $\tfrac{1}{4}(l-2N) < k \leq \tfrac{1}{3}(l-N-M)-1$ and let $x$ be a point such that $x_1\in \tilde{U}^1_{K,k}\setminus \tilde{U}^1_{K,k+1}$, $(x_2,x_3,x_4)\in \tilde{U}^3_{K,l}$. Then
			\begin{equation}\label{BB}
			|D_1G_{K,N,m,\eta}\circ \hat{F}_{\beta, m}\circ \tilde{G}_{K,N,m,\eta}(x+n)| \leq C 2^{\beta +2}  +  C2^{-3k-N}(K+k)^{\alpha+1}.
			\end{equation}
		
		In the intermediary case assume that $M^* \leq l $ and $ \tfrac{1}{3}(l-N-M)-1< k \leq  \tfrac{1}{3}(l-N)+1$ and let $x$ be a point such that $x_1\in \tilde{U}^1_{K,k}\setminus \tilde{U}^1_{K,k+1}$, $(x_2,x_3,x_4)\in \tilde{U}^3_{K,l}$. Then
			\begin{equation}\label{BBC}
			|D_1G_{K,N,m,\eta}\circ \hat{F}_{\beta, m}\circ \tilde{G}_{K,N,m,\eta}(x+n) | \leq +C 2^{\beta +2}  +  C2^{-3k-N}(K+k)^{\alpha+1}.
			\end{equation}
		In the third case assume that $M^* \leq l $ and $ \tfrac{1}{3}(l-N) +1< k \leq  l-M-M^*-1$ and let $x$ be a point such that $x_1\in \tilde{U}^1_{K,k}\setminus \tilde{U}^1_{K,k+1}$, $(x_2,x_3,x_4)\in \tilde{U}^3_{K,l}$. Then
			\begin{equation}\label{BC}
			|D_1G_{K,N,m,\eta}\circ \hat{F}_{\beta, m}\circ \tilde{G}_{K,N,m,\eta}(x+n) | \leq C 2^{\beta}.
			\end{equation}
		
		Further, for $M^*\leq l,k$ and $l -M- M^*-1 < k\leq 3l - 4M+N-1$ and $x_1\in \tilde{U}^1_{K,k}\setminus \tilde{U}^1_{K,k+1}$, and $(x_2, x_3,x_4)\in \tilde{U}^3_{K,l}\setminus \tilde{U}_{K,l+1}^3$
		\begin{equation}\label{BD}
		|D_1G_{K,N,m,\eta}\circ \hat{F}_{\beta, m}\circ \tilde{G}_{K,N,m,\eta}(x+n)| \leq  C2^{M\beta}(K+k)^{\alpha+1}.
		\end{equation}

		Given that $M\leq l$ and $3l - 4M+ N \leq k \leq 4l+4+2N $ and $x_1\in \tilde{U}^1_{K,k}\setminus \tilde{U}^1_{K,k+1}$, and $(x_2, x_3,x_4)\in \tilde{U}^3_{K,l}\setminus \tilde{U}_{K,l+1}^3$, if $x_1\in \tilde{U}^1_{K,3l-3M-2}$, and $(x_2, x_3,x_4)\in \tilde{U}^3_{K,l}\setminus \tilde{U}_{K,l+1}^3$, then 
		\begin{equation}\label{BE}
			|D_1G_{K,N,m,\eta}\circ \hat{F}_{\beta, m}\circ \tilde{G}_{K,N,m,\eta}(x+n)| \leq  C2^{(3l+3M+N-k)\beta} (K+k)^{\alpha+1}.
		\end{equation}
		Finally, given that $M\leq l$ and $x_1\in \tilde{U}^1_{K,4l+5+2N}$, and $(x_2, x_3,x_4)\in \tilde{U}^3_{K,l}\setminus \tilde{U}_{K,l+1}^3$, then
		\begin{equation}\label{BF}
			|D_1G_{K,N,m,\eta}\circ \hat{F}_{\beta, m}\circ \tilde{G}_{K,N,m,\eta}(x+n)| \leq C2^{(3M-l-N)\beta }.
		\end{equation}
		The same holds for each rotation, i.e. for $D_2\hat{F}_{\beta, m}\circ \tilde{G}_{K,N,m,\eta}(x+n)$ when $x_2\in \tilde{U}^1_{K,k}\setminus \tilde{U}^1_{K,k+1}$ and $(x_1,x_3,x_4)\in \tilde{U}^3_{K,l}$ and similarly for $3$. The difference in the $4^{\text{th}}$ coordinate is that in \eqref{BA}, \eqref{BB} and \eqref{BC} the derivative is close to $-e_4$, but the estimates are otherwise the same.
	\end{proclaim}
	\begin{proof}
		
		Note that the cases are all non-empty. This is obvious once one uses $N\geq 3M^*+3M+6$; we have 
			$$
			\tfrac{1}{4}(l-2N)<\tfrac{1}{3}(l-N-M-3) < \frac{1}{3}(l-N)+1< l-M-M^*-1.
			$$

	By the chain rule we have
	\eqn{chainrule}
$$
D_1 (G\circ \tilde{F}\circ \tilde{G})= \sum_{j=1}^{4}D_j G\cdot D_1( \tilde{F}\circ \tilde{G})^j
=D_1 G\cdot D_1 (\tilde{F}\circ \tilde{G})^1+  \sum_{j=2}^{4}D_j G\cdot D_1( \tilde{F}\circ \tilde{G})^j.
$$
	
		Let us first prove \eqref{BA}. By \eqref{C1A} we know that 
		$$ 
		\hat{F}_{\beta, m}\circ \tilde{G}_{K,N,m,\eta}(x)\in 
		[{U}^1_{k-1}\setminus {U}^1_{k+2}]\times[{U}^3_{4k+2N-M}\setminus{U}_{4k+2N+M+1}^3]. 
		$$
		Hence we can use \eqref{ST1A} at this point to obtain 
		$$
|D_1 G_{K,N,m,\eta}| \leq 
 \frac{\tfrac{1}{2}\tilde{r}_{k+1}(K)-\tilde{r}_{k+2}(K)}{\tfrac{1}{2}r_{k+1}-r_{k+2}}+ C2^{-3k-N}2^{k(\beta+1)}
$$
and from \eqref{C1B} we obtain 
$$
|D_1 (\hat{F}_{\beta, m}\circ \tilde{G}_{K,N,m,\eta})^1|\leq 
\frac{\tfrac{1}{2}r_{k}-r_{k+1}}{\tfrac{1}{2}\tilde{r}_k(K)-\tilde{r}_{k+1}(K)}+ C 2^{-4\beta k-2N\beta}(K+k)^{\alpha+1}. 
$$
This allows us to estimate the first term on the righthand side of \eqref{chainrule}. 
We estimate the other terms in \eqref{chainrule} using \eqref{C1B} and \eqref{ST1B} 
$$
|D_j G_{K,N,m,\eta}|\cdot |D_1 (\hat{F}_{\beta, m}\circ \tilde{G}_{K,N,m,\eta})^j|\leq 
C 2^{-4k\beta-2N\beta}(K+k)^{\alpha+1}\cdot 2^{(3(k+1)+N)\beta}. 
$$ 
Inequality \eqref{BA} follows leaving out minor order terms.

		We prove \eqref{BB} much the same way as \eqref{BA}. 
			We have that $\tfrac{1}{4}(l-2N)<k\leq\tfrac{1}{3}(l-M-N-3)<l-M-M^*-1$ because $N\geq 3M^*+3M$. Therefore we apply \eqref{C2B} and \eqref{C2A} and get
			$$
			D_1\hat{F}_{\beta, m}\circ\tilde{G}_{K,N,m,\eta}(x+n) = \frac{\tfrac{1}{2}r_k -r_{k+1}}{\tfrac{1}{2}\tilde{r}_k(K) - \tilde{r}_{k+1}(K) }e_1
			$$
			with $\hat{F}_{\beta, m}\circ\tilde{G}_{K,N,m,\eta}(x+n) \in [{U}^1_{k-1}\setminus {U}^1_{k+2}]\times[{U}^3_{l-M}\setminus {U}_{l+M+1}^3]$. Call $k'=k-1,k,k+1$ and $l' = l-M$ and apply \eqref{ST1A} for $k'$ and $l'$ (we need $3k'+N\leq l'$, the strictest condition is $3k+3+N \leq l-M$ which is $k\leq\tfrac{1}{3}(l-M-N-3)$). We calculate $D_1G_{K,N,m,\eta}\circ \hat{F}_{\beta, m}\circ \tilde{G}_{K,N,m,\eta}(x+n)$ using the chain rule, specifically by \eqref{C2B} we have $\sum_{j=2}^4$ in \eqref{chainrule} is now zero and therefore
			$$
			\begin{aligned}
			|D_1G_{K,N,m,\eta}\circ \hat{F}_{\beta, m}\circ \tilde{G}_{K,N,m,\eta}(x+n)|
			&< 2^{\beta +2} + C\frac{\tfrac{1}{2}r_k -r_{k+1}}{\tfrac{1}{2}\tilde{r}_k(K) - \tilde{r}_{k+1}(K) }2^{-3k-N}2^{(k+1)\beta}\\
			&< 2^{\beta +2} + C2^{-3k-N}(K+k)^{\alpha+1}.
			\end{aligned}
			$$

		
	The calculation in \eqref{BC} is even simpler. We apply \eqref{C2B} (recall that $N\geq M+M^*$) to calculate $D_1\hat{F}_{\beta, m}\circ \tilde{G}_{K,N,m,\eta}(x+n)$
	and the image lies in $\bigcup_{i=-1}^{+1} (U_{k+i}\setminus U_{k+i+1})\times U_{l-M}$ by \eqref{C2A} (because $k+1\leq l-M-M^*$) and Proposition~\ref{EmotionallyDistant}. Therefore we put $l' = l-M$ and $k' = k-1,k, k+1$. Then from the condition $\tfrac{1}{3}(l-N) +1< k \leq  l-M-M^*-1$ we get $\tfrac{1}{3}(l'-N)< \tfrac{1}{3}(l'+M-N) < k -1\leq k' \leq k+1 \leq l'-M^*<l'$ then $k'<l'<3k'+N$, which is the condition from Proposition~\ref{ST2}. Therefore we can apply \eqref{ST2A} (again after using the fact that $\sum_{j=2}^4=0$ in \eqref{chainrule}) to get
	$$
	\begin{aligned}
	&|D_1G_{K,N,m,\eta}\circ \hat{F}_{\beta, m}\circ \tilde{G}_{K,N,m,\eta}(x+n)|\\
	&\qquad\qquad\qquad< \frac{\tfrac{1}{2}r_k -r_{k+1}}{\tfrac{1}{2}\tilde{r}_k(K) - \tilde{r}_{k+1}(K) }
	\cdot \frac{\tfrac{1}{2}\tilde{r}_{k+1}(K)  - \tilde{r}_{k+2}(K) }{\tfrac{1}{2}{r}_{k+1} -{r}_{k+2}}\\
	&\qquad\qquad\qquad< C 2^{\beta}.
	\end{aligned}
	$$

	The intermediary case in \eqref{BBC} is similar. We should either apply \eqref{ST1A} or \eqref{ST2A} with the calculations in the first case identical to those from the proof of \eqref{BB} and the calculations in the second case identical to those from the proof of \eqref{BC}. Since the estimate in \eqref{BB} is larger we estimate using that one.

		We prove \eqref{BD} as follows. In the case that $k\leq l$ we estimate $|D_1 \tilde{G}_{K,N,m,\eta}(x+n)|$ by \eqref{SQ2A} and the right hand side of \eqref{SQ1A}. In the case that $k>l$ we estimate $|D_1 \tilde{G}_{K,N,m,\eta}(x+n)|$ using \eqref{SQ3A1}. The map $\hat{F}_{\beta, m}$ is bi-Lipschitz and we know the position of the image of the point $ \hat{F}_{\beta, m}\circ \tilde{G}_{K,N,m,\eta}(x+n)$ by \eqref{C3A} with $k-M\leq k'\leq k+M$ and $l-M\leq l'\leq l+M$. Then $k\leq3l-4M+N-1$ implies $k' \leq k+M \leq 3(l-M)+N+1 = 3l' + N +1$, which means we are able to apply \eqref{ST3A} to estimate $|DG_{K,N,m,\eta}|$ at that point in the image. We get
		$$
			\begin{aligned}
				|D_1G_{K,N,m,\eta}\circ \hat{F}_{\beta, m}\circ \tilde{G}_{K,N,m,\eta}(x+n)|
				&< C2^{-k\beta} (K+k)^{\alpha+1}2^{k\beta+ M\beta}\\
				& = C2^{M\beta}(K+k)^{\alpha+1}.
			\end{aligned}
		$$


		To get \eqref{BE} we estimate $|D_1 \tilde{G}_{K,N,m,\eta}(x+n)|$ using \eqref{SQ3A1}. The map $\hat{F}_{\beta, m}$ is bi-Lipschitz and by Proposition~\ref{EmotionallyDistant} we have 
		$$
		\hat{F}_{\beta, m}\circ \tilde{G}_{K,N,m,\eta}(x+n)-n \in  (U_{k-M}\setminus U_{k+M})\times(U_{l-M}^3 \setminus U_{l+M}^3) .
		$$
		Obviously the estimate of $|DG_{K,N,m,\eta}|$ from \eqref{ST3A} is greater than that in \eqref{ST3B} for all $k<3l+N$. Therefore we use \eqref{ST3A} with $l' = l+M$ to estimate $|DG_{K,N,m,\eta}| < C2^{(3l+3M+N)\beta}$ for all $3l - 4M+ N \leq k \leq 4l+4+2N$. This gives
		$$
			\begin{aligned}
				|D_1G_{K,N,m,\eta}\circ \hat{F}_{\beta, m}\circ \tilde{G}_{K,N,m,\eta}(x+n)|
				&< C2^{-k\beta} (K+k)^{\alpha+1}2^{(3l+3M+N)\beta}\\
				&  = C2^{(3l+3M+N-k)\beta} (K+k)^{\alpha+1}.
			\end{aligned}
		$$

		The estimate \eqref{BF} is proved simply by applying \eqref{SQ3D1}, \eqref{C3B} and \eqref{ST3B} to get
		$$
			\begin{aligned}
				|D_1G_{K,N,m,\eta}\circ \hat{F}_{\beta, m}\circ \tilde{G}_{K,N,m,\eta}(x+n)|
				&< C2^{-(4l+2N)\beta} 2^{(3l+3M+N)\beta}\\
				&  = C2^{(3M-l-N)\beta}.
			\end{aligned}
		$$	
	\end{proof}
	
\subsection{ACL condition and norm estimates of $G_{K,N,m,\eta}\circ \hat{F}_{\beta, m}\circ\tilde{G}_{K,N,m,\eta}$.}\label{NormEstimates}

Recall that parameter $m$  denotes the size of the boxes (see \eqref{defm}) and that $Z_m  =[-2m-1,2m+1]^3\times[-1, 1]$. 

	\begin{prop}\label{IntegralEstimate}
		Let $1\leq p <2$, let $M^*\geq M$ be the numbers from Lemma~\ref{WrongPlaceWrongTime} and Proposition \ref{EmotionallyDistant} and let $\alpha\geq \frac{4}{2-p}$. 
			There exists an $N_0 \geq M^*$ such that when $2K\geq N\geq 3M+K\geq K\geq N_0 $ and $N\geq 3M^*+3M+6$ we have
		\begin{enumerate}
			\item[$i)$] $
			G_{K,N,m,\eta}\circ \hat{F}_{\beta, m}\circ \tilde{G}_{K,N,m,\eta} \in W^{1,p}((-2m-1,2m+1)^3\times(-1, 1) , \er^4)$.
			\item[$ii)$] For every $j\in\{1,2,3,4\}$
			$$
				\begin{aligned}
				\int_{Z_m} |D_j G_{K,N,m,\eta}&\circ \hat{F}_{\beta, m}\circ \tilde{G}_{K,N,m,\eta}|^p \leq\\
				&\leq C(\alpha, \beta, p) \Big(m^3+ m^2K^{(\alpha+1)(p-1)}  + m^3 (1+ \eta^p K^{p(\alpha+1)})K^{-\alpha-1}\Big).
				\end{aligned}
			$$
			\item[$iii)$] There exists a $K_0(\alpha, \beta, p)$ such that if $K\geq K_0$, then
			$$
				\begin{aligned}
				&\int_{Z_m\setminus \C_{A, K,m}} |D_j G_{K,N,m,\eta}\circ \hat{F}_{\beta, m}\circ \tilde{G}_{K,N,m,\eta}|^p\leq \\
				&\quad \leq C(\alpha, \beta, p) \Big( m^3 K^{2-(\alpha+1)(2-p)} + m^2 K^{(\alpha+1)(p-1)} + m^3 (1+ \eta^p K^{p(\alpha+1)})K^{-\alpha-1} \Big).
				\end{aligned}
			$$
		\end{enumerate}
	\end{prop}
	\begin{proof}
		
		\step{1}{Absolute continuity on lines}{PIES1}
		
			By Proposition~\ref{SQ5} we have that $\tilde{G}_{K,N,m,\eta}$ is locally bi-Lipschitz on $Z_{m} \setminus \C_{A,K,m}$. By Theorem~\ref{reflect} we have that $\hat{F}_{\beta, m}$ is a bi-Lipschitz map which sends $\C_{B,m}$ onto $\C_{B,m}$. Proposition~\ref{ST0} implies that $G_{K,N,m,\eta}$ is locally bi-Lipschitz on $Z_{m}\setminus \C_{B,m}$. This together means that $G_{K,N,m,\eta}\circ \hat{F}_{\beta, m}\circ \tilde{G}_{K,N,m,\eta}$ is locally bi-Lipschitz on $Z_{m}\setminus  \C_{A,K,m}$. Therefore $G_{K,N,m,\eta}\circ \hat{F}_{\beta, m}\circ \tilde{G}_{K,N,m,\eta}$ is absolutely continuous on lines parallel to coordinate axes that do not intersect $\C_{A, K,m}$.
		
			On those lines that do intersect $\C_{A, K,m}$ we apply the following facts. We have that $\tilde{G}_{K,N,m,\eta}(x) = (q_K(x_1), q_K(x_2), q_K(x_3), q_K(x_4))$ on $\C_{B}$, we have \eqref{NewNew} and $G_{K,N,m,\eta}(x) = (t_K(x_1), t_K(x_2), t_K(x_3), t_K(x_4))$. Thus 
\eqn{hhh}
			$$
			G_{K,N,m,\eta}\circ \hat{F}_{\beta, m}\circ \tilde{G}_{K,N,m,\eta}(x) = (x_1,x_2,x_3,-x_4)\text{ on }\C_{A,K,m}
			$$ 
			and 
			$$
			G_{K,N,m,\eta}\circ \hat{F}_{\beta, m}\circ \tilde{G}_{K,N,m,\eta}(x+n) = G_{K,N,m,\eta}\circ \hat{F}_{\beta, m}\circ \tilde{G}_{K,N,m,\eta}(x) +n.
			$$ 
			Further by Proposition~\ref{B} (specifically \eqref{BA}) we see that $G_{K,N,m,\eta}\circ \hat{F}_{\beta, m}\circ \tilde{G}_{K,N,m,\eta}$ is continuous and has bounded derivative along those lines through $\C_{A, K,m}$. Therefore we have that $G_{K,N,m,\eta}\circ \hat{F}_{\beta, m}\circ \tilde{G}_{K,N,m,\eta}$ is absolutely continuous on all lines parallel to coordinate axes. To prove $i)$ it now suffices to prove the integral estimate $ii)$.
			
			We estimate the integral over $Z_m$ by integrating over $Q(0,1)$ and multiplying by the number of cubes $(2m+1)^3$. The integral over $Q(0,1)$ is decomposed further into the integral over $\tilde{U}_{K,0}^4\setminus \tilde{U}_{K,M^*}^4$ and the integral over $\tilde{U}_{K,M^*}^4$. We now proceed to estimate the integral over $\tilde{U}_{K,M^*}^4$ in the following two steps.
		
		\step{2}{Integral estimates on lines not intersecting $\C_{A, K,m}$}{PIES2}
		
			We have to integrate the estimates from Proposition~\ref{B} over their corresponding lines and then multiply by $(2m+1)^3$. We assume that we are working on a line parallel to $e_1$, although for other lines the estimates work in the same way. Call $\ell$ the intersection of this line with $Q(0,1)$ and call $\ell_k$ the subset of $\ell$ such that $x_1 \in \tilde{U}^1_{K,k}\setminus \tilde{U}^1_{K,k+1}$. Also we assume that the line we are working on is farther from $\C_{A,K,m}(3)$ in direction $x_2$ than in directions $x_3$ and $x_4$. Thanks to the symmetry of the map we use these estimates to calculate in the other cases. Therefore, choose $l\in \en, l\geq M^*$ and assume that $(x_2,x_3,x_4) \in \tilde{U}^3_{K,l}\setminus \tilde{U}^3_{K,l+1}$. In the following sums we sum between the maximum of the lower bound written and $M^*$ and the upper bound written (if it is larger than $M^*$, otherwise the sum is empty), but we simplify the notation by excluding the maximum. We calculate
			$$
			\begin{aligned}
				\int_{\ell\cap\tilde{U}_{K,M^*}^4}|D_1 G_{K,N,m,\eta}\circ& \hat{F}_{\beta, m}\circ \tilde{G}_{K,N,m,\eta} (t)|^p d\mathcal L^1(t) = \\
				=\sum_{M^*\leq k \leq 4l+4+2N}& \int_{\ell_k} |D_1 G_{K,N,m,\eta}\circ \hat{F}_{\beta, m}\circ \tilde{G}_{K,N,m,\eta} (t)|^p d\mathcal L^1(t)\\
				+& \int_{\ell\cap \{x_1\in\tilde{U}^1_{K,4l+2N+5}\}} |D_1 G_{K,N,m,\eta}\circ \hat{F}_{\beta, m}\circ \tilde{G}_{K,N,m,\eta} (t)|^p d\mathcal L^1(t).
			\end{aligned}
			$$
Since $N\geq K$ we easily have $2^{-3k-N}(K+k)^{\alpha+1}\leq C$.   	
			Hence we estimate by \eqref{BA}, given $N \geq 3M + K,$ and $N\geq N_0$, (and by \eqref{DistancesEstimate1} we have that $\mathcal{L}^1(\tilde{U}_{K,k}^1\setminus \tilde{U}_{K,k+1}^1) \leq C (K+k)^{-\alpha-1}$) that
			\eqn{odkaz}
			$$
				\begin{aligned}
					\sum_{M^*\leq k \leq \tfrac{1}{4}(l-2N)}& \int_{\ell_k} |D_1 G_{K,N,m,\eta}\circ \hat{F}_{\beta, m}\circ \tilde{G}_{K,N,m,\eta} (t)|^p d \mathcal L^1(t) \\
					&\leq 	\sum_{M^* \leq k \leq \tfrac{1}{4}(l-2N)} C(K+k)^{-\alpha-1} \bigl(1+2^{-3k-N}(K+k)^{\alpha+1}\bigr)^p\\
					&\leq 	\sum_{M^* \leq k \leq \tfrac{1}{4}(l-2N)} C(K+k)^{-\alpha-1}.
			\end{aligned}
			$$
			Moreover by \eqref{BA}, by our choice of $N$ and by \eqref{DistancesEstimate1} this $C$ depends only on $\alpha$, $\beta$ and $p$, which are fixed for our construction. In fact in the following all of the constants $C$ depend only on $\alpha$, $\beta$ and $p$ but not on the parameters $K,N, m$ or $\eta$. Similarly to above we can use \eqref{BB} and \eqref{BBC} and we see that $ |D_1 G_{K,N,m,\eta}\circ \hat{F}_{\beta, m}\circ \tilde{G}_{K,N,m,\eta}| \leq C(\beta)$ and so
			$$
				\begin{aligned}
					\sum_{\tfrac{1}{4}(l-2N)+1 \leq k \leq \tfrac{1}{3}(l-N)+1}& \int_{\ell_k} |D_1 G_{K,N,m,\eta}\circ \hat{F}_{\beta, m}\circ \tilde{G}_{K,N,m,\eta} (t)|^p d \mathcal L^1(t) \\
					&\leq \sum_{\tfrac{1}{4}(l-2N) \leq k \leq \tfrac{1}{3}(l-N)+1}C(K+k)^{-\alpha - 1}.
				\end{aligned}
			$$
			By \eqref{BC} we have that
			$$
				\begin{aligned}
					\sum_{\tfrac{1}{3}(l-N)+2\leq k \leq l-M-M^*-1  }& \int_{\ell_k} |D_1 G_{K,N,m,\eta}\circ \hat{F}_{\beta, m}\circ \tilde{G}_{K,N,m,\eta} (t)|^p d \mathcal L^1(t) \\
					&\leq \sum_{ \tfrac{1}{3}(l-N)+2\leq k \leq l-M-M^*-1   } C 2^{M\beta p}(K+k)^{-\alpha-1}.
				\end{aligned}
			$$
			The combination of the preceding estimates gives
			\begin{equation}\label{PIES2A}
				\begin{aligned}
					&\int_{\bigcup_{k=M^*}^{l-M-M^*-1}\ell_k} |D_1 G_{K,N,m,\eta}\circ \hat{F}_{\beta, m}\circ \tilde{G}_{K,N,m,\eta} (t)|^p d \mathcal L^1(t)\\
					& \qquad \qquad \qquad \qquad \leq \sum_{k=M^*}^{l-M-M^*-1} C(K+k)^{-\alpha-1},
				\end{aligned}
			\end{equation}
			with $C$ independent of $K$. Using \eqref{BD} we get
			\begin{equation}\label{PIES2B}
				\begin{aligned}
					\int_{\bigcup_{k=l-M-M^*}^{3l-4M+N-1}\ell_k} |D_1 G_{K,N,m,\eta}&\circ \hat{F}_{\beta, m}\circ \tilde{G}_{K,N,m,\eta} (t)|^p d \mathcal L^1(t) \\
					&\leq \sum_{l-M-M^* \leq k \leq 3l-4M+N-1} C 2^{\beta p}(K+k)^{(\alpha+1)(p-1)}.
				\end{aligned}
			\end{equation}
			We apply \eqref{BE} to get
			\begin{equation}\label{PIES2C}
				\begin{aligned}
					\int_{\bigcup_{k=3l-4M+N}^{4l+4+2N}\ell_k} &|D_1 G_{K,N,m,\eta}\circ \hat{F}_{\beta, m}\circ \tilde{G}_{K,N,m,\eta} (t)|^p d \mathcal L^1(t) \\
					&\leq \sum_{3l-4M+N \leq k \leq 4l+4+2N} C 2^{(3l+3M+N-k)\beta p} (K+k)^{(\alpha+1)(p-1)} \\
					&\leq  \sum_{3l-4M+N \leq k \leq 4l+4+2N} C (K+k)^{(\alpha+1)(p-1)}.
				\end{aligned}
			\end{equation}
			Finally, we use \eqref{BF} to show
			\begin{equation}\label{PIES2D}
				\begin{aligned}
					 \int_{\ell\cap \{x_1\in\tilde{U}^1_{K,4l+2N+5}\}}& |D_1 G_{K,N,m,\eta}\circ \hat{F}_{\beta, m}\circ \tilde{G}_{K,N,m,\eta} (t)|^p d \mathcal L^1(t) \\
					&\leq C 2^{(3M-l-N)\beta p}.
				\end{aligned}
			\end{equation}
		The summary of the estimates \eqref{PIES2A}, \eqref{PIES2B}, \eqref{PIES2C}, \eqref{PIES2D}, 	\begin{equation}\label{PIES2E}
			\begin{aligned}
				\int_{\ell} |D_1 G_{K,N,m,\eta}&\circ \hat{F}_{\beta, m}\circ \tilde{G}_{K,N,m,\eta} (t)|^p d \mathcal L^1(t) \\
				&\leq \sum_{k=1}^{4l+4+2N} C(K+k)^{(\alpha+1)(p-1)} + C 2^{-\beta p(l+N-3M)}\\
				&\leq C(4l+4+2N+K)^{(\alpha+1)(p-1)+1} + C 2^{(3M-l-N)\beta p}\\
				&\leq C(4l+4+2N+K)^{(\alpha+1)(p-1)+1},
			\end{aligned}
		\end{equation}
		where we have used the fact, that $M$ is an absolute constant introduced in Theorem \ref{BloodSweatAndTears}.
		Multiplying this by the measure of $\tilde{U}_{K,l}^3\setminus \tilde{U}_{K,l+1}^3 \approx (K+l)^{-\alpha -1}$ (see \eqref{DistancesEstimate1}) and summing over $l\geq M^*$ we have
		\begin{equation}\label{PIES2F}
			\begin{aligned}
				& \int_{\tilde{U}_{K,M^*}^1 \times (\tilde{U}_{K, M^*}^3\setminus\C_{A,K}(3))} |D_1 G_{K,N,m,\eta}\circ \hat{F}_{\beta, m}\circ \tilde{G}_{K,N,m,\eta}|^p d \mathcal L^4 \\
				&\qquad\leq \sum_{l=M^*}^{\infty} C(4l+4+2N+K)^{(\alpha+1)(p-1)+1}(K+l)^{-\alpha-1}\\
				&\qquad\leq C\sum_{l=1}^{\infty}(K+l)^{(\alpha+1)(p-1)+1-\alpha-1}\\
				&\qquad\leq CK^{(\alpha+1)(p-2)+2}
			\end{aligned}
		\end{equation}
		since $1\leq p<2$ is fixed, $\alpha \geq \tfrac{4}{2-p}$ and $2K\geq N \geq K+3M \geq N_0$.

		\step{3}{Integral estimates on lines intersecting $\C_{A, K,m}$}{PIES3}
			
			Assuming that we have $(x_2,x_3,x_4) \in \C_{A,K}(3)$, then the segment $\ell=(t,x_2,x_3,x_4)$, $t\in  U^1_{M^*}$ intersects $\C_{A,K}(4)$. We use \eqref{BA} to estimate the derivative on $\{y\in\ell; y_1\in\tilde{U}_{K,k}^1\setminus \tilde{U}_{K,k+1}^1\}$ analogously to \eqref{odkaz} . 
			By \eqref{hhh} we have $D_1G_{K,N,m,\eta}\circ \hat{F}_{\beta, m}\circ \tilde{G}_{K,N,m,\eta} = e_1$ on $\C_{A, K}(4)$, i.e. for $t\in \C_{A,K}(1)$, and thus 
			$$
				\begin{aligned}
					\int_{\ell } |D_1 G_{K,N,m,\eta}\circ \hat{F}_{\beta, m}\circ \tilde{G}_{K,N,m,\eta}|^p d \mathcal L^1 &\leq \sum_{k=M^*}^{\infty} C(K+k)^{-\alpha-1} +\mathcal{L}^1(\C_{A,K}(1))\\
					& \leq CK^{-\alpha} +\mathcal{L}^1(\C_{A,K}(1)).
				\end{aligned}
			$$
			This holds for every $(x_2,x_3,x_4) \in \C_{A,K}(3)$ and $\mathcal{L}^3(\C_{A,K}(3))\leq 2^3$ and hence
			\eqn{PIES3A}
			$$
			\int_{U^1_{M^*}\times \C_{A,K}(3)} |D_1 G_{K,N,m,\eta}\circ \hat{F}_{\beta, m}\circ \tilde{G}_{K,N,m,\eta}|^p d \mathcal L^4\leq C K^{-\alpha} +\mathcal{L}^4(\C_{A,K}).
			$$
			Now this together with \eqref{PIES2F} and Step~\ref{PIES1} prove $i)$. 

			\step{4}{Integral estimates on $\tilde{U}_{K,0}^4 \setminus \tilde{U}_{K,M^*}^4$ on lines close to $\C_{A, K}(3)$}{Annoy1}			
			In steps 4 and 5 we deal with the messy part of $Z_m=[-2m-1,2m+1]^3\times[-1,1]$ close to its boundary. The main aim is to prove the following estimate 
			\begin{equation}\label{Prick}
			\int_{n+\tilde{U}_{K,0}^4 \setminus \tilde{U}_{K,M^*}^4} |D_1G_{K,N,m,\eta}\circ \hat{F}_{\beta, m}\circ \tilde{G}_{K,N,m,\eta}|^p \leq C(\alpha, \beta, p) K^{1-(2-p)(\alpha+1)}
		\end{equation}
		where $n = (2n_1,2n_2,2n_3, 0)$ for $|n_1|\leq m-1$ and $|n_2|, |n_3|\leq m$.	
		
		We want to estimate $|D_1G_{K,N,m,\eta}\circ \hat{F}_{\beta, m}\circ \tilde{G}_{K,N,m,\eta}(x)| $ for $x$ in 
		$$
		(\tilde{U}_{K,0}^1\setminus \tilde{U}_{K,M^*}^1)\times \tilde{U}_{K,M^*}^3=
		(\tilde{U}_{K,0}^1\setminus \tilde{U}_{K,M^*}^1)\times \Big( \bigcup_{l=M^*}^{2N-1} 		
		(\tilde{U}_{K,l}^3\setminus \tilde{U}_{K,l+1}^3)\cup \Bigl[(\tilde{U}_{K,0}^1\setminus \tilde{U}_{K,M^*}^1)\times\tilde{U}_{K,2N}^3\Bigr] \Big)
		$$ 
		and integrate it over lines parallel to $e_1$ in that set. From Lemma~\ref{WrongPlaceWrongTime} and Proposition~\ref{EmotionallyDistant} we get for each $x\in [{U}^1_{k}\setminus {U}^1_{k+1}]\times[{U}^3_{l}\setminus{U}_{l+1}^3]$, $1\leq k\leq M^*-1$, $l\geq k+M^*$ that
		\begin{equation}\label{divideI}
			\hat{F}_{\beta, m}\circ \tilde{G}_{K,N,m,\eta}(x)\in [{U}^1_{k-1}\setminus {U}^1_{k+2}]\times[{U}^3_{l-M}\setminus{U}_{l+M+1}^3]
		\end{equation}
		and for $x\in [{U}^1_{0}\setminus {U}^1_{1}]\times[{U}^3_{l}\setminus{U}_{l+1}^3]$
		\begin{equation}\label{divideII}
		\hat{F}_{\beta, m}\circ \tilde{G}_{K,N,m,\eta}(x)\in [[-2,2]\setminus {U}^1_{2}]\times[{U}^3_{l-M}\setminus{U}_{l+M+1}^3].
		\end{equation}

		Given $k\geq 1$ and $l\geq 4k+2N$, then surely $l\geq k+M^*$ (because $N\geq 3M+3M^*$) and so \eqref{divideI} holds. Then we estimate in the same way as in \eqref{BA}, i.e. we use Proposition~\ref{C1} and Proposition~\ref{ST1} to obtain
		$$
		|D_1G_{K,N,m,\eta}\circ \hat{F}_{\beta, m}\circ \tilde{G}_{K,N,m,\eta}(x+n)| \leq C 2^{\beta +2}+C 2^{\beta}2^{-3k-N}(K+k)^{\alpha+1}.
		$$
		
		Given $k\geq 1$ and $3k+N + M \leq l \leq 4k+2N-1$, then surely $l\geq k+M^*$ (because $N\geq 3M+3M^*$) and so \eqref{divideI} holds. Then we estimate in the same way as in \eqref{BB}, i.e. we use Proposition~\ref{C2} and Proposition~\ref{ST1} to obtain
		$$
		|D_1G_{K,N,m,\eta}\circ \hat{F}_{\beta, m}\circ \tilde{G}_{K,N,m,\eta}(x+n)| \leq C 2^{\beta +2}+C 2^{\beta}2^{-3k-N}(K+k)^{\alpha+1}.
		$$
		 The calculations for the case when $k\geq 1$ and $k+M+M^* \leq l \leq 3k+N+M$ (again $l\geq k+M^*$) are the same as \eqref{BBC} and \eqref{BC} and again
\eqn{tosame}
		$$
		|D_1G_{K,N,m,\eta}\circ \hat{F}_{\beta, m}\circ \tilde{G}_{K,N,m,\eta}(x+n)| \leq C 2^{\beta +2}+C 2^{\beta}2^{-3k-N}(K+k)^{\alpha+1}.
		$$
		
		Now consider the case when $k=0$. In all the above cases (i.e. for $l = M+M^*+1, \dots, 2N$) \eqref{divideII} holds and the estimate of $|D_1\hat{F}_{\beta, m}\circ \tilde{G}_{K,N,m,\eta}(x+n)|$ is the same as before. The main difference is the fact that $\hat{F}_{\beta, m}\circ \tilde{G}_{K,N,m,\eta}(x+n)$ might not lie in $Q(n,1)$ but instead it might lie in the neighboring cube $Q((2n_1\pm 2, 2n_2,2n_3,0),1)$. 
		In the case when $\hat{F}_{\beta, m}\circ \tilde{G}_{K,N,m,\eta}(x+n)\in Q(n,1)$ we obtain the same estimate \eqref{tosame} as before. 
		If this is not the case then 
		$$
		\hat{F}_{\beta, m}\circ \tilde{G}_{K,N,m,\eta}(x+n) \in (U^1_{0}\setminus U^1_{1})\times(U^3_{l-M}\setminus U^3_{l+M+1})+(2n_1\pm 2, 2n_2,2n_3,0)
		$$ 
		(the argument from Proposition~\ref{WrongPlaceWrongTime} that a point cannot skip more than one frame still holds also for points originating outside the cube). 
		For $n_1\in\{-m+1,\hdots, m-1\}$ we are in the neighboring cube where $G_{K,N,m,\eta}$ is given by similar formula and we can estimate its derivatives as before and we have again \eqref{tosame}. 
		The problem is for $n_1=m$ (or similarly for $n_1=-m$) if the first coordinate of $\hat{F}_{\beta, m}\circ \tilde{G}_{K,N,m,\eta}(x+n)$ is bigger than $2m+1$. Then ${G}_{K,N,m,\eta}$ is defined by a convex combination of the frame-to-frame on hyperplane maps $[H^{3,1}_K]_{N}$ and the identity (see \eqref{StretchingBitAroundTop2}). We expound the estimates in detail below.
		
		We use the same estimate of $D_1\tilde{G}_{K,N,m,\eta}(x+n)$ as before, i.e. from \eqref{SQ2A} for any $1 \leq l < 2N$ and (almost) every $x\in (\tilde{U}_{K,0}^1\setminus \tilde{U}_{K,1}^1)\times (\tilde{U}_{K,l}^3\setminus \tilde{U}_{K,l+1}^3)$ we have that
		\begin{equation}\label{ZeeGermanz}
		D_1\tilde{G}_{K,N,m,\eta}(x+n) =  \frac{\frac{1}{2}r_0-r_1}{\frac{1}{2}\tilde{r}_{K,0}-\tilde{r}_{K,1}}e_1
		\end{equation}
		and
		\begin{equation}\label{ZieSweezz}
		\frac{\frac{1}{2}r_0-r_1}{\frac{1}{2}\tilde{r}_{K,0}-\tilde{r}_{K,1}}  \leq C K^{\alpha+1}.
		\end{equation}
		For $x\in (\tilde{U}_{K,0}^1\setminus \tilde{U}_{K,1}^1)\times (\tilde{U}_{K,2N}^3)$ we get from \eqref{SQ1B} that
		\begin{equation}\label{ZeeDutsch}
		|D_1\tilde{G}_{K,N,m,\eta}(x+n) - \frac{\frac{1}{2}r_0-r_1}{\frac{1}{2}\tilde{r}_{K,0}-\tilde{r}_{K,1}}e_1| \leq C2^{-2N\beta}K^{\alpha+1} \leq C2^{-N\beta}
		\end{equation}
		because $N \geq K$.
		
		In the case of $1 \leq l < 2N$ we get by \eqref{ZeeGermanz} that 
		$$
		D_1 \hat{F}_{\beta, m}\circ \tilde{G}_{K,N,m,\eta}(x+n) = D_1\hat{F}_{\beta, m} |D_1\tilde{G}_{K,N,m,\eta}(x+n)|. 
		$$
		Applying \eqref{NotGay1} (using the fact that $l \geq 1$ and $x_1\in [-2m-2,-2m-1]\cup[2m+1,2m+2]$) we get that
		\begin{equation}\label{ZuFwench}
		D_1 \hat{F}_{\beta, m}\circ \tilde{G}_{K,N,m,\eta}(x+n) = \frac{\frac{1}{2}r_0-r_1}{\frac{1}{2}\tilde{r}_{K,0}-\tilde{r}_{K,1}}e_1.
		\end{equation}
		In the case $x\in (\tilde{U}_{K,0}^1\setminus \tilde{U}_{K,1}^1)\times (\tilde{U}_{K,2N}^3)$ we argue by mimicking the proof of \eqref{C1B} that thanks to \eqref{ZeeDutsch} and \eqref{NotGay1} we have that $D_1\hat{F}_{\beta, m}\circ \tilde{G}_{K,N,m,\eta}(x+n)$ has a very small component purpendicular to $e_1$, specifically
		\begin{equation}\label{HuAghhDeeBelgiaanz}
		|D_1 \hat{F}_{\beta, m}\circ \tilde{G}_{K,N,m,\eta}(x+n) - [D_1 \hat{F}_{\beta, m}\circ \tilde{G}_{K,N,m,\eta}(x+n)]^1e_1| \leq C2^{-N\beta},
		\end{equation}
		while
		\begin{equation}\label{Espano}
		[D_1 \hat{F}_{\beta, m}\circ \tilde{G}_{K,N,m,\eta}(x+n)]^1 \leq CK^{\alpha+1}
		\end{equation}
		by \eqref{ZeeDutsch} and \eqref{ZieSweezz}.
		
		 As mentioned earlier we are considering points which are mapped by $\hat{F}_{\beta, m}\circ \tilde{G}_{K,N,m,\eta}$ onto $([-2m-2,-2m-1]\cup[2m+1,2m+2])\times (U_{l-M}^3\setminus U_{l+M+1})$. On this set ${G}_{K,N,m,\eta}$ is a convex combination of $[H^{3}_K]_{N}$ and $\id$ (on hyperplanes purpendicular to $e_1$) and the convex combination goes over a segment of length $1$, i.e. $\pm [2m+1,2m+2]$ (see \eqref{StretchingBitAroundTop2}). Both $[H^{3}_K]_{N}$ and $\id$ send $Q_3(0,1)$ onto itself. Therefore the derivative $|D_1{G}_{K,N,m,\eta}|$ is bounded by 
		 $$
		 C(1+\|[H^{3}_K]_{N}- \id\|_{\infty}) \leq C(1+\diam(Q_3(0,1))) \leq C. 
		 $$
		 On the other hand, for $i=2,3,4$ we have $|D_i{G}_{K,N,m,\eta}| \leq C(|D_i [H^{3,1}_K]_{N}| +1)$ which is calcuated in Proposition~\ref{FrameToFrameDifferential} (see \eqref{Standard5} and \eqref{Standard6}). In summary at the point $\hat{F}_{\beta, m}\circ \tilde{G}_{K,N,m,\eta}(x+n)$ we have
\begin{equation}\label{ZeeVikings}
|D_1 {G}_{K,N,m,\eta}|\leq C
\end{equation}
and
\begin{equation}\label{ZeePicts}
|D {G}_{K,N,m,\eta}|\leq C (\min\{2^{(l+M)\beta}, 2^{N\beta}\} +1)
\end{equation}
because $\hat{F}_{\beta, m}\circ \tilde{G}_{K,N,m,\eta}(x+n)\in ([-2m-2,-2m-1]\cup[2m+1,2m+2])\times (U_{l-M}^3 \setminus U^3_{l+M+1})$.

We calculate $|D_1 {G}_{K,N,m,\eta} \circ \hat{F}_{\beta, m}\circ \tilde{G}_{K,N,m,\eta}(x+n) |$ as follows. In the case that $1\leq l < 2N$ we multiply \eqref{ZuFwench} with \eqref{ZeeVikings} using \eqref{ZieSweezz} to estimate 
\begin{equation}\label{DeeeItaliaaaanoss}
|D_1G_{K,N,m,\eta}\circ \hat{F}_{\beta, m}\circ \tilde{G}_{K,N,m,\eta}(x+n)|\leq C K^{\alpha+1}.
\end{equation}
In the case that $x\in (\tilde{U}_{K,0}^1\setminus \tilde{U}_{K,1}^1)\times (\tilde{U}_{K,2N}^3)$ we use the chain rule with the same estimates as used in the proof of \eqref{BA}. We estimate the first component using \eqref{ZeeVikings} and \eqref{Espano} by
$$
\bigl|D_1G_{K,N,m,\eta}( \hat{F}_{\beta, m}\circ \tilde{G}_{K,N,m,\eta}(x+n))\cdot (D_1 \hat{F}_{\beta, m}\circ \tilde{G}_{K,N,m,\eta}(x+n))^1\bigr| \leq C K^{\alpha+1}
$$
and the other components are estimated by
\begin{equation}\label{TheAngloSaxons}
|DG_{K,N,m,\eta}( \hat{F}_{\beta, m}\circ \tilde{G}_{K,N,m,\eta}(x+n))|\cdot|D_1\hat{F}_{\beta, m}\circ \tilde{G}_{K,N,m,\eta}(x+n) - \frac{\frac{1}{2}r_0-r_1}{\frac{1}{2}\tilde{r}_{K,0}-\tilde{r}_{K,1}}e_1|.
\end{equation}
We estimate \eqref{TheAngloSaxons} using \eqref{HuAghhDeeBelgiaanz}, which we multiply with \eqref{ZeePicts} to estimate \eqref{TheAngloSaxons} by $C2^{N\beta}\cdot2^{-2N\beta}K^{\alpha+1} \leq C$ because $N \geq K$. Altogether we have
$$
|D_1G_{K,N,m,\eta}\circ \hat{F}_{\beta, m}\circ \tilde{G}_{K,N,m,\eta}(x+n)|\leq C K^{\alpha+1}
$$
for $x\in (\tilde{U}_{K,0}^1\setminus \tilde{U}_{K,1}^1)\times (\tilde{U}_{K,2N}^3)+n$ mapped by $\hat{F}_{\beta, m}\circ \tilde{G}_{K,N,m,\eta}$ outside $Z_m$. Now combining the two cases (i.e. the last estimate and \eqref{DeeeItaliaaaanoss}) we obtain
\begin{equation}\label{uuuu2}
|D_1G_{K,N,m,\eta}\circ \hat{F}_{\beta, m}\circ \tilde{G}_{K,N,m,\eta}(x+n)|\leq C K^{\alpha+1}
\end{equation}
for all $x\in (\tilde{U}_{K,0}^1\setminus \tilde{U}_{K,1}^1)\times (\tilde{U}_{K,1}^3)+n$ (in this step we consider only those $x$ with $(x_2,x_3,x_4) \in (\tilde{U}_{K,M^*}^3) + (2n_2,2n_3,0)$).
In summary we obtain 		
	\eqn{uuuu}
		 $$
		 	\begin{aligned}
		 		|D_1G_{K,N,m,\eta}&\circ \hat{F}_{\beta, m}\circ \tilde{G}_{K,N,m,\eta}(x+n)| 
		 		\leq C +C 2^{-N}(K+k)^{\alpha+1} \leq C
		 	\end{aligned}
		 $$
		 for $0 \leq k \leq M^*$, $k \leq l-M^*-M-1$ and $K \leq N \leq 2K$ for $|n_1|\leq m-1$, but only \eqref{uuuu2} holds for $n_1\in\{-m,m\}$. 

The case of $0\leq k\leq M^*$, $l\geq M^*$ and  $ l - M - M^* \leq k \leq M^*$ for $|n_1|\leq m-1$ is much the same as the proof of \eqref{uuuu2}. We have $l\leq 2M^*+M$. The calculation is the same for all values of $n$. By \eqref{SQ2A} we have   
$$
|D_1\tilde{G}_{K,N,m,\eta}(x)| \leq C 2^{-\beta k}(K+k)^{\alpha+1}.
$$ 
Further by Proposition~\ref{EmotionallyDistant} we obtain that $\hat{F}_{\beta, m}\circ \tilde{G}_{K,N,m,\eta}(x+n)$ is far away from $\K_{B,m}$, i.e.
$$
\hat{F}_{\beta, m}\circ \tilde{G}_{K,N,m,\eta}(x+n)\notin U^1_{k+M}\times U_{l+M}^3+n. 
$$
Analogously to reasoning in Proposition \ref{ST0} we obtain that on this set we have 
$$
|DG_{K,N,m,\eta}|\leq C 2^{(2M^*+M)\beta}\leq C
$$
since $k\leq M^*$ and $l\leq 2M^*+M$. It follows that in this case we obtain 
\eqn{dohajzlu}
$$
|D_1G_{K,N,m,\eta}\circ \hat{F}_{\beta, m}\circ \tilde{G}_{K,N,m,\eta}(x)| \leq C 2^{-\beta k}(K+k)^{\alpha+1}\leq 
C K^{\alpha+1}
$$
as $K\geq \tfrac{N}{2}\geq M^*\geq k$. 

		From \eqref{DistancesEstimate1} we obtain for all $K$
		\eqn{ppp}
		$$
		\mathcal{L}^1(\tilde{U}_{K,0}^1\setminus \tilde{U}_{K,M^*}^1) 
		\leq C\sum_{k=1}^{M^*}2^k\frac{1}{2^k}\frac{1}{(k+K)^{\alpha+1}}
		\leq C M^* K^{-\alpha-1}\leq CK^{-\alpha-1}.
		$$		
Analogously we can conclude that
$$
\mathcal{L}^1(\tilde{U}_{K,M^*}^1\setminus \tilde{U}_{K,2M^*+M}^1) \leq CK^{-\alpha-1}.
$$
		Integrating \eqref{uuuu} over $(\tilde{U}_{K,0}^1\setminus \tilde{U}_{K,M^*}^1)\times \tilde{U}_{K,M^*}^3$ and 
		\eqref{dohajzlu} (over the set where $0\leq k\leq M^*$ and $M^*\leq l\leq 2M^*+M$) we obtain that for $|n_1|\leq m-1$ we have 
		\begin{equation}\label{FallingSnow}
			\begin{aligned}
				&\int_{n+(\tilde{U}_{K,0}^1\setminus \tilde{U}_{K,M^*}^1)\times  \tilde{U}_{K,M^*}^3} |D_1G_{K,N,m,\eta}\circ \hat{F}_{\beta, m}\circ \tilde{G}_{K,N,m,\eta}(x)|^p \; dx\\
				&\qquad \qquad \leq CK^{-\alpha-1} C^p+ C K^{-\alpha-1} K^{-\alpha-1} K^{(\alpha+1)p}\\
				& \qquad \qquad \leq CK^{-(\alpha+1)(2-p)}.
			\end{aligned}
		\end{equation}
		Similarly we obtain using \eqref{uuuu2} for $n_1\in\{-m,m\}$
		\begin{equation}\label{FallingSnow2}
			\begin{aligned}
				&\int_{n+(\tilde{U}_{K,0}^1\setminus \tilde{U}_{K,M^*}^1)\times  \tilde{U}_{K,M^*}^3} |D_1G_{K,N,m,\eta}\circ \hat{F}_{\beta, m}\circ \tilde{G}_{K,N,m,\eta}(x)|^p \; dx\\
				&\qquad \qquad \leq CK^{-\alpha-1}K^{(\alpha+1)p}+ C K^{-\alpha-1} K^{-\alpha-1} K^{(\alpha+1)p}\\
				& \qquad \qquad \leq CK^{(\alpha+1)(p-1)}.
			\end{aligned}
		\end{equation}

		\step{5}{Integral estimates on $\tilde{U}_{K,0}^4 \setminus \tilde{U}_{K,M^*}^4$ on lines far away from $\C_{A, K}(3)$}{Annoy2}		
		The remaining lines are those lines parallel to $e_1$ that lie in the set $ \tilde{U}_{K,0}^4 \setminus ([-1,1]\times \tilde{U}_{K,M^*}^3$. When $0\leq l\leq M^*$ and $0\leq k\leq M^*$ we obtain \eqref{dohajzlu} with the same reasoning as before and thus 
		\eqn{nezapomen}	
		$$
		\int_{(\tilde{U}_{K,0}^1\setminus \tilde{U}_{K,M^*}^1)^2\times  \tilde{U}_{K,M^*}^2} |D_1G_{K,N,m,\eta}\circ \hat{F}_{\beta, m}\circ \tilde{G}_{K,N,m,\eta}(x)|^p \leq CK^{-(\alpha+1)(2-p)}	
				$$
				and the same holds for other permutations of $(x_2,x_3,x_4)$. It remains to consider $0\leq l\leq M^*$ and 
				$M^* \leq k$. 
				
				Let us first consider $0\leq l\leq M^*$ and $M^* \leq k \leq 4l+4+2N$, $x_1\in \tilde{U}^1_{K,k}\setminus \tilde{U}^1_{K,k+1}$ and $(x_2, x_3,x_4)\in \tilde{U}^3_{K,l}\setminus \tilde{U}_{K,l+1}^3$. 
				From Proposition \ref{SQ3} we obtain 
				$$
				|D_1\tilde{G}_{K,N,m,\eta}(x)| \leq C2^{-k\beta}(K+k)^{\alpha+1}. 
				$$
				Our $\hat{F}_{\beta, m}$ is bilipschitz and we claim that we can estimate 
				\eqn{snad}
				$$
				\bigl|D{G}_{K,N,m,\eta}\bigl(\hat{F}_{\beta, m}\circ \tilde{G}_{K,N,m,\eta}(x+n)\bigr)\bigr| \leq C2^{(k+M) \beta}. 
				$$
				Indeed, 
				our $\tilde{G}_{K,N,m,\eta}$ maps 
				$$
				(\tilde{U}^1_{K,k}\setminus \tilde{U}^1_{K,k+1})\times( \tilde{U}^3_{K,l}\setminus \tilde{U}_{K,l+1}^3)
				\text{ into }
				(U^1_{k}\setminus U^1_{k+1})\times (U^3_{l}\setminus U_{l+1}^3). 
								$$
By Proposition~\ref{EmotionallyDistant} the index $k$ is shifted to index at most $k+M$ by the mapping $ \hat{F}_{\beta, m}$. If $ \hat{F}_{\beta, m}\circ \tilde{G}_{K,N,m,\eta}(x+n) \in Q(n,1)$ then we can estimate $|D{G}_{K,N,m,\eta}|$ by the minimum of \eqref{ST3A} and \eqref{ST3B}. If the point is outside $Q(n,1)$ then ${G}_{K,N,m,\eta}$ is defined by a convex combination of $[H^{3,j}_K]_{N}$ and identity and so it has derivative majorised by the estimate the minimum of \eqref{ST3A} and \eqref{ST3B}, and therefore in both caes we have \eqref{snad}. 
It follows that in this case we have 	
\begin{equation}\label{PIES3D}
		|D_1G_{K,N,m,\eta}\circ \hat{F}_{\beta, m}\circ \tilde{G}_{K,N,m,\eta}(x+n)| \leq  C(K+k)^{\alpha+1}.
		\end{equation}			
		
		It remains to consider $0\leq l\leq M^*$ and $ x\in \tilde{U}^1_{K,4l+4+2N}\times( \tilde{U}^3_{K,l}\setminus \tilde{U}_{K,l+1}^3) $. By \eqref{SQ3D1} we estimate
		$$
		|D_1\tilde{G}_{K,N,m,\eta}(x+n)| \leq 2^{-2N\beta - 4l\beta}.
		$$
		By Proposition~\ref{EmotionallyDistant} the index $l$ shifts to at most $l+M\leq M^*+M$ in the mapping $ \hat{F}_{\beta, m}$. If $ \hat{F}_{\beta, m}\circ \tilde{G}_{K,N,m,\eta}(x+n) \in Q(n,1)$ then we can estimate $|D{G}_{K,N,m,\eta}|$ by \eqref{ST3B}. If the point is outside $Q(n,1)$ then ${G}_{K,N,m,\eta}$ is defined by a convex combination of $[H^{3,j}_K]_{N}$ and identity and so it has derivative majorised by the estimate from \eqref{ST3B}, i.e. we have
		$$
		|D{G}_{K,N,m,\eta}| \leq C2^{(3M^*+3M+N) \beta}.
		$$ 
		Therefore we obtain 
		\begin{equation}\label{PIES3F}
		\begin{aligned}
		|D_1G_{K,N,m,\eta}\circ \hat{F}_{\beta, m}\circ \tilde{G}_{K,N,m,\eta}(x+n)| &\leq 
		C 2^{-2N\beta} 2^{(3M+3M^*+N)\beta}\\
		&\leq 
		C2^{-N\beta }\leq C.
		\end{aligned}
		\end{equation}

		Using the estimates of the derivatives are in  \eqref{PIES3D} and \eqref{PIES3F} we estimate 
		on each such line $\ell$ 
		\begin{equation}\label{AnnoyingLines}
			\begin{aligned}
				\int_{\ell}|D_1&G_{K,N,m,\eta}\circ \hat{F}_{\beta, m}\circ \tilde{G}_{K,N,m,\eta}(x)| \\
				& \leq  \sum_{k=0}^{4l+2N} C(K+k)^{(\alpha+1)p}(K+k)^{-(\alpha+1)} +C\\
				&\leq C\sum_{k=0}^{4M^* + 2N}(K+k)^{(\alpha+1)(p-1)}+C\\
				&\leq C(K+N)^{(\alpha+1)(p-1)+1}\\
				&\leq CK^{(\alpha+1)(p-1)+1}
			\end{aligned}
		\end{equation}
		because $N_0 \leq K\leq N\leq 2 K$.
		
		

		By \eqref{DistancesEstimate1} we have analogously to \eqref{ppp} that
		$$
		\mathcal{L}^3\Bigl( \tilde{U}_{K,0}^3 \setminus \tilde{U}_{K,M^*}^3 \Bigr)
		\leq C K^{-\alpha -1}.
		$$
		Multiplying \eqref{AnnoyingLines} by this measure estimate and adding to \eqref{nezapomen} we get
		$$
		\int_{\tilde{U}_{K,0}^4 \setminus ([-1,1]\times \tilde{U}_{K,M^*}^3)}|D_1G_{K,N,m,\eta}\circ \hat{F}_{\beta, m}\circ \tilde{G}_{K,N,m,\eta}(x)| \leq C K^{-(\alpha +1)(p-2)+1}
		$$
		
		Combining the previous estimate with Step~\ref{Annoy1} we have
		$$
			\begin{aligned}
				&\int_{n+\tilde{U}_{K,0}^4 \setminus \tilde{U}_{K,M^*}^4} |D_1G_{K,N,m,\eta}\circ \hat{F}_{\beta, m}\circ \tilde{G}_{K,N,m,\eta}|^p \\
				&\qquad \leq CK^{-(\alpha +1)(p-2)}  + C K^{-(\alpha +1)(p-2)+1}\\
				 &\qquad \leq CK^{1-(\alpha+1) (2-p)}
			\end{aligned}
		$$		
		for $|n_1|\leq m-1$ which gives us \eqref{Prick}. In case $n_1\in\{-m,m\}$ we obtain $CK^{(\alpha+1)(p-1)}$ on the righthand side by \eqref{FallingSnow2}. 
		
		\step{6}{Proving $ii)$ and $iii)$ for $j=1,2,3$}{PIES4}
			
Finally adding \eqref{Prick}, \eqref{PIES3A} and    \eqref{PIES2F} (and excluding the $\mathcal{L}^4(\C_{A,K})$ term in \eqref{PIES3A}) we get for $|n_1|\leq m-1$ 
			\begin{equation}\label{PIES3B}
				\begin{aligned}
					&\int_{n+Q(0,1)\setminus\C_{A,K}} |D_1 G_{K,N,m,\eta}\circ \hat{F}_{\beta, m}\circ \tilde{G}_{K,N,m,\eta}|^p d \mathcal L^4 \\
					&\qquad\qquad\qquad \leq C K^{1-(2-p)(\alpha+1)} + CK^{-\alpha}+C K^{(\alpha+1)(p-2)+2} \\
					&\qquad\qquad\qquad \leq C K^{-(\alpha +1)(2-p) + 2} ,
				\end{aligned}
			\end{equation}
			since $1\leq p<2$ is fixed and $\alpha \geq \tfrac{4}{2-p}$. For $n_1\in\{-m,m\}$ we obtain $CK^{(\alpha+1)(p-1)}$ on the righthand side by \eqref{FallingSnow2}. 
			We sum \eqref{PIES3B} over the $(2m-1)(2m+1)^2$ cubes in the Cantor plate construction 
			together with the corresponding case for $n_1\in\{-m,m\}$ and get,
			$$
				\int_{Z_m\setminus \C_{A,K,m}} |D_1 G_{K,N,m,\eta}\circ \hat{F}_{\beta, m}\circ \tilde{G}_{K,N,m,\eta}|^p \leq CK^{-(\alpha+1)(2-p)+2} m^3+C K^{(\alpha+1)(p-1)} m^2. 
			$$
		The inequality in $iii)$ for $j=2,3$ is proved by rotational symmetry and exchanging the role of $|n_1|\leq m-1$ for $n_j$. From here $ii)$ is proved by adding the $m^3\mathcal{L}^4(\C_{A,K}) \leq 16m^3$ term in \eqref{PIES3A} and noting that $-(\alpha+1)(2-p)+2<0$. In fact, in the cases $j=1,2,3$ we don't need the $m^3 (1+ \eta^p K^{p(\alpha+1)})K^{-\alpha-1}$ term from $ii)$ and $iii)$ as the integral is already estimated by the first 2 terms.
		
		\step{7}{Proving $ii)$ and $iii)$ for $j=4$}{PIES7}
		The majority of the proof for $j=4$ is the same as in the previous cases. The difference is as follows, in direction $e_4$ the cantor plate construction $Z_m$ is only $1$ cube thick, i.e. all cubes fall into the equivalent category of the previous step for $|n_1| = m$. We continue to estimate $|D_4G_{K,N,m,\eta}\circ \hat{F}_{\beta, m}\circ \tilde{G}_{K,N,m,\eta}|$ in the case $x\in (\tilde{U}_{K,2N}^3)\times (\tilde{U}_{K,0}^1\setminus \tilde{U}_{K,1}^1)$. The case we need to deal with is when $\hat{F}_{\beta, m}\circ \tilde{G}_{K,N,m,\eta}(x+n) \notin Z_m$.
		
		The same argument which we used to get \eqref{HuAghhDeeBelgiaanz} gives
		\begin{equation}\label{WarriorNation}
		|D_4 \hat{F}_{\beta, m}\circ \tilde{G}_{K,N,m,\eta}(x+n) - [D_4 \hat{F}_{\beta, m}\circ \tilde{G}_{K,N,m,\eta}(x+n)]^4e_4| \leq C2^{-N\beta}.
		\end{equation}
		We have  
		$$
		\tilde{G}_{K,N,m,\eta}(\tilde{U}_{K,2N}^3\times (\tilde{U}_{K,0}^1\setminus \tilde{U}_{K,1}^1)) = {U}_{2N}^3\times ({U}_{0}^1\setminus {U}_{1}^1).
		$$ 
		Each cube in ${U}_{2N}^3$ has diameter $C2^{-2N - 2N{\beta}}$. The map $\hat{F}_{\beta, m}$ is Lipschitz and \eqref{NotGay4} holds and therefore 
		$$
		|(\hat{F}_{\beta, m}(y))^4 - (-y_4)| < C2^{-2N - 2N{\beta}}. 
		$$
		Thus, as soon as $N$ is large enough (i.e. we assume that $N\geq K \geq N_0$) we have $|(\hat{F}_{\beta, m}(y))^4 - (-y_4)|<\tfrac{1}{2}$. That is to say any point in $\tilde{U}_{K,2N}^3\times (\tilde{U}_{K,0}^1\setminus \tilde{U}_{K,1}^1)$ mapped outside $Z_m$ by $ \hat{F}_{\beta, m}\circ \tilde{G}_{K,N,m,\eta}$ lands in the set ${U}_{2N-M}^3\times ([-\tfrac{3}{2},-1]\cup[1,\tfrac{3}{2}])$. By \eqref{StretchingBitAroundTop} we have that at these points
		$$
		D_4 {G}_{K,N,m,\eta}(\hat{F}_{\beta, m}\circ \tilde{G}_{K,N,m,\eta}(x+n)) = \tfrac{\eta}{13}e_4.
		$$
		The application of the chain rule (analogous to the proof of \eqref{BA}) leads us to sum of
		$$
		\bigl|D_4 {G}_{K,N,m,\eta}(\hat{F}_{\beta, m}\circ \tilde{G}_{K,N,m,\eta}(x+n))[D_4 \hat{F}_{\beta, m}\circ \tilde{G}_{K,N,m,\eta}(x+n)]^4\bigr| \leq C\eta K^{\alpha+1}
		$$
		with
		$$
		\bigl|D_j {G}_{K,N,m,\eta}(\hat{F}_{\beta, m}\circ \tilde{G}_{K,N,m,\eta}(x+n))[D_4 \hat{F}_{\beta, m}\circ \tilde{G}_{K,N,m,\eta}(x+n)]^j \bigr|< C2^{-N\beta} 2^{N\beta} = C
		$$
		(recall by \eqref{StretchingBitAroundTop} and \eqref{Standard6} that $|D_j G_{K,N,m,\eta}| \leq 2^{N\beta}$ and the second factor is bounded by \eqref{WarriorNation}) giving 
		\begin{equation}\label{NeverDieYoung}
		|D_4G_{K,N,m,\eta}\circ \hat{F}_{\beta, m}\circ \tilde{G}_{K,N,m,\eta}(x+n)| < C(\eta K^{\alpha+1}+1)
		\end{equation}
		for (almost) all $x\in (\tilde{U}_{K,2N}^3)\times (\tilde{U}_{K,0}^1\setminus \tilde{U}_{K,1}^1)$ with $x+n$ mapped outside $Z_m$ by $ \hat{F}_{\beta, m}\circ \tilde{G}_{K,N,m,\eta}$. On the other hand, if $x$ is mapped inside $Q(n,1)$ we have
		\begin{equation}\label{NeverDieAnAdoloscent}
		|D_4G_{K,N,m,\eta}\circ \hat{F}_{\beta, m}\circ \tilde{G}_{K,N,m,\eta}(x+n)| < C
		\end{equation}  
		for (almost) all $x\in (\tilde{U}_{K,2N}^3)\times (\tilde{U}_{K,0}^1\setminus \tilde{U}_{K,1}^1)$ (calculations the same as in Step~\ref{Annoy1}).
		
		If $x\in (\tilde{U}_{K,1}^3 \setminus \tilde{U}_{K,2N}^3)\times (\tilde{U}_{K,0}^1\setminus \tilde{U}_{K,1}^1)$ we use a simple estimate, i.e. for (almost) all $x$ we have $D_4 \tilde{G}_{K,N,m,\eta}(x+n)$ is parallel with $e_4$ and has size $CK^{\alpha+1}$. By \eqref{NotGay4} we are interested only in $D_4 {G}_{K,N,m,\eta}(\hat{F}_{\beta, m}\circ \tilde{G}_{K,N,m,\eta}(x+n))$ which is either bounded by $C$ in the case that $\hat{F}_{\beta, m}\circ \tilde{G}_{K,N,m,\eta}(x+n)$ is outside $Z_m$ (we cannot exclude the case that $|[\hat{F}_{\beta, m}\circ \tilde{G}_{K,N,m,\eta}(x+n)]^{4}|>\tfrac{3}{2}$ and there all we have is ${G}_{K,N,m,\eta}$ Lipschitz) or $CK^{-\alpha-1}$ in the case the point stays inside the cube. The larger of these is obviously $C$ and (similar to \eqref{uuuu2}) we have
		$$
		|D_4G_{K,N,m,\eta}\circ \hat{F}_{\beta, m}\circ \tilde{G}_{K,N,m,\eta}(x+n)| < CK^{\alpha+1}
		$$
		for (almost) all $x\in (\tilde{U}_{K,1}^3 \setminus \tilde{U}_{K,2N}^3)\times (\tilde{U}_{K,0}^1\setminus \tilde{U}_{K,1}^1)$. For $x\in (\tilde{U}_{K,0}^3 \setminus \tilde{U}_{K,1}^3)\times (\tilde{U}_{K,0}^1\setminus \tilde{U}_{K,1}^1)$ we simply apply \eqref{SQ3A1}, Lemma~\ref{ThatTimeOfTheMonth} and Proposition~\ref{ST0} to get the same estimate of $CK^{\alpha+1}$. Therefore
		\begin{equation}\label{NeverDieAVirgin}
		|D_4G_{K,N,m,\eta}\circ \hat{F}_{\beta, m}\circ \tilde{G}_{K,N,m,\eta}(x+n)| < CK^{\alpha+1}
		\end{equation}  
		for (almost) all $x\in (\tilde{U}_{K,0}^3 \setminus \tilde{U}_{K,2N}^3)\times (\tilde{U}_{K,0}^1\setminus \tilde{U}_{K,1}^1)$.
		
		Clearly we have (independantly of choice of $N$) 
		$$
		\mathcal{L}^3(\tilde{U}_{K,0}^3 \setminus \tilde{U}_{K,2N}^3) \leq\mathcal{L}^3(\tilde{U}_{K,0}^3 \setminus \C_{A,K}(3))\leq CK^{-\alpha}\text{ and }\mathcal{L}^3(\tilde{U}_{K,2N}^3)<8. 
		$$  
		Again we have the estimate $\mathcal{L}^3( \tilde{U}_{K,0}^3 \setminus\tilde{U}_{K,1}^3)<CK^{-\alpha-1}$. Now from \eqref{NeverDieYoung} and \eqref{NeverDieAnAdoloscent} we get 
		\begin{equation}\label{Leprosy}
		\int_{\tilde{U}_{K,2N}^3 \times (\tilde{U}_{K,0}^3 \setminus\tilde{U}_{K,1}^3)}|D_4G_{K,N,m,\eta}\circ \hat{F}_{\beta, m}\circ \tilde{G}_{K,N,m,\eta}(x+n)|^p < C(1+\eta^p K^{p(\alpha+1)})K^{-\alpha-1}.
		\end{equation}
		and from \eqref{NeverDieAVirgin} we obtain 
		\begin{equation}\label{MakePeace}
		\int_{(\tilde{U}_{K,0}^3 \setminus\tilde{U}_{K,1}^3)\times (\tilde{U}_{K,0}^3 \setminus\tilde{U}_{K,1}^3)}|D_4G_{K,N,m,\eta}\circ \hat{F}_{\beta, m}\circ \tilde{G}_{K,N,m,\eta}(x+n)|^p < CK^{-(2-p)(\alpha+1)+1}.
		\end{equation}
		The sum of \eqref{Leprosy} over the $(2m+1)^3$ cubes in $Z_m$ is estimated by the $Cm^3 (1+\eta^pK^{p\alpha+p})K^{-\alpha-1}$--term of $ii)$ and $iii)$. The sum of \eqref{MakePeace} over the $(2m+1)^3$ cubes in $Z_m$ is estimated by the $Cm^3 K^{2 - (2-p)(\alpha+1)}$--term of $ii)$ and $iii)$. The other cases, i.e. when $k\geq 1$, are identical by rotational symmetry to those dealt with in Step~\ref{Annoy1}, Step~\ref{Annoy2} and Step~\ref{PIES4} and are estimated already by the $Cm^3 K^{2-(2-p)(\alpha+1)} + Cm^2 K^{(p-1)(\alpha+1)}$--terms of $ii)$ and $iii)$.
		
	\end{proof}

\section{Proof of Theorem~\ref{TheBigLebowski}.}\label{TheBigDeal}

Recall that $R_{m,\eta}=[-2m-5,2m+5]^{3}\times[-1-\eta,1+\eta]$.

	\begin{prop}\label{TheLittleLebowski}Let $1\leq p<2$ and $\alpha\geq\frac{4}{2-p}$.
		For each $K$ put $\eta = \eta(K) = K^{-\alpha-1}$. Call $f_{K,N,m,\eta} = G_{K,N,m,\eta}\circ \hat{F}_{\beta, m}\circ \tilde{G}_{K,N,m,\eta}: R_{m,\eta} \to \er^4$ with $K\geq K_0$ and $2K\geq N\geq 3M+K$ and $N\geq 3M^*+3M+6$. Then,
		\begin{enumerate}
			\item[$i)$] $f_{K,N,m,\eta} \in W^{1,p}(R_{m,\eta} , \er^4)$,
			\item[$ii)$] $f_{K,N,m,\eta}(x) = x$  for $x\in \partial R_{m,\eta}$, 
			\item[$iii)$] the map $f_{K,N,m,\eta}$ is locally bi-Lipschitz on $R_{m,\eta} \setminus \C_{A, K,m}$,
			\item[$iv)$] $J_{f_{K,N,m,\eta}}<0$ on $\C_{A,K,m}$,
			\item[$v)$] $\displaystyle\int_{R_{m,\eta}\setminus \C_{A,K,m}} |D_j f_{K,N,m,\eta}|^p \leq C(p,\alpha, \beta)K^{(\alpha + 1)(p-1)+1}((m+1)^2 +(m+1)^3\eta) +$\\
			\phantom{a}\qquad\qquad $\displaystyle + C(p,\alpha, \beta)K^{-(\alpha+1)(2-p)+2}(m+1)^3$
					 for every $j\in\{1,2,3\}$,
			\item[$vi)$] $\displaystyle \int_{R_{m,\eta} \setminus \C_{A,K,m}} |D_4 f_{K,N,m,\eta}|^p \leq C(p,\alpha, \beta)K^{(\alpha + 1)(p-1)+1}(m+1)^2+$\\
			\phantom{a}\qquad\qquad$\displaystyle +C(p,\alpha, \beta)(m+1)^3\eta^{1-p} + C(p,\alpha, \beta)K^{-(\alpha+1)(2-p)+2}(m+1)^3$.
		\end{enumerate}	
	\end{prop}
	\begin{proof}
		\step{1}{Proof of $iii)$}{PTLLS1}
		
		In Step~\ref{PIES1} of the proof Proposition~\ref{IntegralEstimate} we showed that $f_{K,N,m,\eta}$ is locally bi-Lipschitz on $Z_m\setminus \C_{A, K,m}$. The fact that $f_{K,N,m,\eta}$ is bi-Lipschitz on $R_{m,\eta}\setminus Z_m$ is obvious from Proposition~\ref{SQ5}, Proposition~\ref{ST0} and Theorem~\ref{reflect} (the fact that $\hat{F}_{\beta, m}$ is the identity on the boundary and that $\hat{F}_{\beta ,m}(\C_{B,m}) = \C_{B,m}$).
		
		\step{2}{Proof of $i)$}{PTLLS2}
		
		The fact that $f_{K,N,m,\eta} \in W^{1,p}((-2m-5,2m+5)^3\times(-1-\eta, 1+\eta) , \er^4)$ comes from Step~\ref{PTLLS1} (on the part outside $Z_m$) and Proposition~\ref{IntegralEstimate} point $i)$ on the union of the cubes.
		
		\step{3}{Proof of $ii)$}{PTLLS3}
		
		This holds from the fact that $\hat{F}_{\beta, m}$ is identity on $\partial R_{m,\eta}$ and that $G_{K,N,m,\eta}$ is the inverse to $\tilde{G}_{K,N,m,\eta}$ on the boundary as can easily be observed from \eqref{DeathToGo1}, \eqref{DeathToGo2}, \eqref{OhDear1B}, \eqref{OhDear2}, \eqref{StretchingBitAroundTop} and \eqref{StretchingBitAroundSides}.
		
		\step{4}{Proof of $iv)$}{PTLLS4}
		
		The fact that $J_{f_{K,N,m,\eta}}<0$ on $C_{A,K,m}$ follows from two facts. The first one is that that $\tilde{G}_{K,N,m,\eta} (x)= (q_K(x_1),q_K(x_2),q_K(x_3),q_K(x_4))$ on $C_{A,K}$ and
		$G_{K,N,m,\eta}(x) = (t_K(x_1), t_K(x_2),t_K(x_3),t_K(x_4))$ on $\C_{B} = \tilde{G}_{K,N,m,\eta}(\C_{A,K})$ and these two maps are mutually inverse on these sets. The combination of the above fact and \eqref{NewNew} precisely prove that the approximative derivative satisfies
		$$
		D_j f_{K,N,m,\eta}(x) = e_j\text{ for }j=1,2,3\text{ and }D_4 f_{K,N,m,\eta}(x)= -e_4\text{ for a.e. }x\in \C_{A,K}.
		$$
		It is well known that the weak (Sobolev) derivative equals to approximative derivative a.e. and so $J_{f_{K,N,m,\eta}}(x) = -1$ for a.e. $x\in\C_{A,K}$.

		\step{5}{Proof of $v)$}{PTLLS5}
		By symmetry it is enough to show this for $j=1$. 
		We calculate using $iii)$ of Proposition~\ref{IntegralEstimate}. Notice that since we have already chosen $\eta = K^{-\alpha-1}$ the right hand side of $iii)$ of Proposition~\ref{IntegralEstimate} simplifies because the term 
		$$
		Cm^3 (1+\eta^pK^{p\alpha+p})K^{-\alpha-1}=Cm^3K^{-\alpha-1}
		$$
		but because $p\geq1$
		$$
		Cm^3K^{-\alpha-1} \leq Cm^3 K^{2-(2-p)(\alpha+1)}
		$$
		which is the first term of $iii)$ of Proposition~\ref{IntegralEstimate}, therefore we use only the first two terms. Then we continue to apply the simplified estimate from $iii)$ of Proposition~\ref{IntegralEstimate} and we get
		\begin{equation}\label{DiCaprio}
			\begin{aligned}
				&\int_{R_{m,\eta}\setminus \C_{A,K,m}} |D_1 G_{K,N,m,\eta}\circ \hat{F}_{\beta, m}\circ \tilde{G}_{K,N,m,\eta}|^p \leq C(\alpha, \beta, p) K^{-(\alpha+1)(2-p)+2} m^3+\\
				& \qquad \qquad + C(\alpha, \beta, p) K^{(\alpha+1)(p-1)} m^2 + \int_{R_{m,\eta}\setminus Z_m} |D_1 G_{K,N,m,\eta}\circ \hat{F}_{\beta, m}\circ \tilde{G}_{K,N,m,\eta}|^p.
			\end{aligned}
		\end{equation}
		By the definition of $\tilde{G}_{K,N,m,\eta}$ especially \eqref{OhDear1A} we can easily calculate that
		\eqn{starrr}
		$$
			\tilde{G}_{K,N,m,\eta}^{-1}(R_{m,2}) = [-2m-5,2m+5]^3 \times [-1-\tfrac{2\eta}{26 - \eta},1+ \tfrac{2\eta}{26 - \eta}]
		$$
		and
		\eqn{starrrr}
		$$
			\tilde{G}_{K,N,m,\eta}^{-1}(R_{m,13} \setminus R_{m,2}) =[-2m-5,2m+5]^3 \times \big([-1-\eta, 1+\eta]\setminus [-1-\tfrac{2\eta}{26 - \eta}, 1+\tfrac{2\eta}{26 - \eta}]\big).
		$$
		With respect to Lemma~\ref{ThatTimeOfTheMonth} and the definition of ${G}_{K,N,m,\eta}$ it is necessary to calculate separately on the sets in \eqref{starrr} and \eqref{starrrr}. 
		We decompose the integral over $R_{m,\eta}\setminus Z_m$ into several sets. The sets around the sides of $Z_m$ are
		$$
			\begin{aligned}
				E_{1} =&( [-2m-5, 2m+5]\setminus [-2m-1, 2m+1])\times[-2m-1,2m+1]^2\times [-1,1],\\
				E_{2} =&[-2m-5, 2m+5]\setminus [-2m-1,2m+1]\\
				&\times([-2m-5,2m+5]^2\setminus[-2m-1,2m+1]^2)\times [-1,1]\\
				E_{3} =&[-2m-1, 2m+1]\times([-2m-5,2m+5]^2\setminus[-2m-3,2m+3]^2)\times [-1,1]\\
							& \cap \Big((\C_{A,K,m} + \er e_2) \cup (\C_{A,K,m} + \er e_3)\Big)\\
				E_{4} =&[-2m-1, 2m+1]\times([-2m-5,2m+5]^2\setminus[-2m-3,2m+3]^2)\times [-1,1]\\
							& \setminus \Big((\C_{A,K,m} + \er e_2) \cup (\C_{A,K,m} + \er e_3)\Big)\\
				E_{5} =&[-2m-1, 2m+1]\times([-2m-3,2m+3]^2\setminus[-2m-1,2m+1]^2)\times [-1,1]\\
							& \cap \Big((\C_{A,K,m} + \er e_2) \cup (\C_{A,K,m} + \er e_3)\Big)\\
				E_{6} =&[-2m-1, 2m+1]\times([-2m-3,2m+3]^2\setminus[-2m-1,2m+1]^2)\times [-1,1]\\
							& \setminus \Big((\C_{A,K,m} + \er e_2) \cup (\C_{A,K,m} + \er e_3)\Big).
				\end{aligned}
			$$
			The sets above and below $Z_m$ are
			$$
					\begin{aligned}
				E_{7} =& [-2m-5, 2m+5]^3 \times ([-1-\tfrac{2\eta}{26 - \eta}, -1]\cup[1,1+\tfrac{2\eta}{26 - \eta}]) \setminus \Big(\C_{A,K,m} + \er e_4\Big)\\
				E_{8} =& [-2m-5, 2m+5]^3 \times ([-1-\tfrac{2\eta}{26 - \eta}, -1]\cup[1,1+\tfrac{2\eta}{26 - \eta}]) \cap \Big(\C_{A,K,m} + \er e_4\Big)\\
				E_{9} =& [-2m-5, 2m+5]^3 \times ([-1-\eta,-1-\tfrac{2\eta}{26 - \eta}]\cup[1+\tfrac{2\eta}{26 - \eta}, 1+\eta])\setminus \Big(\C_{A,K,m} + \er e_4\Big)\\
				E_{10} =& [-2m-5, 2m+5]^3 \times ([-1-\eta,-1-\tfrac{2\eta}{26 - \eta}]\cup[1+\tfrac{2\eta}{26 - \eta}, 1+\eta]).\\
			\end{aligned}
		$$
		It holds that $R_{m,\eta}\setminus Z_m = \bigcup_{i=1}^{10}E_i$ and (see \eqref{rozdil})
		\begin{equation}\label{StupidMeasures}
		\begin{aligned}
			\mathcal{L}^4(E_{1}),\mathcal{L}^4(E_{3}),\mathcal{L}^4(E_{5}) & \leq C(m+1)^2, \\
			\mathcal{L}^4(E_{2})  &\leq C(m+1),\\
			\mathcal{L}^4(E_{4}),\mathcal{L}^4(E_{6})  & \leq CK^{-\alpha}(m+1)^2, \\
			\mathcal{L}^4(E_{7}),\mathcal{L}^4(E_{9})  &\leq CK^{-\alpha} \eta (m+1)^3, \\
			\mathcal{L}^4(E_{8}),\mathcal{L}^4(E_{10}) &\leq C\eta (m+1)^3.
		\end{aligned}
		\end{equation}

		
		We estimate $D_1G_{K,N,m,\eta}\circ \hat{F}_{\beta, m}\circ \tilde{G}_{K,N,m,\eta}$ on $E_1$ as follows. On 
		$$
		E_{1,a}:=( [-2m-5, 2m+5]\setminus [-2m-4, 2m+4])\times[-2m-1,2m+1]^2\times [-1,1]
		$$ 
		we have $|D_1\tilde{G}_{K,N,m,\eta}|\leq C$ by \eqref{OhDear2}. By the definition of $\tilde{G}_{K,N,m,\eta}$ we have
		$$
		\tilde{G}_{K,N,m,\eta}(E_{1,a})=E_{1,a}
		$$
		Further 
		$$
			P_v\bigl( E_{1,a} \bigr) \subset \er^3\setminus [-2m-4+\tfrac{1}{14}, 2m+4-\tfrac{1}{14}]^3
		$$
		(because $\max\{|v_1|,|v_2|,|v_3|\}\leq \tfrac{1}{14}$) and on this set we have $g=0$ by \eqref{gjenula}. Therefore (see \eqref{FDeath} and \eqref{spaghetti}) $\hat{F}_{\beta,m} = \id$ on 
$E_{1,a}$.	Further, on $E_{1,a}$ we have ${G}_{K,N,m,\eta} = \id$ by \eqref{StretchingBitAroundTop2}. Therefore on $E_{1,a}$ we have 
		$$
			|D_1G_{K,N,m,\eta}\circ \hat{F}_{\beta, m}\circ \tilde{G}_{K,N,m,\eta}(x)| = |D_1\id\circ\id\circ \tilde{G}_{K,N,m,\eta}(x)| =|D_1\tilde{G}_{K,N,m,\eta}(x)| \leq C. 
		$$
		
For (almost) all $x$ in 
		$$
			E_{1,b}:=( [-2m-4, 2m+4]\setminus [-2m-3, 2m+3])\times[-2m-1,2m+1]^2\times [-1,1]
		$$
		we have $D_1\tilde{G}_{K,N,m,\eta}(x) = e_1$ by \eqref{OhDear2} and
		$$
		\tilde{G}_{K,N,m,\eta}\big(E_{1,b}\big)= E_{1,b}.
		$$
		On this set we use the fact that $\hat{F}_{\beta,m}$ is Lipschitz to estimate $|D_1\hat{F}_{\beta,m}| < C$. 
		From Lemma~\ref{ThatTimeOfTheMonth} we obtain $\hat{F}_{\beta,m}(x)\notin\T$ since $E_{1,b}\subset \G$.  
		Applying Proposition~\ref{ST0} for these $\hat{F}_{\beta,m}(x)$ we have $|D{G}_{K,N,m,\eta}|\leq C(M,\beta)$. Therefore 
		$$
			|D_1G_{K,N,m,\eta}\circ \hat{F}_{\beta, m}\circ \tilde{G}_{K,N,m,\eta}(x)| \leq 
			|DG_{K,N,m,\eta}|\cdot |D\hat{F}_{\beta, m}|\cdot |D_1\tilde{G}_{K,N,m,\eta}(x))| \leq C.
		$$ 
		
		For (almost) all $x$ in  
		$$
		E_{1,c}:= ( [-2m-3, 2m+3]\setminus [-2m-1, 2m+1])\times[-2m-1,2m+1]^2\times [-1,1]
		$$ 
		we have $D_1\tilde{G}_{K,N,m,\eta}(x) = e_1$ by \eqref{OhDear2}. Let us consider two options
		\eqn{ahaaa}
		$$
		\begin{aligned}
		x&\in ( [-2m-3, 2m+3]\setminus [-2m-1, 2m+1])\times\tilde{U}_{K,M+1}^2\times [-1,1]\text{ or }\\
		y&\in ( [-2m-3, 2m+3]\setminus [-2m-1, 2m+1])\times([-2m-1,2m+1]^2\setminus\tilde{U}_{K,M+1}^2)\times [-1,1].
		\end{aligned}
		$$
		Of course we assume that $N\geq M+1$ and therefore by \eqref{OhDear2} and \eqref{basic} 
		$$
		\begin{aligned}
		\tilde{G}_{K,N,m,\eta}&\big(( [-2m-3, 2m+3]\setminus [-2m-1, 2m+1])\times\tilde{U}_{K,M+1}^2\times [-1,1]\big)\\
		 &= ( [-2m-3, 2m+3]\setminus [-2m-1, 2m+1])\times{U}_{M+1}^2\times [-1,1]\\
		\tilde{G}_{K,N,m,\eta}&\big(( [-2m-3, 2m+3]\setminus [-2m-1, 2m+1])\times([-2m-1,2m+1]^2\setminus\tilde{U}_{K,M+1}^2)\times [-1,1]\big)\\
		 &= ( [-2m-3, 2m+3]\setminus [-2m-1, 2m+1])\times([-2m-1,2m+1]^2\setminus{U}_{M+1}^2)\times [-1,1].
		\end{aligned}
		$$
		From Lemma~\ref{ThatTimeOfTheMonth} we obtain $\hat{F}_{\beta,m}(\tilde{G}_{K,N,m,\eta}(y))\notin\T$ since $\tilde{G}_{K,N,m,\eta}(y)\in \G$.  
		Applying Proposition~\ref{ST0} for these $\hat{F}_{\beta,m}(\tilde{G}_{K,N,m,\eta}(y))$ we have $|D{G}_{K,N,m,\eta}|\leq C(M,\beta)$.
		We use the fact that $\hat{F}_{\beta,m}$ is Lipschitz and combine the above estimates to get
		$$
			|D_1G_{K,N,m,\eta}\circ \hat{F}_{\beta, m}\circ \tilde{G}_{K,N,m,\eta}(y)| \leq 
			|DG_{K,N,m,\eta}|\cdot |D\hat{F}_{\beta, m}|\cdot |D_1\tilde{G}_{K,N,m,\eta}(x))| \leq C. 
		$$
Recall that for $x$ as in \eqref{ahaaa} we still have $D_1\tilde{G}_{K,N,m,\eta}(x) = e_1$. By \eqref{NotGay1} we have $D_1\hat{F}_{\beta, m}( \tilde{G}_{K,N,m,\eta}(x)) = e_1$. 
The map $\hat{F}_{\beta, m}$ is bi-Lipschitz and satisfies \eqref{Oasis} on $\mathcal{K}_{B,m}$ part. 
Therefore calling $z=\tilde{G}_{K,N,m,\eta}(x)$ we have by Theorem~\ref{BloodSweatAndTears} (recall that $M\geq M_0$) 
		$$
		|\hat{F}_{\beta, m}(z) - (z_1,z_2,z_3,-z_4)| <C2^{-M\beta}<\tfrac{1}{4}
		$$
		where $C$ is the bi-Lipschitz constant of $\hat{F}_{\beta, m}$ and up to increase of $M$ (still a fixed finite constant) the above holds. Then $\hat{F}_{\beta, m}(z)$ is furthest from $\C_{B,m}$ in the $e_1$ coordinate and $(\hat{F}_{\beta, m}(z))^1 \in [-2m-3-\tfrac{1}{4}, -2m-1]\cup[2m+1, 2m+3+\tfrac{1}{4}]\cup U_{1}$. On the set where $(\hat{F}_{\beta, m}(z))^1 \in [-2m-3-\tfrac{1}{4}, -2m-1]\cup[2m+1, 2m+3+\tfrac{1}{4}]$ we have $|D_1G_{K,N,m,\eta}|<C$ by \eqref{StretchingBitAroundTop2}. On the set where $(\hat{F}_{\beta, m}(z))^1 \in U_{1}$ we use \eqref{ASeason0} and 
		\eqref{SpeciallyForStanda} for $k=1$ together with $|D_1\zeta_{K,1}(x_1)|\leq C$, $D_1 [H^{3,1}_K]_{l}=0$ and $|D_1t_K|\leq C$ to obtain $|D_1G_{K,N,m,\eta}|<C$. Therefore we calculate
		$$
		\begin{aligned}
		|D_1G_{K,N,m,\eta}\circ \hat{F}_{\beta, m}\circ \tilde{G}_{K,N,m,\eta}(x)| &= |D_1G_{K,N,m,\eta}\circ\hat{F}_{\beta, m}(\tilde{G}_{K,N,m,\eta}(x))| \\
		&= |D_1G_{K,N,m,\eta}(\hat{F}_{\beta, m}\circ\tilde{G}_{K,N,m,\eta}(x))|\\
		&\leq C.
		\end{aligned}
		$$
		
		On $E_2$ we have $|D_1\tilde{G}_{K,N,m,\eta}|<C$ and $\tilde{G}_{K,N,m,\eta}(E_2)=E_2$ by \eqref{DeathToGo1}. By Lemma~\ref{ThatTimeOfTheMonth} $\hat{F}_{\beta, m}(E_2) \subset R_{m,13} \setminus \T$ since $E_2\subset \G$. 		
		On the set $R_{m,13} \setminus \T$ we have $|DG_{K,N,m,\eta}|<C$ by Proposition~\ref{ST0}. Therefore
		$$
		|D_1G_{K,N,m,\eta}\circ \hat{F}_{\beta, m}\circ \tilde{G}_{K,N,m,\eta}|<C \text{ on } E_2.
		$$

		On the set $E_3$ we have $|D_1\tilde{G}_{K,N,m,\eta}|<C$ and $\tilde{G}_{K,N,m,\eta}(E_3)\subset E_3\cup E_4$ by \eqref{OhDear2}. Clearly $E_3\cup E_4\subset\mathcal{G}$ and hence $\hat{F}_{\beta, m}(E_3\cup E_4) \subset R_{m,13} \setminus \T$ by Lemma~\ref{ThatTimeOfTheMonth}. By Proposition~\ref{ST0} we now obtain $|D\tilde{G}_{K,N,m,\eta}(\hat{F}_{\beta, m}\circ \tilde{G}_{K,N,m,\eta}(x))|<C$ for $x\in E_3$. Therefore
		$$
		|D_1G_{K,N,m,\eta}\circ \hat{F}_{\beta, m}\circ \tilde{G}_{K,N,m,\eta}|<C \text{ on } E_3.
		$$
		
		On the set $E_4$ we estimate as follows. It is easy to check from the definition of $\tilde{G}_{K,N,m,\eta}$ \eqref{OhDear2} using Proposition~\ref{FrameToFrameDifferential} that because $2K\geq N$ we have
		$$
			|D_1\tilde{G}_{K,N,m,\eta}| \leq |D_1 [J^{3,i}_K]_{2N}| \leq C K^{\alpha+1}\text{ on }E_4.
		$$
	On $E_4$ we use Lemma~\ref{ThatTimeOfTheMonth} and Proposition~\ref{ST0} (again $\hat{F}_{\beta, m}(\tilde{G}_{K,N,m,\eta}(x))\in \G$) to get $|DG_{K,N,m,\eta}|\leq C$ and thus
		$$
		|D_1G_{K,N,m,\eta}\circ \hat{F}_{\beta, m}\circ \tilde{G}_{K,N,m,\eta}|<CK^{\alpha+1} \text{ on } E_4.
		$$
		Let us consider $x\in E_6$ now and  
		let us assume that for example (other permutations can be treated analogously) 
		$$
		x_3\in [-2m-3,2m+3]\setminus[-2m-1,2m+1]
		$$
		and
		$$
		 x_1,x_2\in [-2m-1,2m+1] \text{ and }x_4\in[-1,1].
		$$
		
 For almost all $x$ there exists exactly one cube $Q(n,1)$ closest to $x$ with $-m\leq n_1,n_2\leq m$ and $n=(2n_1,2n_2,2m,0)$. On the set $[2n_1-1,2n_1+1]\times[2n_2-1,2n_2+2]\times([-2m-3,2m+3]\setminus[-2m-1,2m+1])\times[-1,1]$ the map $\tilde{G}_{K,N,m,\eta}$ is defined as $[J^{3,3}_K]_{2N}(x) + n +x_3e_3$ (see \eqref{OhDear2}). Either there exists some $0\leq k <2N$ such that $(x_1,x_2,x_4)\in \tilde{U}_{K,k}^3\setminus \tilde{U}_{K,k+1}^3$ or $(x_1,x_2,x_4)\in \tilde{U}_{K,2N}^3$. In the case that $(x_1,x_2,x_4)\in \tilde{U}_{K,k}^3\setminus \tilde{U}_{K,k+1}^3$ we use (a rotated version of) \eqref{Standard1} to estimate
		\begin{equation}\label{Spade}
		|D_1\tilde{G}_{K,N,m,\eta}(x)| \leq |D_1 [J^n_K]_{2N}(x)| \leq C 2^{-k\beta}K^{\alpha+1}
		\end{equation}
		and in the second case, where $(x_1,x_2,x_4)\in \tilde{U}_{K,2N}^3$ we use \eqref{Standard3} to estimate
		\begin{equation}\label{PickAxe}
		|D_1\tilde{G}_{K,N,m,\eta}(x)| \leq |D_1 [J^n_K]_{2N}(x)| \leq C 2^{-2N\beta}.
		\end{equation}
		
		It holds, by \eqref{Oasis}, that $\hat{F}_{\beta, m}$ maps $\K_{B,m}\cap [-2m-3,2m+3]^3\times[-3,3]$ onto itself. The map $\hat{F}_{\beta, m}$ is bi-Lipschitz and so $\dist(\K_{B,m},\hat{F}_{\beta, m}(x)) \approx \dist(\K_{B,m},x)$. Especially, when we call
		$$
		O_k = \{(x_1,x_2,x_3,x_4):(x_1,x_2,x_4)\in {U}_{k}^3, x_3\in [-2m-3,2m+3]\setminus[-2m-1,2m+1]\}
		$$
		(and note that each $O_k$ is roughly speaking $\K_{B,m} + Q(0,2^{-k(\beta+1)})$) we use the same reasoning as used in Proposition~\ref{EmotionallyDistant} to get a constant $M$ such that for each $x\in O_k\setminus O_{k+1}$
		\begin{equation}\label{CarryOn}
		([\hat{F}_{\beta, m}(x)]^1,[\hat{F}_{\beta, m}(x)]^2,[\hat{F}_{\beta, m}(x)]^4) \in U_{k-M}^3 \setminus U_{k+M+1}^3.
		\end{equation}
		Similarly in the $e_2$ direction
		$$
		[\hat{F}_{\beta, m}(\tilde{G}_{K,N,m,\eta}(x))]^3 \in ([-2m-3,2m+3]\setminus[-2m-1,2m+1]) \cup (U_{0}\setminus U_{M+1} + n)
		$$
		Therefore, using \eqref{StretchingBitAroundTop2} if we are outside $Q(n,1)$ and \eqref{SpeciallyForStanda} if we are in Q(n,1) for the definition of $G_{K,N,m,\eta}$ and \eqref{Standard5} and \eqref{Standard6} to calculate, we get 
		\begin{equation}\label{DigForVictory}
		|D{G}_{K,N,m,\eta}(y)| \leq C \min \{2^{(k+M)\beta}, 2^{(N+3M)\beta}\}
		\end{equation}
		for $(y_1,y_2,y_4) \in U_{k-M}^3\setminus U_{k+M+1}^3$, $y_2 \in ([-2m-3,2m+3]\setminus[-2m-1,2m+1]) \cup (U_{0}\setminus U_{M+1} + n)$.

		We calculate $|D_1G_{K,N,m,\eta}\circ \hat{F}_{\beta, m}\circ \tilde{G}_{K,N,m,\eta}(x)|$ for $x\in O_k\setminus O_{k+1}$ and $0\leq k <2N$ by multiplying \eqref{Spade} by \eqref{DigForVictory} (using \eqref{CarryOn}), which gives
		$$
		\begin{aligned}
		|D_1G_{K,N,m,\eta}\circ \hat{F}_{\beta, m}\circ \tilde{G}_{K,N,m,\eta}|
		&\leq C2^{-k\beta}(K+k)^{\alpha+1}2^{(k+M)\beta}\\
		&\leq C(K+k)^{\alpha+1}\\
		& \leq C (K+2N)^{\alpha+1}\\
		&\leq CK^{\alpha+1}
		\end{aligned}
		$$
		because $2K\geq N$. For $x\in O_{2N}$ we multiply \eqref{PickAxe} by \eqref{DigForVictory} (using \eqref{CarryOn}) to get
		$$
		\begin{aligned}
		|D_1G_{K,N,m,\eta}\circ \hat{F}_{\beta, m}\circ \tilde{G}_{K,N,m,\eta}|
		& \leq C2^{-2N\beta}2^{(N+3M)\beta}\\ 
		&\leq C2^{-N\beta} .
		\end{aligned}
		$$
		The combination of these two estimates gives that
		$$
		|D_1G_{K,N,m,\eta}\circ \hat{F}_{\beta, m}\circ \tilde{G}_{K,N,m,\eta}| \leq CK^{\alpha+1}
		$$
		on $E_6$ by considering the various permutations of the coordinates.
		
		For a point in $E_5$ such that its nearest cube is $Q(n,1)$ we use exaclty the same estimates as before. Namely, because $\C_{A, K}\subset \tilde{U}_{K, 2N}^3$, we use precisely the estimates from above i.e. we multiply \eqref{PickAxe} by \eqref{DigForVictory} (using \eqref{CarryOn}) to get
		$$
		\begin{aligned}
		|D_1G_{K,N,m,\eta}\circ \hat{F}_{\beta, m}\circ \tilde{G}_{K,N,m,\eta}(x)|
		& \leq C2^{-2N\beta}2^{(N+3M)\beta}\\ 
		&\leq C2^{-N\beta} .
		\end{aligned}
		$$
		Therefore
		$$
		|D_1G_{K,N,m,\eta}\circ \hat{F}_{\beta, m}\circ \tilde{G}_{K,N,m,\eta}| \leq C
		$$
		on $E_5$ by considering the various permutations of the coordinates.

Estimates on $E_7\cup E_8\cup E_9\cup E_{10}\cup E_{11}\cup E_{12}$ are similar to estimates above. The only difference is that we map slab of thickness $\eta$ on slab of thickness $13$ by $\tilde{G}_{K,N,m,\eta}$ (see \eqref{OhDear1A}, \eqref{OhDear1B}) and 
then slab of thickness $13$ back to slab of thickness $\eta$ by $G_{K,N,m,\eta}$ (see \eqref{StretchingBitAroundTop}). 
The derivative $D_1\tilde{G}_{K,N,m,\eta}$ can be estimated by the same expression as before since the additional factor $\frac{1}{\eta}$ is influencing only $D_4 \tilde{G}_{K,N,m,\eta}$ (note that e.g. in \eqref{OhDear1B} we have  
$$
\bigl|D_1\tfrac{2}{\eta}(x_4-1-\tfrac{\eta}{2})(x_1,x_2,x_3,0)\bigr|\leq C\text{ as }x_4\in \bigl[1+\frac{\eta}{2},1+\eta\bigr]
$$ 
and so there is no additional $\frac{1}{\eta}$ there). The derivative $DG_{K,N,m,\eta}$ can be estimated by the same expression as the additional factor $\eta$ in some of the terms can only help us. 
		
		The estimate of $|D_1G_{K,N,m,\eta}\circ \hat{F}_{\beta, m}\circ \tilde{G}_{K,N,m,\eta}|$ on $E_7$ is calculated the same way as on $E_6$ and
		$$
		|D_1G_{K,N,m,\eta}\circ \hat{F}_{\beta, m}\circ \tilde{G}_{K,N,m,\eta}|<CK^{\alpha+1} \text{ on } E_7.
		$$
		Similarly $|D_1G_{K,N,m,\eta}\circ \hat{F}_{\beta, m}\circ \tilde{G}_{K,N,m,\eta}|$ is estimated on $E_8$ is calculated the same way as on $E_5$ and
		$$
		|D_1G_{K,N,m,\eta}\circ \hat{F}_{\beta, m}\circ \tilde{G}_{K,N,m,\eta}|<C \text{ on } E_8.
		$$
		The estimate on $E_9$  is the same as on $E_4$
		$$
		|D_1G_{K,N,m,\eta}\circ \hat{F}_{\beta, m}\circ \tilde{G}_{K,N,m,\eta}|<CK^{\alpha+1} \text{ on } E_9.
		$$
		The estimate on $E_{10}$ is the same as on $E_5$ 
		$$
		|D_1G_{K,N,m,\eta}\circ \hat{F}_{\beta, m}\circ \tilde{G}_{K,N,m,\eta}|<C \text{ on } E_{10}.
		$$
		
		Combining the above with \eqref{StupidMeasures} we easily calculate
		\begin{equation}\label{Dragon}
			\begin{aligned}
				\int_{E_{1}\cup E_2\cup E_3\cup E_5} |D_1 G_{K,N,m,\eta}\circ \hat{F}_{\beta, m}\circ \tilde{G}_{K,N,m,\eta}|^p &\leq C
				(m+1)^2,\\
				\int_{E_{4} \cup E_6} |D_1 G_{K,N,m,\eta}\circ \hat{F}_{\beta, m}\circ \tilde{G}_{K,N,m,\eta}|^p &\leq CK^{(\alpha+1)(p-1)+1} (m+1)^2,\\
				\int_{E_{7} \cup E_9 } |D_1 G_{K,N,m,\eta}\circ \hat{F}_{\beta, m}\circ \tilde{G}_{K,N,m,\eta}|^p &\leq C K^{(\alpha+1)(p-1)+1} \eta (m+1)^3,\\
				\int_{E_{8} \cup E_{10}  } |D_1 G_{K,N,m,\eta}\circ \hat{F}_{\beta, m}\circ \tilde{G}_{K,N,m,\eta}|^p &\leq C\eta (m+1)^3.\\
			\end{aligned}
		\end{equation}
		Adding the estimates in \eqref{Dragon} to point $iii)$ of Proposition~\ref{IntegralEstimate} we get precisely $v)$. By rotating the sets we get the same estimates for the $D_j$-derivative, $j=2,3$.
		
		\step{6}{Proof of $vi)$}{PTLLS6}		
		
		Now it is useful to slightly change the decomposition of $R_{m,\eta}\setminus Z_m$ from Step~\ref{PTLLS5}. We will not use set $E_1, E_2,\hdots, E_6$ instead we define
		$$
			\begin{aligned}
		\tilde{E}_{3} =&[-2m-5, 2m+5]^3\setminus [-2m-3,2m+3]^3 \times [-1,1] \cap \Big(\bigcup_{j=1}^3(\C_{A,K,m} + \er e_j))\Big)\\
		\tilde{E}_{4} =&[-2m-5, 2m+5]^3\setminus [-2m-3,2m+3]^3 \times [-1,1] \setminus \Big(\bigcup_{j=1}^3(\C_{A,K,m} + \er e_j))\Big)\\
		\tilde{E}_{5} =&[-2m-3, 2m+3]^3\setminus [-2m-1,2m+1]^3 \times [-1,1] \cap \Big(\bigcup_{j=1}^3(\C_{A,K,m} + \er e_j))\Big)\\
		\tilde{E}_{6} =&[-2m-3, 2m+3]^3\setminus [-2m-1,2m+1]^3 \times [-1,1] \setminus \Big(\bigcup_{j=1}^3(\C_{A,K,m} + \er e_j))\Big).\\
		\end{aligned}
		$$
		 We need to estimate $D_4\tilde{G}_{K,N,m,\eta}$. The estimates of measure from \eqref{StupidMeasures} still hold even when we replace $E_i$ with $\tilde{E}_i$ for $i=3,4,5,6$. 
		 By the rotational symmetry of the map, the calculations of $|D_4 G_{K,N,m,\eta}\circ \hat{F}_{\beta, m}\circ \tilde{G}_{K,N,m,\eta}|$ on $\tilde{E}_3 \cup \tilde{E}_4 \cup\tilde{E}_5 \cup \tilde{E}_6$ are exactly the same as estimates of $|D_1 G_{K,N,m,\eta}\circ \hat{F}_{\beta, m}\circ \tilde{G}_{K,N,m,\eta}|$ in Step~\ref{PTLLS5} for ${E}_3 \cup {E}_4 \cup{E}_5 \cup {E}_6$. Precisely we have
		 \begin{equation}\label{tE3456}
		 |D_4G_{K,N,m,\eta}\circ \hat{F}_{\beta, m}\circ \tilde{G}_{K,N,m,\eta}|<CK^{\alpha+1} \text{ on } \tilde{E}_3 \cup \tilde{E}_4 \cup\tilde{E}_5 \cup \tilde{E}_6.
		 \end{equation}
		 
		  It is easy, by \eqref{OhDear1A}, \eqref{OhDear1B} and \eqref{DeathToGo2}, to estimate that 
		  $$
		  |D_4\tilde{G}_{K,N,m,\eta}| \leq \frac{C}{\eta}\text{ on }E_{7} \cup  E_8\cup E_9 \cup E_{10} \cup  E_{11}\cup E_{12}.
		  $$
		  Especially on $E_8$ we have $D_4\tilde{G}_{K,N,m,\eta} = \tfrac{C}{\eta}e_4$. Again 
		  $$
		  \hat{F}_{\beta,m}(x_1,x_2,x_3,x_4) = (x_1,x_2,x_3,-x_4)\text{ for }x \in \tilde{G}_{K,N,m,\eta}(E_8).
		  $$ 
		  Easily from \eqref{StretchingBitAroundTop}  
		  we see that $|D_4{G}_{K,N,m,\eta}(\hat{F}_{\beta, m}\circ \tilde{G}_{K,N,m,\eta}(x))| < C$ on $x\in E_8$. Therefore
		  \begin{equation}\label{tE8}
		  	|D_4G_{K,N,m,\eta}\circ \hat{F}_{\beta, m}\circ \tilde{G}_{K,N,m,\eta}|<\tfrac{C}{\eta} \text{ on } E_8.
		  \end{equation}
		  
		  We calculate on $E_7$ in two parts. The first part is when
		  $$
		  	x\in \bigcup_{n_1,n_2,n_3 = -m}^m\big(\tilde{U}_{K,2}^3\times( [-1-\tfrac{2\eta}{26 - \eta}, -1]\cup[1,1+\tfrac{2\eta}{26 - \eta}])\big)+ (2n_1,2n_2,2n_3,0)
		  $$
		  From \eqref{OhDear1A} and \eqref{starrr} we obtain that $\tilde{G}_{K,N,m,\eta}(x)$ maps this set into 
		  $$
		  \bigcup_{n_1,n_2,n_3 = -m}^m\big(U_2^3\times( [-3, -1]\cup[1,3])\big)+ (2n_1,2n_2,2n_3,0)
		  $$ 		  
		  and that $D_4\tilde{G}_{K,N,m,\eta} = \tfrac{C}{\eta}e_4$ there. 
		  By \eqref{NotGay4} we have $D_4\hat{F}_{\beta,m}(\tilde{G}_{K,N,m,\eta}(x))  = -e_4$ and by \eqref{StretchingBitAroundTop} we have $|D_4{G}_{K,N,m,\eta}|<C$. Therefore 
		  $$
		  |D_4G_{K,N,m,\eta}\circ \hat{F}_{\beta, m}\circ \tilde{G}_{K,N,m,\eta}|<\tfrac{C}{\eta}
		  $$
		  on the first part. The second part is
		  $$
		  	\Big([-2m-5, 2m+5]^3\setminus \big[\bigcup_{n_1,n_2,n_3 = -m}^m\tilde{U}_{K,2}^3+(2n_1,2n_2,2n_3)\big]\Big)\times\big( [-1-\tfrac{2\eta}{26 - \eta}, -1]\cup[1,1+\tfrac{2\eta}{26 - \eta}]\big).
		  $$
		  From \eqref{OhDear1A} and \eqref{starrr} we obtain that $\tilde{G}_{K,N,m,\eta}(x)$ maps this set into 
		  $$
		  \big([-2m-5, 2m+5]^3\setminus \bigcup_{n_1,n_2,n_3 = -m}^m U_2^3+ (2n_1,2n_2,2n_3,0)\big)\times( [-3, -1]\cup[1,3]).
		  $$ 	
		  On this set we have that $|D\hat{F}_{\beta, m}|<C$ and that $\hat{F}_{\beta, m}(\tilde{G}_{K,N,m,\eta}(x)) \notin \T$ by Lemma~\ref{ThatTimeOfTheMonth}. Therefore $|DG_{K,N,m,\eta}( \hat{F}_{\beta, m}\circ \tilde{G}_{K,N,m,\eta}(x))| < C$ on this second part as well and 
		  \begin{equation}\label{tE7}
		  |D_4G_{K,N,m,\eta}\circ \hat{F}_{\beta, m}\circ \tilde{G}_{K,N,m,\eta}|<\tfrac{C}{\eta} \text{ on } E_7.
		  \end{equation}

		  It remains to consider $E_9 \cup E_{10}$. From \eqref{OhDear1A} and \eqref{starrrr} we see that $\tilde{G}_{K,N,m,\eta}(x)$ maps this set into 
		  $$
		  [-2m-5, 2m+5]^3\times ([-14, -3]\cup[3,14]). 
		  $$
		  From Lemma~\ref{ThatTimeOfTheMonth} we see that $\hat{F}_{\beta, m}(\tilde{G}_{K,N,m,\eta}(x)) \notin \T$ and so $|DG_{K,N,m,\eta}( \hat{F}_{\beta, m}\circ \tilde{G}_{K,N,m,\eta}(x))| <C$. Therefore
		  \begin{equation}\label{tE9101112}
		  |D_4G_{K,N,m,\eta}\circ \hat{F}_{\beta, m}\circ \tilde{G}_{K,N,m,\eta}|<\tfrac{C}{\eta} \text{ on } E_9\cup E_{10}.
		  \end{equation}
		 Then
		\begin{equation}\label{Dungeon}
			\begin{aligned}
				\int_{\tilde{E}_3 \cup\tilde{E}_5} |D_4 G_{K,N,m,\eta}\circ \hat{F}_{\beta, m}\circ \tilde{G}_{K,N,m,\eta}|^p &\leq C(m+1)^2,\\
				\int_{\tilde{E}_4 \cup \tilde{E}_6} |D_4 G_{K,N,m,\eta}\circ \hat{F}_{\beta, m}\circ \tilde{G}_{K,N,m,\eta}|^p &\leq C K^{(\alpha+1)(p-1)+1} (m+1)^2,\\
				\int_{E_{7}\cup E_8 \cup E_9\cup E_{10}} |D_4 G_{K,N,m,\eta}\circ \hat{F}_{\beta, m}\circ \tilde{G}_{K,N,m,\eta}|^p &\leq C  \eta^{1-p} (m+1)^3.
			\end{aligned}
		\end{equation}
		Adding the estimates in \eqref{Dungeon} to point $iii)$ of Proposition~\ref{IntegralEstimate} we get precisely $vi)$.	
	\end{proof}

	\begin{proof}[Proof of Theorem~\ref{TheBigLebowski}]
		We choose $\alpha = \tfrac{2+p}{3-2p}$ and note that for $p\in[1,\frac{3}{2}]$ we easily have $\alpha>\frac{4}{2-p}$. Further we set $N= 2K$ and $\eta = K^{-\alpha-1}$. We assume that $K\geq K_1$ so that 
		$N\geq 3M^*+3M+6$ and $N\geq 3M+K$. 
		We calculate $J = J(K,m) \approx m$  the largest natural number strictly smaller than $\tfrac{2m+5}{2+ 2K^{-\alpha}}$ and call $d_J = \tfrac{2+2K^{-\alpha}}{2m+5}$. We divide the cube $Q(0,1)$ into $J$ disjoint (up to common boundaries) `plates' call them
		$$
		P_i = [-1,1]^3 \times [-1+(i-1)d_J, -1+i\cdot d_J]
		$$
		for $i=1,\dots, J$. We define $g_{K,m}$ on each $P_i$ we apply a rescaled and translated version of the mapping from Proposition~\ref{TheLittleLebowski}, i.e. for each $x\in P_i$ we put
		$$
		g_{K,m}(x) = \frac{1}{2m+5} f_{K,3M+K,m,K^{-\alpha}}\Big(\big(2m+5\big)\big(x-(1+(i-\tfrac{1}{2})d_J)e_4\big) \Big) + (1+(i-\tfrac{1}{2})d_J)e_4
		$$
		and finally, for $x\in Q(0,1)\setminus (\bigcup_{i=1}^{J(K,m)}P_i)$ (a set of measure less than $3/m$) we define $g_{K,m}(x) = x$.
		
		Points $i)$, $ii)$ and $iii)$ follow from Proposition~\ref{TheLittleLebowski} (points $i)$, $ii)$, $iii)$ and $iv)$) given we call the set
		$$
			E_{K,m} = \bigcup_{i=1}^{J} \Big(\frac{1}{2m+5}\C_{A, K,m} -e_4 +i\cdot d_Je_4 - \tfrac{d_J}{2}e_4\Big).
		$$
		
		Now notice by \eqref{rozdil} we have 
		$$
		\frac{\mathcal{L}^4(\C_{A,K}) } { \mathcal{L}^4(Q(0,1))} \geq 1-CK^{-\alpha-1}\text{ and therefore }
		\frac{\mathcal{L}^4(\C_{A, K,m})}  {\mathcal{L}^4(Z_m)} \geq 1-CK^{-\alpha-1}.
		$$
		Further, it is an obvious geometrical fact that $ \mathcal{L}^4(Z_m)/  \mathcal{L}^4(R_{m,K^{-\alpha}}) \to \frac{1}{1+K^{-\alpha}}$ as $m \to \infty$. 
	Hence it is easy to see that given $\epsilon>0$ we  can choose $K$ big enough and $m$ big enough so that 
		$$
			\mathcal{L}^4\Big(\bigcup_{i=1}^{J}P_i\setminus E_{K,m} \Big) < \frac{\epsilon}{2}
			\ \text{ and } \
			\mathcal{L}^4\Big(Q(0,1) \setminus \bigcup_{i=1}^{J}P_i\Big) < \frac{\epsilon}{2}
		$$
		and so
		$$
			\mathcal{L}^4\Big(Q(0,1) \setminus E_{K,m}\Big) < \epsilon
		$$
		 for all $m\geq m_0$ and $K > K_0$. Therefore the set $E_{K,m}$ satisfy the condition $iv)$.
		
		Clearly $Dg_{K,m}(x)= Df_{K,K+3M, m,K^{-\alpha}}(\Phi_{i,m}(x))$ where $\Phi_{i,m}$ is the affine map of $P_i$ onto $R_{m,\eta}$ and hence by a simple linear change of variables we have
		$$
		\int_{P_i}|D_j g_{K,m}(y)|^p(2m+5)^4\; dy=\int_{R_{m,K^{-\alpha}}} |D_jf_{K,3M+K,m,m^{-1}}(x)|^p \; dx
		$$
and 		also
		$$
		\int_{P_i\setminus (E_{K,m}\cap P_i)}|D_j g_{K,m}(y)|^p(2m+5)^4\; dy=\int_{R_{m,K^{-\alpha}}\setminus \C_{A, K,m}} |D_jf_{K,3M+K,m,m^{-1}}(x)|^p \; dx
		$$
		for all $j=1,2,3,4$. We sum this over all $1\leq i\leq J \approx m$, and use Proposition~\ref{TheLittleLebowski} points $v)$ and $vi)$ 
		$$
			\begin{aligned}
				\int_{Q(0,1) \setminus E_{K,m}} |D_j g_{K,m}|^p
				&\leq C(p,\alpha, \beta) K^{(\alpha+1)(p-1)+1}(m+1)^2 \frac{J}{(2m+5)^4} \\
				& \quad + C(p,\alpha, \beta) K^{(\alpha+1)(p-1)+1}(m+1)^3 K^{-\alpha} \frac{J}{(2m+5)^4} \\
				&\quad + C(p,\alpha, \beta) K^{-(\alpha+1)(2-p)+2} (m+1)^3 \frac{J}{(2m+5)^4}\\
				&\leq C(p,\alpha, \beta)\big[K^{(\alpha+1)(p-1)+1} m^{-1}+K^{-(\alpha+1)(2-p)+2}\big]
			\end{aligned}
		$$
		for $j=1,2,3$. Similarly
		$$
			\begin{aligned}
				\int_{Q(0,1) \setminus E_{K,m}} |D_4 g_{K,m}|^p
				& \leq C(p,\alpha, \beta)K^{(\alpha + 1)(p-1)+1}(m+1)^2 \frac{J}{(2m+5)^4} \\
				&\quad + C(p,\alpha, \beta)(m+1)^3 K^{-(1-p)\alpha} \frac{J}{(2m+5)^4}\\
				&\quad + C(p,\alpha, \beta)K^{-(\alpha+1)(2-p)+2}(m+1)^3 \frac{J}{(2m+5)^4}\\
				& \leq C(p,\alpha, \beta) \big[K^{(\alpha+1)(p-1)+1} m^{-1} 
				+ K^{\alpha(p-1)}+  K^{-(\alpha+1)(2-p)+2}\big]\\
				& \leq C(p,\alpha, \beta) \big[K^{(\alpha+1)(p-1)+1} m^{-1} 
				+ K^{\alpha(p-1)}\big] 
			\end{aligned}
		$$
		where we have used  $K^{-(\alpha+1)(2-p)+2}\leq 1\leq K^{\alpha(p-1)}$ in the last step (recall that $\alpha> \frac{4}{2-p}$). 
		Therefore for $j=1,2,3$ we have
		\begin{equation}\label{PureEvil}
			\begin{aligned}
				\int_{Q(0,1) \setminus E_{K,m}} |D_j g_{K,m}|^p \cdot &\int_{Q(0,1) \setminus E_{K,m}} |D_4 g_{K,m}|^p \\
				 &\leq C K^{2(\alpha+1)(p-1)+2} m^{-2} + CK^{(\alpha+1)(2p-2)-p+2} m^{-1}\\
				 &\qquad+ CK^{-(\alpha+1)(3-2p)+3} m^{-1}  + C K^{-\alpha(3-2p)+p}.\\
			\end{aligned}
		\end{equation}
		First we show that the last term is bounded. By the choice of $\alpha = \tfrac{2+p}{3-2p}$ we have $K^{-\alpha(3-2p)+p} = K^{-2}$, therefore when $K\geq \frac{5}{\sqrt{C}}$ we have $CK^{-\alpha(3-2p)+p} \leq \tfrac{1}{25}$. Now we fix $K > \max\{K_0, K_1, \frac{5}{\sqrt{C}}\}$. Clearly 
		$$
			m^{-1}\bigl(C K^{2(\alpha+1)(p-1)+2} m^{-1} + CK^{(\alpha+1)(2p-2)-p+2}+ CK^{-(\alpha+1)(3-2p)+3}\bigr) \to 0 \text{ as } m\to \infty
		$$
		and therefore we can choose $m$ large enough so that the other terms are smaller than $\tfrac{1}{24}$. Thus we have
	\begin{equation}\label{DontQuoteMe}
	\int_{Q(0,1) \setminus E_{K,m}} |D_j g_{K,m}|^p \cdot \int_{Q(0,1) \setminus E_{K,m}} |D_4 g_{K,m}|^p \leq \frac{1}{12}
	\end{equation}
	which is $v)$. Having chosen $K$ and $m$ we put $f_1 = g_{K,m}$ and so Theorem~\ref{TheBigLebowski} is proven.
	\end{proof}

\section{Proof of Theorem~\ref{main}}

\subsection{Construction of $f_2$ by composing $f_1$ with itself}
	\begin{thm}\label{LouisArmstrong}
		Let $Q(0,1)$ be the cube in dimension $4$, let $\epsilon>0$ and $1\leq p < \frac 3 2$. There exists a closed set $\mathcal{E} \subset Q(0,1)$ and a map $f_2 \in W^{1,p}(Q(0,1), \er^4)$ such that,
		\begin{enumerate}
			\item[$i)$] $f_2(x) = x$ for $x\in \partial Q(0,1)$,
			\item[$ii)$] $f_2$ is locally bi-Lipschitz on $Q(0,1) \setminus \mathcal{E}$,
			\item[$iii)$] $J_{f_2}<0$ almost everywhere on $\mathcal{E}$,
			\item[$iv)$] 
			\begin{equation}\label{Merlin}
				\int_{Q(0,1)\setminus \mathcal{E}}|Df_2|^p< \tfrac{1}{3}\mathcal{L}^4(Q(0,1)),
			\end{equation}
			\item[$v)$] $\mathcal{L}(Q(0,1)\setminus\mathcal{E})< \epsilon$. 
		\end{enumerate}
	\end{thm}
	
	\begin{proof}
		The starting point for our construction is the mapping $f_1$ from Theorem~\ref{TheBigLebowski}, where $p$ is the $p$ of our claim, $\epsilon$ is some small fixed positive number. By point $ii)$ of Theorem~\ref{TheBigLebowski} we have that $f_1$ is locally bi-Lipschitz on $Q(0,1)\setminus E$. In the following we have the pair $a$ in the preimage and $c$ in the image, so that $c = f_1(a)$. Assume that $a = f_1^{-1}(c) \in Q(0,1)\setminus E$ is a point of differentiability of $f_1$ and a point of approximate continuity of $Df_1$. Then we define
		\begin{equation}\label{defS}
			S(c) = \frac{Df_1(a)(e_4)}{|Df_1(a)(e_4)|},
		\end{equation}
		i.e. $S(c)$ is the unit vector in the direction of the image of $e_4$ in the differential of $f_1$ at $a$ (note that $Df_1(a)(e_4)\neq$ by the local bi-Lipschitz quality of $f_1$). Call $O_c$ a sense-preserving unitary map so that $O_c(S(c)) = e_1$. For each such $c$ we define the cube $\mathcal{Q}_{c,r} = O_c^{-1}(Q(0,r))+c$. Then we define
		\eqn{defgcr}
		$$
			g_{c,r}(x) = O_c^{-1}\Big(rf_1\Big(O_c\Big(\frac{x-c}{r}\Big)\Big)\Big) + c
		$$
		 for $x\in \mathcal{Q}_{c,r}$ for each of the $c$ where we have defined $S(c)$. Note especially that $g_{c,r}(x) = x$ on $\partial\mathcal{Q}_{c,r}$. We intend to apply Lemma~\ref{Patchwork}, which is formulated for $f_{c,r} = rf(\tfrac{x-c}{r})+c$ but our $g_{c,r}$ is composed with a unitary map from inside and outside. This does not make any difference to the estimates however.
		
		 For each $c \in Q(0,1)$ call $a = f_1^{-1}(c)$. We denote $A_{c,r}(x) = c + Df_1(a)(x-a)$. Now $f_{1}$ is locally Lipschitz on $Q(0,1) \setminus E$ and so for any $c\in Q(0,1) \setminus E$ there is $L_c>0$ so that for $0<r<\tfrac{1}{2}\dist(c, E)$ we have
		$$
		|Df_1(x)|^p \leq L^p_c\text{ for almost all }x\in[f^{-1}_1(\mathcal{Q}_{c,r})\cup A_{c,r}^{-1}(\mathcal{Q}_{c,r})].
		$$
		This, together with the approximate continuity of $Df_1$, the absolute continuity of its integral and the fact that $a$ is a point of differentiability of $f_1$ means that if $c$ is fixed then there exists an $r_c$ such that when $0<r<r_c$ it holds that
		 \begin{equation}\label{Babicky}
		 |D_jf_1(a)|^p \mathcal{L}^4 \big(A_{c,r}^{-1}(\mathcal{Q}_{c,r})\big)  \leq \frac{3}{2}	\int_{f_1^{-1}(\mathcal{Q}_{c,r})}|D_jf_1|^p\text{ for }j=1,2,3,4.
		 \end{equation}
		We may moreover assume that $0<r_c$ is small enough so that the conclusion of Lemma \ref{Patchwork} holds for $g_{c,r}$ on $\mathcal{Q}_{c,r}$ with $F=f_1$.
		Therefore by Lemma~\ref{StickThatUpYourPipeAndSmokeItTillTheCowsComeHome} we have a covering of $Q(0,1) \setminus (f_1(E)\cup \mathcal N)$ (where $E$ is the closed set from Theorem~\ref{TheBigLebowski} and $\mathcal N$ is the closed null set of Lemma~\ref{StickThatUpYourPipeAndSmokeItTillTheCowsComeHome}) by rotated cubes $\{\mathcal{Q}_{c_i,r_i}\}$ such that $r_i<r_{c_i}$, for $i \in \en$. We can infer from our covering and from \eqref{Babicky} that
		\begin{equation}\label{DebilniKecy}
			\sum_{i=1}^{\infty}|D_jf_1(a_i)|^p \mathcal{L}^4 \big(A_{c_i,r_i}^{-1}(\mathcal{Q}_{c_i,r_i})\big) \leq \tfrac{3}{2} \int_{Q(0,1)}|D_jf_1|^p \text{ for }j=1,2,3,4.
		\end{equation}
		
		 Then we define $f_2$ as $f_1$ composed with a scaled and rotated copy of $f_1$ on $f_1^{-1}(\mathcal{Q}_{c_i, r_i})$, i.e.
		\begin{equation}\label{Helsinki}
			f_2(x) = \begin{cases}
				f_1(x) \qquad & x\in Q(0,1)\setminus f_1^{-1}(\bigcup_{i=1}^{\infty}\mathcal{Q}_{c_i, r_i})\\
				g_{c_i,r_i}\circ f_1(x), \qquad &x\in f_1^{-1}(\mathcal{Q}_{c_i, r_i}).\\
			\end{cases}
		\end{equation}
		Let us note that 
		$Q(0,1)\setminus f_1^{-1}(\bigcup_{i=1}^{\infty}\mathcal{Q}_{c_i, r_i})=E\cup 
		f_1^{-1}(\mathcal{N})$ and hence it is $E$ up to a zero measure set. 
		This immediately implies that $f_2(x) = x$ on $\partial Q(0,1)$ (point $i)$). The fact that $f_1$ is a homeomorphism implies that $g_{c,r}$ is a homeomorphism on 
		$\mathcal{Q}_{c_i,r_i}$ and thus we see that $f_2$ is a homeomorphism. We call
		$$
			E_i = c_i + O_c^{-1}(r_iE)\text{ and }
	\mathcal{E} = E\cup f_1^{-1}\big(\bigcup_{i=1}^{\infty} E_i\big) \cup N,
	$$
	where $N = Q(0,1) \setminus \bigcup_i \mathcal{Q}_{c_i,r_i}^{\circ}$ is closed and has zero measure. Since $\mathcal{Q}_{c_i,r_i}$ have pairwise disjoint interiors and each $E_i$ is closed we have that the complement of $\mathcal{E}$ is the union of open sets and so $\mathcal{E}$ is closed in $Q(0,1)$ (point $ii)$ first claim). Since $f_1$ is locally bi-Lipschitz on $Q(0,1)\setminus E$, $g_{c_i,r_i}$ is locally bi-Lipschitz on $Q(0,1) \setminus \big(c_i + O_c^{-1}(r_iE)\big)$ and so $f_2$ is locally bi-Lipschitz on $Q(0,1) \setminus \mathcal{E}$ (point $ii)$ second claim). Further, the fact that $J_{f_1}<0$ on $E$ and $J_{f_1}(a_i)>0$ easily gives that $J_{g_{c_i,r_i}\circ f_1}<0$ on $f_1^{-1}\big(c_i + O_c^{-1}(r_iE)\big)$ and so we get $J_{f_2} < 0$ almost everywhere on $\mathcal{E}$ and $J_{f_2}>0$ almost everywhere on $Q(0,1)\setminus \mathcal{E}$ (point $iii)$). It remains to prove point $iv)$.
		
		The main idea is that the derivative of the composition is
		$\bigl|Dg_{c_i,r_i}(f_1(x))Df_1(x)\bigr|$ and the derivative of $f_1$ is big only in the $x_4$ direction (see Theorem \ref{TheBigLebowski}) and small in $x_1,x_2, x_3$-directions. We have rotated our $g_{c_i,r_i}$ so that big derivative $D_4 f_1$ is multiplied by derivatives of $g_{c_i,r_i}$ only in the directions where it is small (i.e. $D_1$ of the rotated and scaled $f_1$). Thus the derivative of the composition is small in average. More precisely we use derivative of the composition, \eqref{defgcr}, $O_{c}S(c) = e_1$ and the fact that $O_c^{-1}$ is unitary to obtain
\eqn{idiot}		
$$
\begin{aligned}
|Dg_{c,r}(y)S(c)|
&=\bigl|O_c^{-1}rDf_1\bigl(O_c(\frac{x-c}{r})\bigr)O_c\frac{1}{r}S(c)\bigr|\\
&=\bigl|O_c^{-1}Df_1\bigl(O_c(\frac{x-c}{r})\bigr)e_1\bigr|
=|D_1f_1\bigl(O_c(\frac{x-c}{r})\bigr)|.\\
\end{aligned}
$$
		
		By $Z_i$ we denote the linearized preimage of $\mathcal{Q}_{c_i,r_i}\setminus E_i$, i.e.
		$$
			Z_i = {a_i + [Df_1(a_i)]^{-1}[(\mathcal{Q}_{c_i,r_i}\setminus E_i) - c_i]}.
		$$
		We calculate using \eqref{Helsinki}, \eqref{WeAllLive} of Lemma~\ref{Patchwork} with $\rho = \frac {1} {64}$, a linear change of variables, \eqref{defS} and \eqref{idiot}
		$$
			\begin{aligned}
				&\int_{f_1^{-1}(\mathcal{Q}_{c_i,r_i} \setminus E_i)}|D_4f_2|^p
				= \int_{f_1^{-1}(\mathcal{Q}_{c_i,r_i}\setminus E_i)}\bigl|Dg_{c_i,r_i}(f_1(x))D_4f_1(x)\bigr|^p\; dx\\
				&\quad \leq \int_{Z_i} |Dg_{c_i,r_i}(c_i + Df_1(a_i)(x-a_i))D_4f_1(a_i)|^p\; dx
				+\frac {1} {64} r_i^4\\
				& \quad  \leq  |D_4f_1(a_i)|^p \int_{\mathcal{Q}_{c_i,r_i}\setminus E_i} \frac{|Dg_{c_i,r_i}(y)S(c_i)|^p}{\det Df_1(a_i)}\; dy+\frac {1} {64} r_i^4\\
				& \quad  = |D_4f_1(a_i)|^p \frac{\mathcal{L}^4\big([Df_1(a_i)]^{-1}\mathcal{Q}_{c_i,r_i}\big)}{\mathcal{L}^4\big(\mathcal{Q}_{c_i,r_i}\big)}  \int_{\mathcal{Q}_{c_i,r_i}\setminus E_i}|Dg_{c_i,r_i}(y)S(c_i)|^p\; dy+\frac {1} {64} r_i^4\\
				& \quad  =  |D_4f_1(a_i)|^p \frac{\mathcal{L}^4\big([Df_1(a_i)]^{-1}\mathcal{Q}_{c_i,r_i}\big)}{\mathcal{L}^4\big(Q(0,1)\big)}  \int_{Q({0,1}) \setminus E}|D_1f_1(x)|^p\; dx+\frac {1} {64} r_i^4.\\
			\end{aligned}
		$$
		Using \eqref{DebilniKecy} we may sum over $\mathcal{Q}_{c_i,r_i}$; we  recall that $f$ is locally Lipschitz and that we use Theorem \ref{TheBigLebowski} :
		$$
			\begin{aligned}
				\int_{Q(0,1) \setminus \mathcal{E}}|D_4f_2|^p &
				\leq \sum_{i=1}^{\infty} \Bigl[  |D_4f_1(a_i)|^p \mathcal{L}^4\big([Df_1(a_i)]^{-1}\mathcal{Q}_{c_i,r_i}\big)\int_{Q({0,1}) \setminus E}|D_1f_1(x)|^p +\frac {1} {64} r_i^4\Bigr]\\
				&  \leq \frac{3}{2}  \int_{Q(0,1)\setminus E}|D_4f_1|^p\int_{Q({0,1}) \setminus E}|D_1f_1(x)|^p+\frac {1} {64} \mathcal{L}^4(Q(0,1)) \\
				& \leq \frac{1}{8}+\frac{1}{4}.
			\end{aligned}
		$$
		We estimate the other derivatives $j=1,2,3$ in similar fashion to the above (recall that $f_1^{-1}\big(\bigcup_i \mathcal{Q}_{c_i, r_i} \cup N\big)= Q(0,1) \setminus E$ and that $f_1$ is bi-Lipschitz on the set $f_1^{-1}(N)$)
		$$
			\begin{aligned}
				\int_{Q(0,1) \setminus \mathcal{E}}|D_jf_2|^p &\leq \sum_{i=1}^{\infty} \int_{f_1^{-1}(\mathcal{Q}_{c_i,r_i} \setminus E_i)} |Dg_{c_i,r_i}(f_1(x))D_jf_1(x)|^p\; dx\\
				&  \leq   \sum_{i=1}^{\infty} \int_{Z_i} |Dg_{c_i,r_i}(c_i + Df_1(a_i)(x-a_i))D_jf_1(a_i)|^p\; dx+\frac {1} {64} r_i^4\\
				&   \leq  \sum_{i=1}^{\infty}  |D_jf_1(a_i)|^p \int_{\mathcal{Q}_{c_i,r_i}\setminus E_i} \frac{|Dg_{c_i,r_i}(y)|^p}{\det Df_1(a_i)}\; dy+\frac {1} {64} r_i^4\\
				&  \leq   \sum_{i=1}^{\infty} |D_jf_1(a_i)|^p \frac{\mathcal{L}^4\big([Df_1(a_i)]^{-1}\mathcal{Q}_{c_i,r_i}\big)}{\mathcal{L}^4\big(Q(0,1)\big) }  \int_{Q(0,1)\setminus E}|Df_1|^p+\frac {1} {64} r_i^4\\
				&  \leq \frac{3}{2}  \int_{Q(0,1)\setminus E}|D_jf_1|^p \int_{Q(0,1)\setminus E}|Df_1|^p+\frac {1} {64} \mathcal{L}^4(Q(0,1))\\
				&  \leq \frac{1}{8} +\frac {1} {4}.
			\end{aligned}
		$$
		Now 
		$$
			\begin{aligned}
				\int_{Q(0,1) \setminus \mathcal{E}}|Df_2|^p &\leq \sum_{j=1}^4 \int_{Q(0,1)\setminus E}|D_jf_2|^p\\
				&\leq  \sum_{j=1}^4 (\frac{1}{8}+\frac 1 4)\\
				&\leq \frac{1}{3}\mathcal{L}^4(Q(0,1)),
			\end{aligned}
		$$
		thus proving point $iv)$. It is obvious from Theorem~\ref{TheBigLebowski} that
		\begin{equation}\label{Size}
			\mathcal{L}^4 (Q(0,1)\setminus \mathcal{E})  <\epsilon,
		\end{equation}
		which gives $v)$.	
		
	\end{proof}
\subsection{Proof of Theorem~\ref{main}}
	By induction we construct a sequence of maps $\{f_m\}_{m=2}^{\infty}$, that converge in $W^{1,p}$ to some $f$, which satisfies our claim. We refer to the following properties as the induction hypothesis; for each $f_m$ we have a closed set $\mathcal{E}_m$ such that $J_{f_m}<0$ almost everywhere on $\mathcal{E}_m$ and $f_m$ is locally bi-Lipschitz on the complement of $\mathcal{E}_m$. 
	The start point $f_2$ is the map from Theorem~\ref{LouisArmstrong} and we denote $\mathcal{E}_2 = \mathcal{E}$ the set from Theorem~\ref{LouisArmstrong}. The fact that $f_2$ satisfies the induction hypothesis is included in the claim of Theorem~\ref{LouisArmstrong}.
	
	Assume then that we have constructed $f_{m'}$, $2 \leq m' \leq m-1$ so that they satisfy the induction hypothesis and we now want to construct $f_m$. We continue to construct a system of cubes $\{Q_i = Q(c_i,r_i) ; i \in \en\}$ with pairwise disjoint interiors, which covers $Q(0,1) \setminus \mathcal{E}_{m-1}$ up to a set $N_m$ of measure $0$. We use the notation that for a given $c\in Q(0,1)$ we have $a = f^{-1}(c)$, the set $P_{c,r}= \bigl\{a + [Df_{m-1}(a)]^{-1}Q(0,r) \bigr\}$, the linearization
	$$
	A_{m, c,r}(x) = c + Df_{m-1}(a)(x-a)\text{ for }x\in P_{c,r}\text{ and the function }
	f_{c,r} = rf_2(\tfrac{x-c}{r})+c
	$$
	for $x\in Q(c,r)$. By the induction hypothesis we have that $f_{m-1}$ is locally bi-Lipschitz and so for every $c = f_2(a)$ there exists an $R_c, L_c > 0$ such that $f_{m-1}$ is $L_c$-Lipschitz on $B(a,R_c)$. Thus by the approximative continuity of the function $|Df_{m-1}(a)|^p$ at almost every $a$, we have, for almost every $c$ with $a=f^{-1}(c) \in Q(0,1)\setminus \mathcal{E}_{m-1}$ there is $r_c$, so that for given $\rho>0$ and every $0<r<r_c$
	$$
	|Df_{m-1}(a)|^p \mathcal{L}^4 \big(A_{m,c,r}^{-1}(\mathcal{Q}_{c,r})\big) \leq \frac{6}{5} \int_{f_{m-1}^{-1}(\mathcal{Q}_{c,r})}|Df_{m-1}|^p	.
	$$
	Moreover, we can assume that this $r_c$ is so small that $f_{m-1}^{-1}(Q(c,r)) \subset Q(0,1) \setminus \mathcal{E}_{m-1}$ and that we can apply Lemma~\ref{Patchwork} on $f_{m-1}$ and especially by \eqref{TheFrench} we see that
	\begin{equation}\label{aha}
	\begin{aligned}
	\int_{f_{m-1}^{-1}(Q(c,r))} &|Df_{c,r}(f_{m-1}(x))Df_{m-1}(x)|^p dx \\
	&\leq\int_{A_{m,c,r}^{-1}(Q(c,r))} |Df_{c,r}(A_{m,c,r}(x))Df_{m-1}(a)|^p dx+\rho r^4	. \\
	\end{aligned}
	\end{equation}
where we fix and choose $\rho >0$.
	Therefore we can apply Corollary~\ref{NoCakeForLosers} and get a system $\{Q_i(m) = Q(c_i,r_i)\}$ covering $Q(0,1) \setminus \mathcal{E}_{m-1}$ up to a closed null set $N_m$ such that on each cube $Q_i(m)$ we have \eqref{aha} with fixed $\rho = \frac{1}{100}\int_{Q(0,1)\setminus \mathcal{E}_{m-1}}|Df_{m-1}|^p $ and
	\begin{equation}\label{DalsiDebilniKecy}
			\sum_{i=1}^{\infty} \mathcal{L}^4({P_{c_i, r_i}}) |Df_{m-1}(a_i)|^p
			\leq \tfrac{6}{5}\int_{Q(0,1)\setminus \mathcal{E}_{m-1}} |Df_{m-1}|^p.
	\end{equation}
	
	We define $f_{c_i,r_i} = r_if_2(\tfrac{x-c_i}{r_i})+c_i$ and
	\begin{equation}\label{MegaDef}
		f_m(x) =\begin{cases}
			f_{c_i,r_i}(f_{m-1}(x)) \qquad &x\in f_{m-1}^{-1}(Q_i(m))\\
			f_{m-1}(x) \qquad & x\in \mathcal{E}_{m-1}.\\
		\end{cases}
	\end{equation}
	Thus we see that $f_{m}(x) = f_{m-1}(x)$ for $x\in \partial Q_i(m)$ and that $f_m$ is a homeomorphism. Call $E_i = r_i\mathcal{E}_2 + c_i$ and call $\mathcal{E}_m = \mathcal E_{m-1} \cup \bigcup_i f_m^{-1}(E_i) \cup N_m$. The set $\mathcal{E}_m$ is closed because ${Q}_{i}(m)$ have pairwise disjoint interiors and each $f_m^{-1}(E_i)$ is closed in $Q_i(m)$  and so $Q(0,1) \setminus\mathcal{E}_m$ is the union of open sets. The fact that $f_{m-1}$ is locally bi-Lipschitz outside $\mathcal{E}_{m-1}$ and $f_2$ is locally bi-Lipschitz outside $\mathcal{E}_2$ shows that $f_m$ is locally bi-Lipschitz outside $\mathcal{E}_m$. The fact that $J_{f_{m-1}}<0$ a.e. on $\mathcal{E}_{m-1}$ and $J_{f_{m-1}}>0$ a.e. outside $\mathcal{E}_{m-1}$ together with the fact that $J_{f_2}<0$ on $\mathcal{E}_2$, shows that $J_{f_m}<0$ a.e. on $\mathcal{E}_m$ and $J_{f_m}>0$ a.e. outside $\mathcal{E}_{m}$. Thus we immediately see our maps satisfy the induction hypothesis.
	
	By \eqref{Size}, the fact that $f_m$ is bi-Lipschitz on $Q_i$ and $J_{f_{m-1}}$ is close to being constant on $Q_i$ we have that $\mathcal{L}^4(Q(0,1) \setminus \mathcal{E}_m) < 2^m\epsilon^{m-1}$ and so we see that the limit mapping satisfies $J_f < 0$ almost everywhere, if we can at least say that it is in $W^{1,1}$. Since $r_i$ is chosen very small and $f_{c_i,r_i}$ has identity on the boundary, we may assume, without loss of generality, that $f_m$ converge uniformly. Therefore in the rest of our proof we show that $Df_m$ forms a Cauchy sequence in $L^p$.
	
	We calculate the size of $\int_{Q(0,1)\setminus \mathcal{E}_{m}}|Df_{m}|^p$ inductively. We take Theorem~\ref{LouisArmstrong}, point $iv)$ as our starting point and we intend to prove
	\begin{equation}\label{FabKa}
		\int_{Q(0,1)\setminus \mathcal{E}_{m}}|Df_{m}|^p \leq \tfrac{1}{2^{m-1}} \mathcal{L}^4(Q(0,1)).
	\end{equation}
	The estimate \eqref{FabKa} holds for $m=2$ by \eqref{Merlin}. Call
	$$
		Z_i(m) = \Big\{a_i + [Df_{m-1}(a_i)]^{-1}\big(Q(0,r_i) \setminus r_i\mathcal{E}_2\big) \Big\}.
	$$
	By using \eqref{MegaDef}, \eqref{aha} with $\rho=\frac{1}{100}\int_{Q(0,1)\setminus \mathcal{E}_{m-1}}|Df_{m-1}|^p$, the definition of $f_{c_i,r_i}$, a linear change of variables, \eqref{DalsiDebilniKecy} and \eqref{Merlin} we get for $m\geq 3$
	$$
		\begin{aligned}
			\int_{Q(0,1)\setminus \mathcal{E}_{m}}&|Df_{m}|^p
			\leq \sum_{i=1}^{\infty}\int_{Z_i}|Df_{c_i,r_i}(c_i + Df_m(a_i)(x-a_i))Df_{m-1}(a_i)|^p\; dx+\rho r_i^4\\
			&\leq  \sum_{i=1}^{\infty}|Df_{m-1}(a_i)|^p\int_{Q_i\setminus E_i}\frac{|Df_{c_i,r_i}(y)|^p}{\det Df_{m-1}(a_i)}\; dy+\rho r_i^4\\
			&\leq \ \sum_{i=1}^{\infty}|Df_{m-1}(a_i)|^p \mathcal{L}^4\big([Df_{m-1}(a_i)]^{-1}Q(0,r_i)\big)
			\frac{\int_{Q_i\setminus E_i}|Df_{c_i,r_i}(y)|^p \; dy}{\mathcal{L}^4(Q(0,r_i))}+\rho r_i^4\\
			& \leq \tfrac{6}{5}\int_{Q(0,1)\setminus \mathcal{E}_{m-1}}|Df_{m-1}(y)|^p \frac{\int_{Q(0,1)\setminus \mathcal{E}_2} |Df_2(y)|^p\; dy}{\mathcal{L}^4(Q(0,1))}+\rho\mathcal{L}^4(Q(0,1))\\
			& \leq \tfrac{1}{2}\int_{Q(0,1)\setminus \mathcal{E}_{m-1}}|Df_{m-1}|^p
		\end{aligned}
	$$
	and so our claim by induction. Using the fact that $Df_{m-1} = Df_{m}$ on $\mathcal{E}_{m-1}$ and $|Df_m| = |Df_{m-1}|$ on $\mathcal{E}_{m-1}\setminus \mathcal{E}_{m}$
	$$
	\begin{aligned}
	\int_{Q(0,1)}\bigl|&Df_{m-1}(x)-Df_{m}(x)\bigr|^p\; dx\\
	&\leq 2^p\int_{Q(0,1)\setminus \mathcal{E}_{m-1}}|Df_{m-1}|^p +2^p\int_{Q(0,1)\setminus \mathcal{E}_{m-1}}|Df_{m}|^p \\
	&\leq 2^p\int_{Q(0,1)\setminus \mathcal{E}_{m-1}}|Df_{m-1}|^p +2^p\int_{Q(0,1)\setminus \mathcal{E}_{m}}|Df_{m}|^p +2^p\int_{\mathcal{E}_{m-1}\setminus \mathcal{E}_{m}}|Df_{m}|^p\\
	&\leq 2^p \frac{1}{2^{m-2}}\mathcal{L}^4(Q(0,1))+2^p\frac{1}{2^{m-1}}\mathcal{L}^4(Q(0,1)) +2^p\int_{\mathcal{E}_{m-1}\setminus \mathcal{E}_{m}}|Df_{m-1}|^p\\
	&\leq 2^p \frac{1}{2^{m-2}}\mathcal{L}^4(Q(0,1))+2^p\frac{1}{2^{m-1}}\mathcal{L}^4(Q(0,1)) +2^p\int_{Q(0,1)\setminus \mathcal{E}_{m}}|Df_{m-1}|^p\\
	&\leq 2^p \frac{1}{2^{m-3}}\mathcal{L}^4(Q(0,1))+2^p\frac{1}{2^{m-1}}\mathcal{L}^4(Q(0,1))
	\end{aligned}
	$$
	which shows that $f_m$ is a Cauchy sequence in $W^{1,p}$.
	Calling $f$ the limit of the sequence we prove our claim. Concerning the Lusin properties of the function we have the following; the map $f_1$ is linear on each scaled copy of $\C_{A,K,m}$ and that $f_1$ is locally bi-Lipschitz outside of the union of these sets. This means that $f_2$ is locally bi-Lipschitz on the set $\mathcal{E}_2$. By the same argument $f_m$ is locally bi-Lipschitz on $\mathcal{E}_{m}$. Because $f = f_m$ on $\mathcal{E}_m$ we easily see that $f$ satisfies the $(N)$ and the $(N^{-1})$ conditions on $\mathcal{E}_m$. On the other hand the union of these sets has full measure in $Q$ and so $f$ satisfies the $(N)$ and the $(N^{-1})$ conditions on $Q$.
	\qed

\prt{Remark}
\begin{proclaim}
We would like to know if there is an example like in Theorem for every $1\leq p<2$. 

We can slightly improve on our construction. So far we have iterated $f_1$ twice to obtain $f_2$ so that big derivative of first iteration in $x_4$ direction meets the derivative of the next iteration in the $x_1$ direction, where it is small. It would be possible to iterate $f_1$ four times and each time rotate the next iteration cleverly not to meet the big derivatives (in the $x_4$ direction) of previous iterations. This should give us the result for any 
$1\leq p<\frac{7}{4}$. We have not pursued this direction as our computations are already quite technical.   

To briefly hint that we note that with our choice $\eta=K^{-\alpha}$ after four iterations we get in the key estimate analogous to \eqref{PureEvil}
$$
\Bigl(\int_{Q(0,1) \setminus E} |D_j g_{K,m}|^p\Bigr)^3 \cdot \int_{Q(0,1) \setminus E} |D g_{K,m}|^p \leq  
				\frac{C(K)}{m}
				+C K^{3[-(\alpha+1)(2-p)+2]}K^{\alpha(p-1)}.
$$
The last key term is $K^{\alpha(4p-7)+p}$ and for any $p<\frac{7}{4}$ we can choose $\alpha$ big enough so that the exponent is negative. Finally for $m$ large enough the first term on the right hand side is as small as we wish. 
\end{proclaim}

\vskip 5pt
\noindent
{\bf Acknowledgments.} The authors would like to thank Giuseppe Buttazzo for asking the question about the results of \cite{CHT} that initiated this research.


\end{document}